\documentclass[12pt,reqno]{amsart}
\usepackage{amsmath,amssymb,amsthm,bbm,mathtools,calc,verbatim,enumitem,tikz,url,mathrsfs,cite,fullpage,hyperref}
\usepackage{float}
\usepackage{subcaption}

\usepackage{comment}

\newtheorem{theorem}{Theorem}[section]
\newtheorem{lemma}[theorem]{Lemma}
\newtheorem{corollary}[theorem]{Corollary}

\newtheorem{obs}[theorem]{Observation}
\newtheorem{claim}[theorem]{Claim}
\newtheorem{fact}[theorem]{Fact}

\theoremstyle{definition}

\newtheorem*{RBalg}{The Book Algorithm}
\theoremstyle{remark}

\newenvironment{clmproof}[1]{\begin{proof}[Proof of Claim~\ref{#1}]\let\qednow\qedsymbol\renewcommand{\qedsymbol}{}}{\; \qednow \end{proof}}

\newcommand\N{\mathbb{N}}

\newcommand\cB{\mathcal{B}}

\newcommand\cZ{\mathcal{Z}}

\def\cS{\mathcal{S}}
\newcommand\Ex{\mathbb{E}}

\newcommand\eps{\varepsilon}

\renewcommand{\le}{\leqslant}
\renewcommand{\ge}{\geqslant}
\renewcommand{\to}{\rightarrow}

\def\eps{\varepsilon}

	\def\N{\mathbb{N}}
	
	\def\1{\mathbbm{1}}

	\def\cD{\mathcal{D}}
	\def\cR{\mathcal{R}}
	\def\cB{\mathcal{B}}

\begin{document}

\title{An exponential improvement for diagonal Ramsey}

\author{Marcelo Campos \and Simon Griffiths \and Robert Morris \and Julian Sahasrabudhe}

\address{IMPA, Estrada Dona Castorina 110, Jardim Bot\^anico,
Rio de Janeiro, 22460-320, Brasil}\email{marcelo.campos@impa.br}

\address{Departamento de Matem\'atica, PUC-Rio, Rua Marqu\^{e}s de S\~{a}o Vicente 225, G\'avea, 22451-900 Rio de Janeiro, Brasil}\email{simon@puc-rio.br}

\address{IMPA, Estrada Dona Castorina 110, Jardim Bot\^anico,
Rio de Janeiro, 22460-320, Brasil}\email{rob@impa.br}

\address{Department of Pure Mathematics and Mathematical Statistics, Wilberforce Road, Cambridge, CB3 0WA, UK}\email{jdrs2@cam.ac.uk}

\thanks{MC was supported during this research by CNPq, and by a FAPERJ Bolsa Nota 10; SG was partially supported by FAPERJ (Proc.~201.194/2022) and by CNPq (Proc.~307521/2019-2); RM was partially supported by FAPERJ (Proc.~E-26/200.977/2021) and by CNPq (Proc.~303681/2020-9)}

\begin{abstract}
The Ramsey number $R(k)$ is the minimum $n \in \N$ such that every red-blue colouring of the edges of the complete graph $K_n$ on $n$ vertices contains a monochromatic copy of $K_k$. We prove that 
$$R(k) \le (4 - \eps)^k$$ 
for some constant $\eps > 0$. This is the first exponential improvement over the upper bound of~Erd\H{o}s and Szekeres, proved in 1935. 
\end{abstract}

\maketitle

\section{Introduction}

The Ramsey number $R(k)$ is the minimum $n \in \N$ such that every red-blue colouring of the edges of the complete graph on $n$ vertices contains a monochromatic clique on $k$ vertices. Ramsey~\cite{R30} proved in 1930 that $R(k)$ is finite for every $k \in \N$, and a few years later Erd\H{o}s and Szekeres~\cite{ESz35} rediscovered Ramsey's theorem, and showed that $R(k) \le 4^k$. An exponential lower bound on $R(k)$ was obtained by Erd\H{o}s~\cite{E47} in 1947, whose beautiful non-constructive proof of the bound $R(k) \ge 2^{k/2}$ initiated the development of the probabilistic method (see~\cite{AS}). For further background on Ramsey theory, we refer the reader to the classic text~\cite{GRS}, or the excellent recent survey~\cite{CFS}. 

Over the 75 years since Erd\H{o}s' proof, the problem of improving either bound has attracted an enormous amount of attention. Despite this, however, progress has been extremely slow, and it was not until 1988 that the upper bound of Erd\H{o}s and Szekeres was improved by a polynomial factor, by Thomason~\cite{T88}. About 20 years later, an important breakthrough was made by Conlon~\cite{C09}, who improved the upper bound by a super-polynomial factor, in the process significantly extending the method of Thomason. More recently, Sah~\cite{S23} optimised Conlon's technique, obtaining the bound
$$R(k) \le 4^{k - c(\log k)^2}$$
for some constant $c > 0$. This represents a natural barrier for the approach of Thomason and Conlon, which exploits quasirandom properties of colourings of $E(K_n)$ with no monochromatic $K_k$ that hold when $n$ is close to the Erd\H{o}s--Szekeres bound. For the lower bound, even less progress has been made: the bound proved by Erd\H{o}s in 1947 has only been improved by a factor of $2$, by Spencer~\cite{S77} in 1977, using the Lovász Local Lemma.  

In this paper we will prove the following theorem, which gives an exponential improvement over the upper bound of Erd\H{o}s and Szekeres~\cite{ESz35}. 
 
\begin{theorem}\label{thm:diagonal}
There exists $\eps > 0$ such that 
$$R(k) \le (4 - \eps)^k$$ 
for all sufficiently large $k \in \N$. 
\end{theorem}

We will make no serious attempt here to optimise the value of $\eps$ given by our approach, and will focus instead on giving a relatively simple and transparent proof. However, let us mention here for the interested reader that we will give two different proofs of Theorem~\ref{thm:diagonal}, the first (which is a little simpler) with $\eps = 2^{-10}$, and the second with $\eps = 2^{-7}$. The method introduced in this paper has recently been optimised by Gupta, Ndiaye, Norin and Wei~\cite{GNNW}, who showed that
$$R(k) \le 3.8^k,$$
and also gave an elegant alternative presentation of the approach using induction. 

Our method can also be used to bound off-diagonal Ramsey numbers $R(k,\ell)$ for a wide range of values of $\ell$. We remind the reader that $R(k,\ell)$ is the minimum $n \in \N$ such that every red-blue colouring of $E(K_n)$ contains either a red copy of $K_k$ or a blue copy of $K_\ell$ (so, in particular, $R(k,k) = R(k)$). Erd\H{o}s and Szekeres~\cite{ESz35} proved that 
\begin{equation}\label{eq:ESz:bound}
R(k,\ell) \le {k + \ell \choose \ell}
\end{equation}
for all $k,\ell \in \N$, and this bound was improved by R\"odl (see~\cite{GR}), Thomason~\cite{T88}, Conlon~\cite{C09} and Sah~\cite{S23}, the best known bound when $\ell = \Theta(k)$ being of the form\footnote{The method of~\cite{T88,C09,S23} begins to break down when $\ell = o(k)$, and the only known improvement over~\eqref{eq:ESz:bound} that holds for all $k$ and $\ell$ is by a poly-logarithmic factor, proved in unpublished work of R\"odl. The proof of a weaker bound, improving~\eqref{eq:ESz:bound} by a factor of $\log\log k$, is given in the survey~\cite{GR}.} 
$$R(k,\ell) \le e^{- c(\log k)^2} {k + \ell \choose \ell}$$
for some constant $c > 0$. We will prove the following theorem, which improves this bound by an exponential factor. 

\begin{theorem}\label{thm:off:diagonal}
There exists $\delta > 0$ such that 
$$R(k,\ell) \le e^{-\delta \ell + o(k)} {k + \ell \choose \ell}$$
for all $k,\ell \in \N$ with $\ell \le k$. 
\end{theorem}

As was shown in~\cite{GNNW}, our approach can also be adapted to give a similar exponential improvement for all $\log k \ll \ell \le k$, 
but in order to simplify the presentation here, we will restrict our attention to the case $\ell = \Theta(k)$. 
The precise bounds we prove in this paper are stated in Theorems~\ref{thm:off:diagonal:near},~\ref{thm:off:diagonal:nearer} and~\ref{thm:off:diagonal:explicit}. 


Let us also note here that there is unfortunately still a large gap between the best known upper and lower bounds for $R(k,\ell)$, except in the cases $\ell = 3$ (see~\cite{AKSz,Boh,BK,Kim,FGM,Sh83}) and $\ell = 4$, due to the recent breakthrough work by Mattheus and Verstraete~\cite{MV}. 


The approach to bounding $R(k)$ using quasirandomness, which was pioneered in~\cite{T88} and later refined in~\cite{C09,S23}, has been extremely influential in combinatorics and computer science, with many exciting developments over the past 35 years, see for example the survey~\cite{KS}. The approach we use to prove Theorems~\ref{thm:diagonal} and~\ref{thm:off:diagonal} is very different, however, and does not involve any notion of quasirandomness. Indeed, the main idea we use from previous work is that to bound $R(k)$ it suffices to find a sufficiently large `book' $(S,T)$, that is, a graph on vertex set $S \cup T$ that contains all edges incident to $S$. To be precise, we will use the following observation: if every red-blue colouring of $E(K_n)$ contains a monochromatic book with $|T| \ge R(k,k - |S|)$, then $R(k) \le n$. Motivated by this observation (which goes back to the original paper of Ramsey), there has recently been significant progress on determining the Ramsey numbers of books~\cite{C19,CFW22,CFW23} using the regularity method, and advanced techniques related to quasirandomness. Our approach, on the other hand, is more elementary: 
we will introduce a new algorithm that allows us to find a large monochromatic book in a colouring $\chi$ with no monochromatic copy of $K_k$. Unfortunately, this book will not be \emph{quite} large enough to allow us to complete the proof using the Erd\H{o}s--Szekeres bound~\eqref{eq:ESz:bound}. Fortunately, however, our algorithm works even better away from the diagonal, and allows us to prove the bound 
\begin{equation}\label{eq:off:diag:just:enough}
R(k,\ell) \le e^{-\ell/50 + o(k)} {k + \ell \choose \ell}
\end{equation}
when $\ell \le k/4$. Applying~\eqref{eq:off:diag:just:enough} inside $T$, where $(S,T)$ is the book given by our algorithm applied to $\chi$, gives our first proof of Theorem~\ref{thm:diagonal}. We remark that~\eqref{eq:off:diag:just:enough} is quite far from the best bound that one can obtain with our method; however, it has the advantage of being (just) strong enough to imply Theorem~\ref{thm:diagonal}, while also having a relatively simple proof. 

In Sections~\ref{outline:sec} and~\ref{book:alg:sec} we will first informally outline our approach, and then define precisely the `Book Algorithm' that we will use to prove our main theorems. Before doing so, however, we shall attempt to motivate our approach by comparing it to that of Erd\H{o}s and Szekeres. We think of their proof as an algorithm that builds a red clique $A$ and a blue clique $B$, and tracks the size of the set\footnote{We write $N_R(x)$ for the set of vertices $y$ such that $xy$ is coloured red by $\chi$, and similarly for $N_B(x)$.} 
$$X = \bigcap_{u \in A} N_R(u) \cap \bigcap_{v \in B} N_B(v)$$
of vertices that send only red edges into $A$ and blue edges into $B$. In each step, they choose a vertex $x \in X$, and add it to either $A$ or $B$, depending on the size of $N_B(x) \cap X$. More precisely, when bounding $R(k,\ell)$ they\footnote{This is not \emph{exactly} the algorithm of Erd\H{o}s and Szekeres, but it gives the same bound up to a factor that is polynomial (and thus sub-exponential) in $k$, and such factors will not concern us in this paper.} add $x$ to $B$ if $|N_B(x) \cap X| \ge \gamma |X|$, where $\gamma = \frac{\ell}{k+\ell}$. 

In our algorithm, we will choose one of the colours (red, say), and track not only the size of $X$, but also the size of a certain subset
$$Y \subset \bigcap_{u \in A} N_R(u)$$
of the vertices that send only red edges into $A$; our red book will be the final pair $(A,Y)$. We would like to add vertices to $A$ and $B$ one by one, as in the Erd\H{o}s--Szekeres algorithm, but there is a problem: if we are not careful, the density $p$ of red edges between $X$ and $Y$ could decrease significantly in one of these steps. This would then have the knock-on effect of making $Y$ shrink much faster than we can afford in later steps.


In order to deal with this issue (which is, in fact, the main challenge of the proof), we introduce a new `density-boost step', which we will use whenever the density of blue edges inside $X$ is too low for us to take a `blue' step, and when taking a normal `red' step would cause the density of red edges between $X$ and $Y$ to drop by too much. The definition of this step is very simple: we add a (carefully-chosen) vertex $x$ to $B$, and replace $Y$ by $N_R(x) \cap Y$. We will show that in each step we can either add a red vertex to $A$, or many blue vertices to $B$, without decreasing $p$ too much, or we can perform a density-boost step that increases $p$ by a significant amount. In particular, 
the increase in $p$ will be proportional to the decrease in $p$ that we allow in a red step, and inversely proportional to the size of the set $N_B(x) \cap X$. This will be important in the proof, since density-boost steps are expensive (in the sense that they decrease the size of both $X$ and $Y$), and we therefore need to control them carefully. 

In the next section we will describe more precisely (though still informally) the three `moves' that we use in the algorithm: red steps, big blue steps, and density-boost steps. We will also explain how we choose the `size' of each step (which plays a crucial role in the proof), why we need to add `many' vertices to $B$ in a single step, and why our approach works better away from the diagonal. Having motivated each of the steps of the algorithm, we will then define it precisely in Section~\ref{book:alg:sec}. The analysis of the algorithm is carried out in Sections~\ref{big:blue:sec}--\ref{zigzag:sec}, and our exponential improvements on $R(k,\ell)$ are proved in Sections~\ref{simple:sec}--\ref{finalproof:sec}.

\section{An outline of the algorithm}\label{outline:sec}

In this section we will give a more detailed (though still imprecise) description of our algorithm for finding a large monochromatic book in a colouring that contains no large monochromatic cliques. The algorithm will be defined precisely in Section~\ref{book:alg:sec}, below.

To set the scene, let $k,\ell \in \N$ with $\ell \le k$, let $n \in \N$, and let $\chi$ be a red-blue colouring of $E(K_n)$ with no red $K_k$ and no blue $K_\ell$. During the algorithm we will maintain disjoint sets of vertices $X$, $Y$, $A$ and $B$ with the following properties:
\begin{itemize}
\item[$(a)$] all edges inside $A$ and between $A$ and $X \cup Y$ are red;\smallskip
\item[$(b)$] all edges inside $B$ and between $B$ and $X$ are blue.
\end{itemize}
We will write $p$ for the density of red edges between $X$ and $Y$, that is,
$$p = \frac{e_R(X,Y)}{|X||Y|},$$
where $e_R(X,Y)$ denotes the number of red edges with one end-vertex in each of $X$ and $Y$. We begin the algorithm with $A = B = \emptyset$, and $(X,Y)$ a partition of the vertex set of our colouring. Our aim is to build a large red clique $A$, while also keeping $Y$ as large as possible; to do so, it will be important that we carefully control the evolution of the density $p$. \bigskip

\begin{figure}[ht]
\centering
\includegraphics[width=0.5\textwidth]{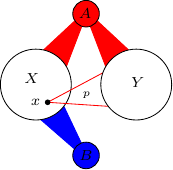}
\caption{The sets $A$, $B$, $X$ and $Y$. All edges incident to $A$ are red, and all edges inside $B$ and between $B$ and $X$ are blue. }
\label{fig:ABXY}
\end{figure}

\vskip-1cm

\subsection{Three moves}\label{3moves:chat:sec}

We will use three different `moves' to build the sets $A$ and $B$; we call these moves `red steps', `big blue steps' and `density-boost steps', respectively. To begin, let us describe how we update the four sets in each case. 

\medskip
\noindent \textbf{Red steps:} Ideally, we would like to choose a vertex $x \in X$ with `many' red neighbours in both $X$ and $Y$, and update the sets as follows:
$$X \to N_R(x) \cap X, \qquad Y \to N_R(x) \cap Y, \qquad A \to A \cup \{x\}, \qquad B \to B.$$
That is, we add $x$ to $A$, and replace $X$ and $Y$ by the red neighbourhood of $x$ in each set. Note that this move maintains properties $(a)$ and $(b)$ above. We will only be able to afford to make this move, however, if it does not cause either the size of the sets $X$ and $Y$, or the density $p$ of red edges between $X$ and $Y$, to decrease by too much. 

\medskip
\noindent \textbf{Big blue steps:} One way in which we could be prevented from taking a red step is if no vertex of $X$ has sufficiently many red neighbours in $X$. In this case $X$ must have many vertices of high blue degree, and we can use these vertices to find a large blue book $(S,T)$ inside $X$ (in particular, with $|S| \gg 1$). We then update the sets as follows:
$$X \to T, \qquad Y \to Y, \qquad A \to A, \qquad B \to B \cup S$$
That is, we add $S$ to $B$, and replace $X$ by $T$. Note that this move also maintains the two properties $(a)$ and $(b)$ above. In order to bound the decrease in $p$ due to this move, we will need to pre-process the set $X$ before each step, removing vertices whose red neighbourhood in $Y$ is significantly smaller than $p|Y|$ (note that this boosts the red density between the remaining vertices and $Y$). Since $\chi$ contains no blue $K_\ell$, there can only be $o(\ell)$ big blue steps, and therefore the decrease in $p$ due to these steps will not be significant. 

\medskip
\noindent \textbf{Density-boost steps:} If there are few vertices of high blue degree in $X$, and also every potential red step causes $p$ to decrease by more than we can afford, we will instead attempt to choose subsets of $X$ and $Y$ that `boost' the red density. We will do this in a very simple way: we will choose a vertex $x$ of low blue degree in $X$, and update the sets as follows:
$$X \to N_B(x) \cap X, \qquad Y \to N_R(x) \cap Y, \qquad A \to A, \qquad B \to B \cup \{x\}$$
That is, we add $x$ to $B$, and replace $X$ and $Y$ by the blue and red neighbourhoods of $x$, within the respective sets. Note that this move again maintains properties $(a)$ and $(b)$. 

Why does this move boost the red density between $X$ and $Y$? Roughly speaking, the reason is that, given the red density between $X$ and $N_R(x) \cap Y$, low red density between $N_R(x) \cap X$ and $N_R(x) \cap Y$ implies\footnote{This is simply because the sets $N_R(x) \cap X$ and $N_B(x) \cap X$ partition the set $X \setminus \{x\}$, see Lemma~\ref{lem:density:boost:steps:boost:the:density} for the precise statement. Note also that if moving $x$ into $A$ would cause $p$ to decrease significantly, then we do indeed have low red density between $N_R(x) \cap X$ and $N_R(x) \cap Y$.} high red density between $N_B(x) \cap X$ and $N_R(x) \cap Y$. However, to make this argument work, we need to choose the vertex $x$ so that the red density between $X$ and $N_R(x) \cap Y$ is not much smaller than $p$. We can do this because only few vertices have high blue degree, and the average red density over all vertices $x \in X$ is easily shown (by counting paths of length $2$ centred in $Y$, and using convexity) to be at least $p$. 

It is important to observe that we are able to add $x$ to the set $B$ in a density-boost step; if this were not the case, then the decrease in the size of $X$ would be too expensive. Note also that the increase in $p$ in a density-boost step is proportional to the decrease in $p$ that we allow in a red step; in the next subsection we will discuss how we exploit this fact.

\subsection{Choosing the sizes of the steps}\label{scale:sec}

Density-boost steps are expensive, for two reasons: they reduce the size of $Y$ without adding any vertices to $A$, and they can reduce the size of $X$ by a large factor, if $x$ has very low blue degree in $X$. We will therefore need to limit the number of density-boost steps (as a function of the number of red steps), and also their total `weight', measured by the decrease they cause in the size of $X$. 

We do so by varying the `scale' at which steps occur, depending on the current value of~$p$. To see why this is necessary, suppose first that instead we allowed the same decrease 
in~$p$ in every red step. Since we would like $p \ge p_0 - o(1)$ throughout the algorithm, where $p_0$ is the initial density of red edges (otherwise $|Y|$ would decrease too quickly), and we would like the algorithm to contain $\Omega(k)$ red steps (to construct a sufficiently large book), the maximum decrease we can allow is $\eps/k$ in each step, for some $\eps = o(1)$.\footnote{Here, and throughout the paper, the term $o(1)$ is assumed to tend to zero as $k \to \infty$.} In this case, however, we can only guarantee that the density of red edges increases by roughly $\eps/\ell$ in a density-boost step, and therefore even if there are $\ell$ density-boost steps (the maximum possible, since $\chi$ contains no blue $K_\ell$), the density of red edges will only increase by $o(1)$. 

The solution to this problem is simple: we change the maximum decrease that we allow in a red step from $\eps/k$ to (roughly) $\eps/k + \eps(p - p_0)$ whenever $p \ge p_0$. In Section~\ref{zigzag:sec} we will show that this implies that there can only be about $\eps^{-1} \log k$ more density-boost steps than we `expect' (that is, more than are cancelled out by the red steps). We will set $\eps = k^{-1/4}$ (we have a lot of flexibility in choosing the exact value) to ensure that this bound is much smaller than $k$, and that therefore these extra steps do not cause problems. 

To deal with the second issue, that density-boost steps can reduce the size of $X$ by a large factor, we will use the following important fact: if $N_B(x) \cap X$ is very small, then we get a correspondingly large boost in our density. This implies that if $X$ shrinks more than expected in the density-boost steps, then we obtain a stronger bound on the number of density-boost steps; as a consequence, the algorithm produces a smaller red clique $A$, but a larger red neighbourhood~$Y$. 
There is a delicate trade-off between these two effects (on the sizes of $A$ and $Y$), and the worst case overall for the size of $N_B \cap X$ lies somewhere between the worst case for $A$ and the worst case for $Y$. 
In particular, we cannot choose the exact size $t$ of the red clique $A$ that is produced by the algorithm, or the number $s$ of density-boost steps that are taken, and we instead need to deal with all possible pairs $(t,s)$.
  

\subsection{From diagonal to off-diagonal Ramsey numbers}\label{diag:chat:sec}

With the algorithm described (imprecisely) above in hand, our first attempt to prove an upper bound on $R(k)$ is as follows: we choose an equipartition $(X,Y)$ of the vertex set and, assuming (by symmetry) that the density of red edges between $X$ and $Y$ is at least $1/2$, run the algorithm to obtain a large red book $(A,Y)$. If we have 
\begin{equation}\label{eq:Y:vs:Rkk-t}
|Y| \ge {2k - t \choose k-t} \ge R(k,k-t),
\end{equation}
where $|A| = t$, then we can apply the Erd\H{o}s--Szekeres bound~\eqref{eq:ESz:bound} to find a monochromatic copy of $K_k$. This strategy fails, but only \emph{just}. More precisely, it only fails if the pair $(t,s)$ lies in a narrow strip near the point $(4k/5,4k/9)$, where $s$ denotes the number of density-boost steps that occur during the algorithm (see Figure~\ref{fig:ESz}). This suggests that in order to use the algorithm to prove Theorem~\ref{thm:diagonal}, we need to obtain a (not \emph{too} small) exponential improvement over the Erd\H{o}s--Szekeres bound for $R(k,\ell)$, where $\ell \approx k/5$, so that even if the first inequality in~\eqref{eq:Y:vs:Rkk-t} fails, we can still show that $|Y| \ge R(k,k-t)$. 

\begin{figure}[h]
  \centering
  \begin{subfigure}[b]{0.46\textwidth}
    \includegraphics[width=\textwidth]{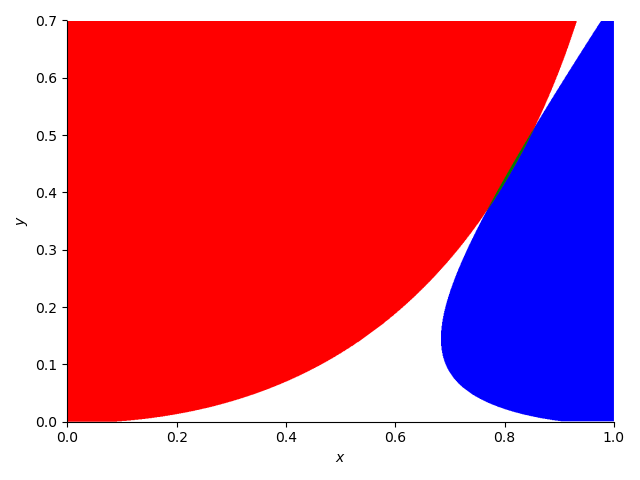}
  \end{subfigure}
  \hspace{0.5cm}
  \begin{subfigure}[b]{0.46\textwidth}
    \includegraphics[width=\textwidth]{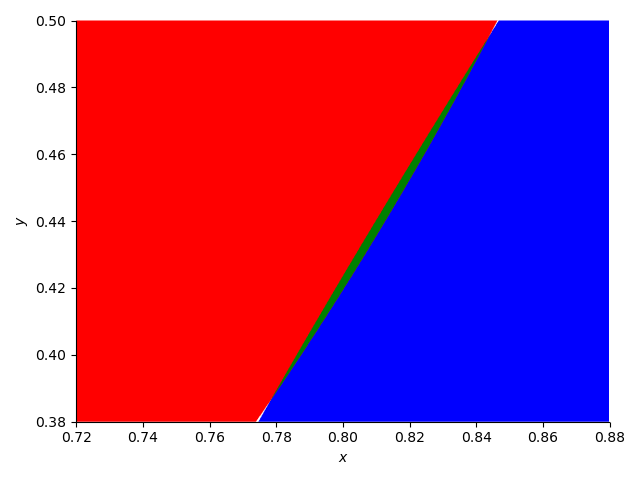}
  \end{subfigure}
   \vskip-0.2cm
\caption{We will show that if $n \ge (4 - o(1))^k$, then 
the pair $(x,y)$, where $x = t/k$ and $y = s/k$, lies in the blue region, and that if it lies \emph{outside} the red region, then $|Y| \ge {2k - t \choose k-t}$. We are therefore happy unless $(x,y)$ lies in the intersection of the red and blue regions, which is coloured green.}
\label{fig:ESz}
\end{figure}

\begin{figure}[h]
  \centering
  \begin{subfigure}[b]{0.46\textwidth}
      \includegraphics[width=\textwidth]{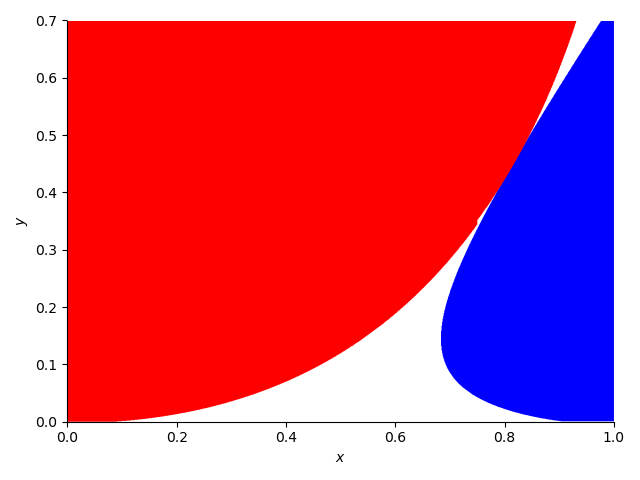}
  \end{subfigure}
  \hspace{0.5cm}
  \begin{subfigure}[b]{0.46\textwidth}
      \includegraphics[width=\textwidth]{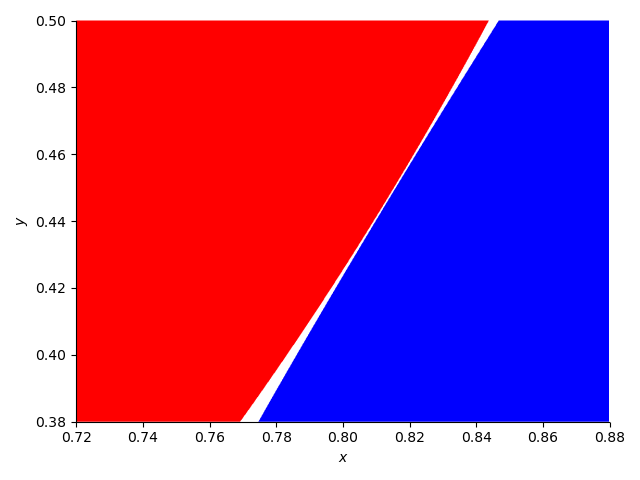}
  \end{subfigure}    
  \vskip-0.2cm
\caption{With our improved off-diagonal bound~\eqref{eq:off:diag:just:enough}, the red region (which corresponds to books $(A,Y)$ that are not big enough to find a monochromatic copy of $K_k$) becomes a little smaller, and the green region disappears.}
\label{fig:diagonal:intro}
\end{figure}

In order to obtain such an improvement, we again use the algorithm described above. However, we now have an additional advantage: since $\ell$ is now much smaller than $k$, we can afford to use a significantly lower threshold for a vertex to have `high' blue degree in $X$. This has the following crucial consequence: the increase in $p$ that we obtain from a density-boost step is now much larger (by a factor of roughly $k/\ell$) than the decrease in a red step.

This turns out to be sufficient to obtain a reasonable exponential improvement even when $\ell$ is only moderately smaller than $k$, though when $\ell \ge k/4$ some careful optimisation is needed. Fortunately the bound we need to deduce Theorem~\ref{thm:diagonal} is relatively weak, and can be proved with a relatively simple argument (see Sections~\ref{simple:sec} and~\ref{just:enough:sec}). 

We remark that a significant additional complication arises in the off-diagonal setting: due to the lack of symmetry, we cannot guarantee any lower bound on the density of red edges at the start of the algorithm. This would not be a problem if we only wanted an arbitrarily small exponential improvement over the Erd\H{o}s--Szekeres bound; that is not our situation, however, and we will therefore have to prepare the ground for our algorithm by taking blue ``Erd\H{o}s--Szekeres" steps until we find a subset of the vertices with suitable red density.

\subsection{Organisation of the paper}

The rest of the paper is organised as follows. First, in Section~\ref{book:alg:sec}, we will define precisely the algorithm outlined above. Next, in Sections~\ref{big:blue:sec}--\ref{zigzag:sec}, we prove a number of straightforward lemmas about the algorithm, the most challenging of which is the Zigzag Lemma, proved in Section~\ref{zigzag:sec}. We will then, in Sections~\ref{simple:sec} and~\ref{just:enough:sec}, be in a position to use the algorithm to obtain an exponential improvement for Ramsey numbers away from the diagonal (and, in particular, to prove~\eqref{eq:off:diag:just:enough}). Having done so, in Sections~\ref{moving:sec} and~\ref{diagonal:sec}, we will be able to use the algorithm to deduce Theorem~\ref{thm:diagonal} from~\eqref{eq:off:diag:just:enough}. Finally, in Sections~\ref{neardiagonal:sec} and~\ref{finalproof:sec}, we will prove Theorem~\ref{thm:off:diagonal}.

\section{The Book Algorithm}\label{book:alg:sec}

In this section we will define precisely the algorithm that we use to find our large red book. Let us fix sufficiently large integers $n$, $k$ and $\ell$, with $\ell \le k$, and suppose that $\chi$ is a red-blue colouring of $E(K_n)$ with no red $K_k$ or blue $K_\ell$. Let $X$ and $Y$ be disjoint sets of vertices, and let $p_0$ be the density of red edges between $X$ and $Y$ in $\chi$. We will update the sets $X$ and $Y$ as the process unfolds; recall that during the algorithm we will write 
$$p = \frac{e_R(X,Y)}{|X||Y|}$$
for the density of red edges between the \emph{current} sets $X$ and $Y$. 

Recall from Section~\ref{3moves:chat:sec} that in order to perform a red or density-boost step, we will need to choose a `central' vertex $x \in X$ such that the red density between $X$ and $N_R(x) \cap Y$ is not much smaller than $p$. In order to choose $x$, we define the \emph{weight} of an ordered pair $(x,y)$ of vertices in $X$ to be 
\begin{equation}\label{def:weight}
\omega(x,y) = \frac{1}{|Y|} \Big( |N_R(x) \cap N_R(y) \cap Y|  - p \cdot |N_R(x) \cap Y| \Big).
\end{equation}
Note that in general $\omega(x,y) \ne \omega(y,x)$. We will also write 
\begin{equation}\label{def:vtx:weight}
\omega(x) = \sum_{y \in X \setminus \{x\}} \omega(x,y),
\end{equation}
and refer to $\omega(x)$ as the weight of $x$. It is easy to show (see Observation~\ref{obs:weights}, below) that 
$$\sum_{x,y \in X} \omega(x,y) \ge 0.$$
Next, to define the scale at which we perform each step, set $\eps = k^{-1/4}$ and for each $p \in (0,1)$, define $h(p) \in \N$, the \emph{height} of $p$, to be the minimal positive integer such that 
\begin{equation}\label{def:height}
p \le q_h := p_0 + \frac{(1+\eps)^h - 1}{k}.
\end{equation}
Observe that $1 \le h(p) \le (2/\eps) \log k$ for all $p \in (0,1)$, and that $q_0 = p_0$. 
The maximum decrease in $p$ that we will allow in a red step will be $\alpha_{h(p)}$, where 
\begin{equation}\label{def:alpha}
\alpha_h = q_h - q_{h-1} = \frac{\eps(1+\eps)^{h-1}}{k}
\end{equation}
for each $h \in \N$. Finally, we will write $\mu \in (0,1/2]$ for the threshold density of blue edges for taking a big blue step; $\mu$ will be fixed throughout the algorithm, and moreover (in this paper) will be bounded away from both $0$ and $1$ as $k \to \infty$. In the off-diagonal setting we will set $\mu = \ell / (k + \ell)$, while in the diagonal setting we will (perhaps surprisingly) use $\mu = 2/5$.\footnote{This is not exactly the optimal value of $\mu$ to use, but it is close, and the slight loss will not make a significant difference (except in simplifying some of the calculations).} 

We are now ready to define the algorithm that we use to construct a large red book. 

\begin{RBalg}
Let $\chi$ be a red-blue colouring of $E(K_n)$ with no red $K_k$ or blue~$K_\ell$, let $X$ and $Y$ be disjoint sets of vertices of $K_n$, and let $\mu \in (0,1/2)$ be a fixed constant. Set~$A = B = \emptyset$. We run the following process until either $|X| \le R(k,\ell^{3/4})$ or $p \le 1/k$:\smallskip 
\begin{enumerate}[label=\arabic*., ref=\arabic*] \item\label{Alg:Step1} Degree regularisation: We remove from $X$ all vertices with few red neighbours in~$Y$; that is, we update $h \to h(p)$, and update $X$ as follows:
$$X \to \big\{ x \in X : |N_R(x) \cap Y| \ge \big( p - \eps^{-1/2} \alpha_h \big) |Y| \big\}.$$ 
Now once again update\footnote{Note that $p$ may have increased when we updated $X$, and therefore $h$ may also have increased.} $h \to h(p)$ and go to Step~2.\smallskip 
\item\label{Alg:Step2} Big blue step: If there exist $R(k,\ell^{2/3})$ vertices $x \in X$ such that
$$|N_B(x) \cap X| \ge \mu |X|,$$
then let $(S,T)$ be a blue book in $X$ with $|T| \ge \mu^{|S|} |X|/2$, with $|S|$ as large as possible. 
Now update the sets as follows:
$$X \to T, \qquad Y \to Y, \qquad A \to A, \qquad B \to B \cup S$$
and go to Step~1. Otherwise go to Step~3.\smallskip
\item\label{Alg:Step3} The central vertex: Choose a vertex $x \in X$ with $\omega(x)$ maximal such that
$$|N_B(x) \cap X| \le \mu |X|,$$
and go to Step~4.\smallskip
\item\label{Alg:Step4} Red step: If the density of red edges between $N_R(x) \cap X$ and $N_R(x) \cap Y$ is at least 
$$p - \alpha_h,$$
then we update the sets as follows:
$$X \to N_R(x) \cap X, \qquad Y \to N_R(x) \cap Y, \qquad A \to A \cup \{x\}, \qquad B \to B$$
and go to Step~1. Otherwise go to Step~5.\smallskip
\item\label{Alg:Step5} Density-boost step: We update the sets as follows:
$$X \to N_B(x) \cap X, \qquad Y \to N_R(x) \cap Y, \qquad A \to A, \qquad B \to B \cup \{x\}$$
and go to Step~1.
\end{enumerate}  
\end{RBalg} 

Observe that the algorithm continues until $|X| \le R(k,\ell^{3/4}) = 2^{o(k)}$, unless at some point the density of red edges between $X$ and $Y$ becomes very small (which, as we will show later, never happens). Note that the algorithm terminates, since in each round at least one vertex is removed from $X$ and added to $A \cup B$.  

In Sections~\ref{Y:sec} and~\ref{X:sec} we will bound the sizes of the sets $X$ and $Y$ in terms of the number of red and density-boost steps. For the proofs of these bounds we will find it convenient to introduce an index $i \in \N \cup \{0\}$ that increases whenever the set $X$ is updated\footnote{Here we will find it convenient to include updates that do not change $X$; note that this can occur in a degree regularisation step.}, and write $X_i$ and $Y_i$ for the sets $X$ and $Y$ at time $i$. We thus obtain sequences 
\begin{equation}\label{def:X:Y:sequences}
X_0 \supseteq X_1 \supseteq \cdots \supseteq X_m \qquad \text{and} \qquad Y_0 \supseteq Y_1 \supseteq \cdots \supseteq Y_m
\end{equation}
for some $m \in \N$, where $(X_0,\ldots,X_m)$ is the sequence of sets $X$ that appear during the algorithm. We will write $p_i$ for the density of red edges between $X_i$ and $Y_i$, 
and $\cR$, $\cB$, $\cS$ and $\cD$ for the sets of indices such that $X_i$ was formed by a red step, a big blue step, a density-boost step, and a degree regularisation step, respectively. We thus obtain a partition 
$$\cR \cup \cB \cup \cS \cup \cD = [m].$$
For each $i \in \cR \cup \cS$, let $x_i$ be the central vertex of the corresponding red or density-boost step, and define $0 \le \beta_i \le \mu$ by
\begin{equation}\label{def:beta:is}
|N_B(x_i) \cap X_{i-1}| = \beta_i |X_{i-1}|.
\end{equation}
Finally, we will write $t = |\cR|$ for the number of red steps, and $s = |\cS|$ for the number of density-boost steps taken during the algorithm. 

In the next five sections we will prove a number of fundamental properties of the Book Algorithm. We will assume throughout these sections that we are in the setting of the algorithm, and that the sets $X_i$ and $Y_i$ are as in~\eqref{def:X:Y:sequences}. Since it will be the case in all of our applications, and will simplify things slightly, let us also assume that\footnote{All of the lemmas below actually hold under the weaker assumptions $\mu_0 \le \mu \le 1 - \mu_0$ and $p_0 \ge \mu_0$.} 
\begin{equation}\label{eq:assumptions}
0 < \mu_0 \le \mu \le \frac{1}{2} \le p_0 \le 1
\end{equation}
for some constant $\mu_0 > 0$, 
and that both $k$ and $\ell$ are sufficiently large, depending only on $\mu_0$. Under these assumptions, all of the $o(k)$ error terms in the lemmas in Sections~\ref{big:blue:sec}--\ref{zigzag:sec} can be taken to be equal to $k^{1-c}$ for some absolute constant $c > 0$. 


\section{Big blue steps}\label{big:blue:sec}

In this section we begin our analysis of the Book Algorithm by proving the following simple (but crucial) lemma, which implies that big blue steps really are big (and hence that there are few big blue steps). Recall that we only perform a big blue step if there exist $R(k,\ell^{2/3})$ vertices $x \in X$ with at least $\mu |X|$ blue neighbours in $X$. The lemma states that whenever this is the case, we can find a large blue book $(S,T)$ in $X$ with 
$|T| \ge \mu^{|S|} |X|/2$.


\begin{lemma}\label{lem:big:blue:step}
Set $b = \ell^{1/4}$. If there are $R(k,\ell^{2/3})$ vertices $x \in X$ such that
\begin{equation}\label{eq:blue:degree:mu}
|N_B(x) \cap X| \ge \mu |X|,
\end{equation}
then $X$ contains either a red $K_k$, or a blue book $(S,T)$ with $|S| \ge b$ and $|T| \ge \mu^{|S|} |X|/2$.  
\end{lemma}

In the proof of the lemma we will need the following simple inequalities; for the reader's convenience, we provide a proof in Appendix~\ref{app:simple}.


\begin{fact}\label{binomial:fact1}
Let $a,b \in \N$ and $\sigma \in (0,1)$, with $b \le \sigma a / 2$. Then
$$ \sigma^b {a \choose b} \exp\bigg( - \frac{b^2}{\sigma a} \bigg) \le {\sigma a \choose b} \le \sigma^b {a \choose b}  .$$ 
\end{fact}

To prove Lemma~\ref{lem:big:blue:step}, we simply choose a random subset of size $b$ in a blue clique of $a = \ell^{2/3}$ vertices that satisfy~\eqref{eq:blue:degree:mu}, and count the expected number of common blue neighbours. 

\begin{proof}[Proof of Lemma~\ref{lem:big:blue:step}]
Let $W \subset X$ be the set of vertices with blue degree at least $\mu |X|$, set $a = \ell^{2/3}$, and note that $|W| \ge R(k,a)$, so $W$ contains either a red $K_k$ or a blue~$K_a$. In the former case we are done, so assume that $U \subset W$ is the vertex set of a blue $K_a$. Let $\sigma$ be the density of blue edges between $U$ and $X \setminus U$, and observe that 
\begin{equation}\label{eq:sigma:bound}
\sigma \,=\, \frac{e_B(U,X \setminus U)}{|U| \cdot |X \setminus U|} \,\ge\, \frac{\mu |X| - |U|}{|X| - |U|} \,\ge\, \mu - \frac{1}{k},
\end{equation}
since $|U| = a$ and $|X| \ge R(k,a) \gg ka$, and each vertex of $U$ has at least $\mu |X| - |U|$ blue neighbours in $X \setminus U$. Since $\mu \ge \mu_0$, both $k$ and $\ell$ are sufficiently large (depending on $\mu_0$), and $b = \ell^{1/4}$ and $a = \ell^{2/3}$, it follows that $b \le \sigma a / 2$. 

Let $S \subset U$ be a uniformly-chosen random subset of size $b$, and let $Z = |N_B(S) \cap (X \setminus U)|$ be the number of common blue neighbours of $S$ in $X \setminus U$. By convexity, we have
$$\Ex[Z] \,=\, {a \choose b}^{-1} \sum_{v \in X \setminus U} {|N_B(v) \cap U| \choose b} \,\ge\, {a \choose b}^{-1} {\sigma a \choose b} \cdot |X \setminus U|.$$
Now, by Fact~\ref{binomial:fact1}, and recalling~\eqref{eq:sigma:bound}, and that $b = \ell^{1/4}$ and $a = \ell^{2/3}$, it follows that
\begin{equation}\label{eq:bigblue:ExZ:convexity}
\Ex[Z] \,\ge\, \sigma^b \exp\bigg( - \frac{b^2}{\sigma a} \bigg) \cdot |X \setminus U| \,\ge\, \frac{\mu^b}{2} \cdot |X|,
\end{equation}
and hence there exists a blue clique $S \subset U$ of size $b$ with at least this many common blue neighbours in $X \setminus U$, as required.
\end{proof}

It follows immediately from Lemma~\ref{lem:big:blue:step} that since $\chi$ contains no red $K_k$ and no blue $K_\ell$, there are not too many big blue steps during the algorithm. 

\begin{lemma}\label{lem:few:big:blue}
$|\cB| \le \ell^{3/4}$. That is, there are at most $\ell^{3/4}$ big blue steps during the algorithm.
\end{lemma}

\begin{proof}
Since $\chi$ contains no red $K_k$, by Lemma~\ref{lem:big:blue:step} we add at least $\ell^{1/4}$ vertices to $B$ in each big blue step. Therefore, after $\ell^{3/4}$ big blue steps we would have $|B| \ge \ell$. Since $B$ is a blue clique, this would contradict our assumption that $\chi$ contains no blue $K_\ell$. 
\end{proof}

For the sake of comparison, and since they will be needed later on, we take the opportunity to note the following easy bounds on the number of steps of each type. 

\begin{obs}\label{obs:RBSD:sizes}
$$|\cR| \le k, \qquad |\cB| + |\cS| \le \ell \qquad \text{and} \qquad |\cD| \le k + \ell + 1.$$
\end{obs}

\begin{proof}
In each red step we add a vertex to $A$, and in each big blue or density-boost step we add at least one vertex to $B$. Since $\chi$ does not contain any red $K_k$ or blue $K_\ell$, we have $|A| < k$ and $|B| < \ell$ at the end of the algorithm, and so the first two inequalities hold. The third inequality now follows, since each degree regularisation step (except maybe the last) is followed by a red, big blue or density-boost step, so $|\cD| \le |\cR| + |\cB| + |\cS| + 1$. 
\end{proof}

\section{Density-boost steps boost the density}\label{density:sec}

The goal of this section is to prove the following lemma, which shows that density-boost steps really do boost the density of red edges between $X$ and $Y$. Moreover, and crucially, the increase in $p$ is (at least) proportional to $\alpha_{h(p)}$, and inversely proportional to $\beta_i$. 

\begin{lemma}\label{lem:density:boost:steps:boost:the:density}
Let $i \in \cR \cup \cS$. Then at least one of the following holds:\,\footnote{Recall that $x_i$ is the central vertex in Step~$i$, and that $\beta_i$ was defined in~\eqref{def:beta:is}. In order to avoid having too many distracting subscripts, we write $X$, $Y$ and $p$ for $X_{i-1}$, $Y_{i-1}$ and $p_{i-1}$.} 
\begin{itemize}
\item[$(a)$] the density of red edges between $N_R(x_i) \cap X$ and $N_R(x_i) \cap Y$ is at least 
$$p - \alpha_{h(p)};$$ 
\item[$(b)$] $\beta_i > 0$, and the density of red edges between $N_B(x_i) \cap X$ and $N_R(x_i) \cap Y$ is at least 
\begin{equation}\label{eq:density:boost}
p + \big(1 - \eps \big) \bigg( \frac{1-\beta_i}{\beta_i} \bigg) \, \alpha_{h(p)}.
\end{equation}
\end{itemize}
\end{lemma}


We will in fact use the following immediate consequence of Lemma~\ref{lem:density:boost:steps:boost:the:density}.

\begin{lemma}\label{lem:density:boost:steps:boost:the:density:notation}
If\/ $i \in \cS$, then $\beta_i > 0$ and 
$$p_i - p_{i-1} \ge \big(1 - \eps \big) \bigg( \frac{1-\beta_i}{\beta_i} \bigg) \, \alpha_h,$$
where $h = h(p_{i-1})$. 
\end{lemma}

\begin{proof}
Recall that if $i \in \cS$, then we reached Step~\ref{Alg:Step5} of the algorithm, and we therefore chose (in Step~\ref{Alg:Step3}) a central vertex $x_i \in X_{i-1}$ with $|N_B(x_i) \cap X_{i-1}| = \beta_i |X_{i-1}|$ for some $\beta_i \le \mu$. Moreover, since we did not perform a red step, the density of red edges between the sets $N_R(x_i) \cap X_{i-1}$ and $N_R(x_i) \cap Y_{i-1}$ is strictly less than $p - \alpha_{h(p)}$, where $p = p_{i-1}$. 

By Lemma~\ref{lem:density:boost:steps:boost:the:density}, it follows that the density of red edges between the sets $N_B(x_i) \cap X_{i-1}$ and $N_R(x_i) \cap Y_{i-1}$ (which is equal to $p_i$, by Step~\ref{Alg:Step5} of the algorithm) is at least~\eqref{eq:density:boost}, as claimed.
\end{proof}

Our main application of Lemma~\ref{lem:density:boost:steps:boost:the:density:notation} will be in Section~\ref{zigzag:sec}, where we bound $s$, the number of density-boost steps that occur during the algorithm. 
We also record here the following weaker bounds, which we will use in Sections~\ref{Y:sec} and~\ref{X:sec} to control the sizes of $X$ and~$Y$.  

\begin{lemma}\label{lem:density:boost:weak}
If\/ $i \in \cS$, then\/ $p_i \ge p_{i-1}$ and\/ $\beta_i \ge 1/k^2$.
\end{lemma}

\begin{proof}
Note that $\alpha_h \ge 0$ for every $h \in \N$, by~\eqref{def:alpha}, and recall that $\beta_i \le \mu \le 1/2$ for every $i \in \cS$. It thus follows from Lemma~\ref{lem:density:boost:steps:boost:the:density:notation} that $p_i \ge p_{i-1}$. The bound on $\beta_i$ also follows from Lemma~\ref{lem:density:boost:steps:boost:the:density:notation}, 
since $\alpha_h \ge \eps/k = k^{-5/4}$, by~\eqref{def:alpha}, and $p_i - p_{i-1} \le 1$, by definition. 
\end{proof}

In order to prove Lemma~\ref{lem:density:boost:steps:boost:the:density}, we first need to bound the weight of the central vertex.

\subsection{The weight of the central vertex} 

Recall that in Step~\ref{Alg:Step3} of the algorithm we choose the central vertex $x \in X$ to maximise $\omega(x)$, subject to the condition that $x$ has at most $\mu |X|$ blue neighbours in $X$, where $\omega(x)$ was defined in~\eqref{def:vtx:weight}. The purpose of this choice is to guarantee that the density of red edges between $X$ and $N_R(x) \cap Y$ is not too much smaller than $p$. We will prove the following lower bound on the weight of the central vertex. 


\begin{lemma}\label{lem:weight:bound}
If\/ $i \in \cR \cup \cS$, then
$$\omega(x_i) \ge - \frac{|X_{i-1}|}{k^5}.$$
\end{lemma}

We will prove Lemma~\ref{lem:weight:bound} using the following simple application of convexity. 

\begin{obs}\label{obs:weights}
\begin{equation} 
\sum_{x,y \in X} \omega(x,y) \ge 0 
\end{equation}
\end{obs}

\begin{proof}
By the definition~\eqref{def:weight} of $\omega(x,y)$, we are required to show that
$$\sum_{x,y \in X} |N_R(x) \cap N_R(y) \cap Y| \ge \sum_{x,y \in X} p \cdot |N_R(x) \cap Y|.$$
To see this, note first that 
$$\sum_{x,y \in X} p \cdot |N_R(x) \cap Y| = p |X| \cdot e_R(X,Y) = p^2 |X|^2 |Y|.$$
Now, counting red walks of length 2 centred in $Y$, and using convexity, 
we obtain
$$\sum_{x,y \in X} |N_R(x) \cap N_R(y) \cap Y| = \sum_{z \in Y} |N_R(z) \cap X|^2 \ge p^2 |X|^2 |Y|,$$
so the claimed inequality follows.  
\end{proof} 

We also need the following weak bound on the relative sizes of different Ramsey numbers.

\begin{lemma}\label{lem:simple:Ramsey:ratio}
Let $\ell \in \N$ be sufficiently large. Then
$$R(k,\ell^{3/4}) \ge k^6 \cdot R(k,\ell^{2/3})$$
for all $\ell \le k \in \N$.
\end{lemma} 

\begin{proof}
By the Erd\H{o}s--Szekeres bound~\eqref{eq:ESz:bound}, we have
$$R(k,\ell^{2/3}) \le {k + \ell^{2/3} \choose \ell^{2/3}} \le \exp\big( \ell^{2/3} \log k \big).$$
On the other hand, if we colour the edges independently at random, with probability $k^{-1/8}$ of being blue, then a simple (and standard) first moment argument (see~\cite[Section~3.1]{AS}) shows that 
$$R(k,\ell^{3/4}) \ge \exp\big( c \hspace{0.05cm} \ell^{3/4} \log k \big)$$
for some absolute constant $c > 0$. Comparing the two bounds, the lemma follows. 
\end{proof}

We can now deduce our lower bound on the weight of the central vertex.

\begin{proof}[Proof of Lemma~\ref{lem:weight:bound}]
Note first that if $i \in \cR \cup \cS$, then at most $R(k,\ell^{2/3})$ vertices of $X_{i-1}$ have more than $\mu |X_{i-1}|$ blue neighbours in $X_{i-1}$, since otherwise we would instead have performed a big blue step. Set $X = X_{i-1}$, note that
$$0 \le \sum_{x,y \in X} \omega(x,y) = \sum_{y \in X} \big( \omega(y) + \omega(y,y) \big),$$
by Observation~\ref{obs:weights}, and recall that we chose $x_i$ to maximise $\omega(x_i)$ among vertices with at most $\mu |X|$ blue neighbours in $X$. Since $|\omega(x,y)| \le 1$ for every $x,y \in X$, and hence $\omega(y) + \omega(y,y) \le |X|$ for every $y \in X$, it follows that 
$$0 \le R(k,\ell^{2/3}) \cdot |X| + \big( |X| - R(k,\ell^{2/3}) \big) \big( \omega(x_i) + 1 \big),$$
Thus, recalling that $|X| \ge R(k,\ell^{3/4}) \ge k^6 \cdot R(k,\ell^{2/3})$, by Lemma~\ref{lem:simple:Ramsey:ratio}, we have 
$$\frac{\omega(x_i)}{|X|} \ge - \frac{R(k,\ell^{2/3})}{|X| - R(k,\ell^{2/3})} - \frac{1}{|X|} \ge - \frac{1}{k^5},$$
as claimed. 
\end{proof}

\subsection{Density-boost steps}

Before proving Lemma~\ref{lem:density:boost:steps:boost:the:density}, let us make a couple of simple observations, both of which will be used again later in the proof. 

\begin{obs}\label{obs:alpha:bounds}
We have
$$\eps / k \le \alpha_{h(p)} \le \eps \big( p - q_0 + 1/k \big)$$
for all\/ $p \ge q_0$. Moreover, if\/ $p \le q_1$ then $\alpha_{h(p)} = \alpha_1 = \eps/k$.
\end{obs}

\begin{proof}
Set $h = h(p)$, and observe that, by~\eqref{def:height} and~\eqref{def:alpha}, and since $h \ge 1$, we have 
$$\alpha_h = \frac{\eps (1+\eps)^{h-1}}{k} = \eps \big( q_{h-1} - q_0 + 1/k \big).$$
The claimed bounds now follow because $p \ge q_{h(p)-1} \ge q_0$ for all $p \ge q_0$. 
\end{proof}

Our second observation gives a lower bound on the size of the red neighbourhood in $Y_i$ of a vertex in $X_i$ immediately after a degree regularisation step.

\begin{obs}\label{obs:red:degree:lower:bound}
If\/ $i \in \cD$, then 
$$|N_R(x) \cap Y_i| \ge \big( 1 - \eps^{1/2} \big) p_{i-1} |Y_i|$$
for every $x \in X_i$. 
\end{obs}

\begin{proof}
Set $p = p_{i-1}$, and note that, since $i \in \cD$, the set $X_i$ is formed by removing from $X_{i-1}$ all vertices with at most $\big( p - \eps^{-1/2} \alpha_{h(p)} \big) |Y_i|$ red neighbours in the set $Y_i = Y_{i-1}$. 
Recall that $\eps = k^{-1/4}$, and note that $p \ge 1/k$ (since otherwise the algorithm would have halted). By~Observation~\ref{obs:alpha:bounds}, it follows that 
$$\eps^{-1/2}  \alpha_{h(p)} \le \eps^{1/2} \big( p - q_0 + 1/k \big) \le \eps^{1/2} p$$
whenever $p \ge q_0$. If $p \le q_0$, then $\eps^{-1/2} \alpha_{h(p)} = \eps^{1/2} / k \le \eps^{1/2} p$, as claimed.
\end{proof}

We are now ready to show that density-boost steps boost the density. The proof is straightforward: we simply count red edges.

\begin{proof}[Proof of Lemma~\ref{lem:density:boost:steps:boost:the:density}]
Recall that $i \in \cR \cup \cS$, that $X = X_{i-1}$ and $Y = Y_{i-1}$, that $p = p_{i-1}$, and let $x = x_i$ be the current central vertex. Set $\alpha = \alpha_{h(p)}$, and suppose first that
\begin{equation}\label{eq:red:step}
\sum_{y \in N_R(x) \cap X} \omega(x,y) \ge - \, \alpha \cdot \frac{|N_R(x) \cap X| \cdot |N_R(x) \cap Y|}{|Y|}.
\end{equation}
Note that the number of red edges between $N_R(x) \cap X$ and $N_R(x) \cap Y$ is 
$$\sum_{y \in N_R(x) \cap X} |N_R(x) \cap N_R(y) \cap Y| \, = \, \sum_{y \in N_R(x) \cap X} \Big( p \cdot |N_R(x) \cap Y| + \omega(x,y) |Y| \Big),$$
by~\eqref{def:weight}, and observe that, by~\eqref{eq:red:step}, this is at least 
$$p \cdot |N_R(x) \cap X| \cdot |N_R(x) \cap Y| - \alpha \cdot |N_R(x) \cap X| \cdot |N_R(x) \cap Y|.$$ 
It follows that the density of red edges between $N_R(x) \cap X$ and $N_R(x) \cap Y$ is at least $p - \alpha$, and thus $(a)$ holds, as required. 

We may therefore assume, recalling~\eqref{def:vtx:weight}, that
\begin{equation}\label{eq:boost:step}
\sum_{y \in N_B(x) \cap X} \omega(x,y) \ge \, \omega(x) + \alpha \cdot \frac{|N_R(x) \cap X| \cdot |N_R(x) \cap Y|}{|Y|}.
\end{equation}
We claim that in this case $(b)$ holds. To show this, note first that if $\beta_i = 0$, then the left-hand side of~\eqref{eq:boost:step} is equal to zero. On the other hand, we have $\omega(x) \ge - |X| / k^5$, by Lemma~\ref{lem:weight:bound}, $\alpha \ge \eps / k$, by Observation~\ref{obs:alpha:bounds}, and $|N_R(x) \cap X| = |X| - 1$ if $\beta_i = 0$. Moreover, we have $p_{i-2} \ge 1/k$ (since otherwise the algorithm would have halted), and therefore
\begin{equation}\label{eq:NRY:not:too:tiny}
|N_R(x) \cap Y| \ge \big(1 - \eps^{1/2} \big) p_{i-2} |Y| \ge \frac{|Y|}{2k},
\end{equation}
by Observation~\ref{obs:red:degree:lower:bound} (note that $i - 1 \in \cD$, since $i \in \cR \cup \cS$, and that $x = x_i \in X_{i-1}$ and $Y = Y_{i-1}$). This implies that the right-hand side of~\eqref{eq:boost:step} is strictly positive, and this contradiction shows that $\beta_i > 0$, as required.

To prove~\eqref{eq:density:boost}, note first that there are 
$$\sum_{y \in N_B(x) \cap X} |N_R(x) \cap N_R(y) \cap Y| \, = \, \sum_{y \in N_B(x) \cap X} \Big( p \cdot |N_R(x) \cap Y| + \omega(x,y) |Y| \Big)$$
red edges between $N_B(x) \cap X$ and $N_R(x) \cap Y$, which, by~\eqref{eq:boost:step}, is at least 
$$p \cdot |N_B(x) \cap X| \cdot |N_R(x) \cap Y| + \alpha \cdot |N_R(x) \cap X| \cdot |N_R(x) \cap Y| + \omega(x) \cdot |Y|.$$
Therefore, recalling that $|N_B(x) \cap X| = \beta_i |X|$ (and thus $|N_R(x) \cap X| = (1 - \beta_i) |X| - 1$), it follows that the density of red edges between $N_B(x) \cap X$ and $N_R(x) \cap Y$ is at least
\begin{equation}\label{eq:boost:beta:ugly}
p + \bigg( \frac{1-\beta_i}{\beta_i} \bigg) \, \alpha  - \frac{\alpha}{\beta_i |X|} + \frac{\omega(x) \cdot |Y|}{\beta_i |X| \cdot |N_R(x) \cap Y|}.
\end{equation}

To bound the last two terms, note that $|X| \ge R(k,\ell^{3/4}) \ge k^5$, that $\alpha \le 1$, and that $\omega(x) \ge -|X|/k^5$, by Lemma~\ref{lem:weight:bound}. Since $|N_R(x) \cap Y| \ge |Y| / 2k$, by~\eqref{eq:NRY:not:too:tiny}, it follows that~\eqref{eq:boost:beta:ugly} is at least
\begin{equation}\label{eq:boost:beta:maybe:small}
p + \bigg( \frac{1-\beta_i}{\beta_i} \bigg) \, \alpha - \frac{3}{\beta_i k^4}.
\end{equation}
Finally, recall that $\beta_i \le \mu \le 1/2$, that $\alpha \ge k^{-5/4}$, and that $k$ is sufficiently large. 
The claimed bound therefore follows from~\eqref{eq:boost:beta:maybe:small}. 
\end{proof}

\section{Bounding the size of $Y$}\label{Y:sec}

Our aim in this section is to prove the following lemma, which bounds the size of the set $Y$ in terms of the initial density $p_0$ of red edges between $X$ and $Y$, and the number of red and density-boost steps that occur during the algorithm.  

\begin{lemma}\label{lem:Ybound}
$$|Y| \ge 2^{o(k)} p_0^{s + t} \cdot |Y_0|.$$
\end{lemma}


Recall that $Y$ is only updated in red and density-boost steps, when it is replaced by $N_R(x_i) \cap Y$. Lemma~\ref{lem:Ybound} is therefore an almost immediate consequence of the following lemma, which states that at every step of the algorithm, the density of red edges between the sets $X$ and $Y$ is not much smaller than it is at the start. 

\pagebreak

\begin{lemma}\label{lem:bounding:p}
For every $j \in [m]$, 
\begin{equation}\label{eq:p:lower:bound}
p_j \ge p_0 - 3\eps.
\end{equation}
\end{lemma}

In order to prove Lemma~\ref{lem:bounding:p}, we will consider pairs of consecutive steps that decrease the density and take place below $p_0$, so define
\begin{equation}\label{def:Z}
\cZ = \big\{ i \in \cR \cup \cB \cup \cS \,:\, p_i < p_{i-2} \le p_0 \big\}.
\end{equation}
Note that the steps of the algorithm alternate between the sets $\cD$ and $\cR \cup \cB \cup \cS$, so the steps of $\cZ$ capture the entire decrease below $p_0$. We will prove the following bound. 

\begin{lemma}\label{lem:Z:decrease}
$$\sum_{i \in \cZ} \big( p_{i-2} - p_i \big) \le 2\eps.$$
\end{lemma}

We will deduce Lemma~\ref{lem:Z:decrease} from the following simple bounds on the change in $p$ caused by the various types of step of the algorithm. 

\begin{lemma}\label{lem:jumping:p}
\[
    p_i \,\ge\, \left\{
    \begin{array} {c@{\quad \textup{if} \quad}l}
      p_{i-1} - \alpha_{h(p_{i-1})} & i \in \cR, \\[+1ex]
      p_{i-2} - \eps^{-1/2} \alpha_{h(p_{i-2})} & i \in \cB,\\[+1ex]
      p_{i-1} & i \in \cS \cup \cD.
    \end{array}\right.
  \]
\end{lemma}

\begin{proof}
The bound for $i \in \cR$ follows immediately from Step~\ref{Alg:Step4} of the algorithm, and the bound for $i \in \cS$ was proved in Lemma~\ref{lem:density:boost:weak}. The bounds for $i \in \cB$ and $i \in \cD$ both follow from Step~\ref{Alg:Step1} of the algorithm; indeed, removing vertices whose degree is less than the average 
can only increase the density from the remaining set, and if $i \in \cB$, then each vertex of $X_{i-1}$ has at least 
$$\big( p_{i-2} - \eps^{-1/2} \alpha_{h(p_{i-2})} \big) |Y_i|$$
red neighbours in $Y_i = Y_{i-2}$. 
\end{proof}

We can now bound the total decrease in pairs of consecutive steps that occur below $p_0$. 

\begin{proof}[Proof of Lemma~\ref{lem:Z:decrease}]
Note first that $p_i \ge p_{i-1} \ge p_{i-2}$ for all $i \in \cS$, by Lemma~\ref{lem:jumping:p} and since $i-1 \in \cD$, and therefore $\cS \cap \cZ = \emptyset$. Note also that if $i \in \cZ$, then $h(p_{i-2}) = 1$, and thus if $i \in \cB \cap \cZ$, then
$$p_{i-2} - p_i \le \eps^{-1/2} \alpha_1 \le \frac{1}{k},$$
by Lemma~\ref{lem:jumping:p} and~\eqref{def:alpha}, since $\alpha_1 = \eps/k$. Since $|\cB| \le \ell^{3/4}$, by Lemma~\ref{lem:few:big:blue}, we deduce that
$$\sum_{i \in \cB \cap \cZ} \big( p_{i-2} - p_i \big) \le \frac{\ell^{3/4}}{k} \le \eps.$$

Finally, if $i \in \cR \cap \cZ$, then we claim that 
$$p_{i-2} - p_i \le p_{i-2} - p_{i-1} + \alpha_h \le \frac{\eps}{k},$$
where $h = h(p_{i-1})$. Indeed, the first inequality follows from Lemma~\ref{lem:jumping:p}. To see the second, note first that $p_{i-2} \le p_{i-1}$, by Lemma~\ref{lem:jumping:p} and since $i - 1 \in \cD$, so if $h = 1$ the inequality follows from~\eqref{def:alpha}. If $h \ge 2$, on the other hand, then~$p_{i-2} \le q_0 \le p_{i-1}$, since $i \in \cZ$, and therefore, by Observation~\ref{obs:alpha:bounds},
$$p_{i-2} - p_{i-1} + \alpha_h \le q_0 - p_{i-1} + \eps\big( p_{i-1} - q_0 + 1/k \big) \le \frac{\eps}{k},$$
as claimed. Since $t \le k$, by Observation~\ref{obs:RBSD:sizes}, we deduce that
$$\sum_{i \in \cR \cap \cZ} \big( p_{i-2} - p_i \big) \le \eps,$$
as required.
\end{proof}

In order to deduce Lemma~\ref{lem:bounding:p} from Lemma~\ref{lem:Z:decrease}, we will need the following bounds on $h(p_i) - h(p_{i-1})$ and $h(p_i) - h(p_{i-2})$, which follow easily from the bounds in Lemma~\ref{lem:jumping:p}. 

\begin{lemma}\label{lem:jumping:down}
\[
    h(p_i) \,\ge\, \left\{
    \begin{array} {c@{\quad \textup{if} \quad}l}
      h(p_{i-1}) - 2 & i \in \cR, \\[+1ex]
      h(p_{i-2}) - 2 \eps^{-1/2} & i \in \cB, \\[+1ex]
      h(p_{i-1}) & i \in \cS \cup \cD.
    \end{array}\right.
  \]
\end{lemma}

\begin{proof}
Let $i \in \cR$, and set $p = p_{i-1}$ and $h = h(p_{i-1})$. If $h \le 3$ then the claimed inequality holds trivially, so assume that $h \ge 4$. In particular, this implies that $p_{i-1} > q_{h-1}$, by the definition of $h(p)$. By Lemma~\ref{lem:jumping:p}, it follows that
$$p_i \ge p_{i-1} - \alpha_h > q_{h-1} - \big( q_h - q_{h - 1} \big) \ge q_{h-3},$$  
as claimed, 
where the last inequality holds by the definition~\eqref{def:height} of $q_h$. It also follows immediately from Lemma~\ref{lem:jumping:p} that $h(p_i) \ge h(p_{i-1})$ if $i \in \cS \cup \cD$.

If $i \in \cB$, then set $h = h(p_{i-2})$ and assume that $h > 2\eps^{-1/2} + 1$ (and thus $p_{i-2} > q_{h-1}$), since otherwise the claimed inequality holds trivially. By Lemma~\ref{lem:jumping:p} and~\eqref{def:alpha}, we have
$$p_i \ge p_{i-2} - \eps^{-1/2} \alpha_h > q_{h-1} - \frac{\eps^{1/2}(1+\eps)^{h-1}}{k}.$$
Recalling~\eqref{def:height}, and noting that $1 - \eps^{1/2} \ge (1+\eps)^{- 2\eps^{-1/2}}$, it follows that 
\begin{equation}\label{eq:Step1:bound}
p_i > p_0 + \frac{(1 - \eps^{1/2})(1+\eps)^{h - 1} - 1}{k} \ge p_0 + \frac{(1+\eps)^{h - 2\eps^{-1/2} - 1} - 1}{k} = q_{h-2\eps^{-1/2} - 1},
\end{equation}
as required.
\end{proof}

We can now prove our claimed lower bound on $p$. 

\begin{proof}[Proof of Lemma~\ref{lem:bounding:p}]
Note first that it suffices to prove the result for $j \not\in \cD$, so we may assume that $j$ is even. 
Suppose that $p_j < p_0 = q_0$, and let $0 \le j' < j$ be maximal such that $p_{j'} \ge p_0$ and $j' \not\in \cD$. We claim that by Lemma~\ref{lem:Z:decrease}, 
\begin{equation}\label{eq:bounding:p:first}
p_j \ge p_{j'+2} - \sum_{i \in \cZ} \big( p_{i-2} - p_{i} \big) \ge p_{j'+2} - 2\eps.
\end{equation}
Indeed, for the first inequality observe that if $p_i < p_{i-2}$ for some $i \not\in \cD$ with $j' + 2 < i \le j$, then $p_{i-2} \le p_0$ (by the maximality of $j'$), and therefore $i \in \cZ$.


Now, since $p_{j'+1} \ge p_{j'} \ge p_0$, it follows from Lemma~\ref{lem:jumping:p} and Observation~\ref{obs:alpha:bounds} that
$$p_{j'+2} \ge \min_{p \ge p_0} \big\{ p - \eps^{-1/2} \alpha_{h(p)} \big\} \ge q_0 - \frac{\eps^{1/2}}{k} \ge p_0 - \eps.$$
Combining this with~\eqref{eq:bounding:p:first}, we obtain $p_j \ge p_{j'+2} - 2\eps \ge p_0 - 3\eps$, as required.
\end{proof}

The claimed bound on the size of $Y$ now follows easily. 

\begin{proof}[Proof of Lemma~\ref{lem:Ybound}]
Observe that the set $Y$ only changes in a red or density-boost step, in which case $Y$ is replaced by $N_R(x_i) \cap Y$, where $x_i$ is the central vertex. Thus $|Y_i| = |Y_{i-1}|$ for all $i \in \cB \cup \cD$, and if $i \in \cR \cup \cS$, then
$$|Y_i| = |N_R(x_i) \cap Y_{i-1}| \ge \big( 1 - \eps^{1/2} \big) p_{i-2} |Y_{i-1}| \ge \big( p_0 - 2\eps^{1/2} \big) |Y_{i-1}|$$
where the first inequality follows by applying Observation~\ref{obs:red:degree:lower:bound} to step $i - 1 \in \cD$ and $x = x_i \in X_{i-1}$, and the second holds by Lemma~\ref{lem:bounding:p}. It follows that
$$\frac{|Y_m|}{|Y_0|} \, = \, \prod_{i = 1}^m \frac{|Y_i|}{|Y_{i-1}|} \ge \big( p_0 - 2\eps^{1/2} \big)^{s+t} = 2^{o(k)} p_0^{s+t},$$
since $p_0 \ge 1/2$ 
and $s+t \le k+\ell \le 2k$, by Observation~\ref{obs:RBSD:sizes}. 
\end{proof}

\section{Bounding the size of $X$}\label{X:sec}

In this section we will prove a lower bound on the size of the set $X$ at the end of the algorithm in terms of $t$ and $s$, the number of red and density-boost steps respectively. In order to state our bound, we need a new parameter $\beta \in (0,1)$ of the process, defined by\footnote{Recall that $\beta_i$ was defined in~\eqref{def:beta:is}. If $\cS^* = \emptyset$ then we set $\beta := \mu$.}
\begin{equation}\label{def:beta}
\frac{1}{\beta} = \frac{1}{|\cS^*|} \sum_{i \in \cS^*} \frac{1}{\beta_i},
\end{equation}
where $\cS^* \subset \cS$ is the set of `moderate' density-boost steps
\begin{equation}\label{def:Sstar}
\cS^* = \big\{ i \in \cS : h(p_i) - h(p_{i-1}) \le \eps^{-1/4} \big\}.
\end{equation}
The main goal of this section is to prove the following lower bound on the size of $X$. 

\begin{lemma}\label{lem:Xbound}
\begin{equation}\label{eq:Xbound}
|X| \ge 2^{o(k)} \mu^\ell (1 - \mu)^t \bigg( \frac{\beta}{\mu} \bigg)^s |X_0|.
\end{equation}
\end{lemma}

In order to understand the right-hand side of~\eqref{eq:Xbound}, consider how the size of the set $X$ changes in each of the different types of step in the algorithm. The factor of $(1 - \mu)^t$ comes from the $t$ red steps, since the central vertex always has at most $\mu |X|$ blue neighbours in $X$; the factor of $\mu^{\ell - s}$ comes from the big blue steps (note that at most $\ell - s$ vertices are added to $B$ in big blue steps); the factor of  $\beta^s$ comes from the moderate density-boost steps, via the AM-GM inequality; and the factor of $2^{o(k)}$ comes from the degree regularisation steps, the `immoderate' density-boost steps, and other minor losses in the calculation. We remark that we will later, in Section~\ref{zigzag:sec}, bound $s$ in terms of $\beta$, which will prevent the factor of $(\beta/\mu)^s$ from hurting us too much. 

We begin with the red steps, for which the proof is especially simple. 

\subsection{Red steps}

Bounding the effect of red steps on the size of $X$ is easy, since Step~\ref{Alg:Step3} of the algorithm guarantees that the central vertex $x$ has at most $\mu |X|$ blue neighbours in $X$. 

\begin{lemma}\label{lem:red:total:decrease}
$$\prod_{i \in \cR} \frac{|X_i|}{|X_{i-1}|} \ge 2^{o(k)} (1 - \mu)^t,$$
where $t = |\cR|$. 
\end{lemma}

\begin{proof}
For each $i \in \cR$, we have 
$$X_i = N_R(x_i) \cap X_{i-1} \qquad \text{and} \qquad |N_R(x_i) \cap X_{i-1}| \ge (1 - \mu) |X_{i-1}| - 1.$$
Since $t \le k$ (by Observation~\ref{obs:RBSD:sizes}), $\mu \le 1/2$ is constant, and $|X_{i-1}| \ge R(k,\ell^{3/4})$ (otherwise the algorithm would have stopped), the claimed bound follows. 
\end{proof}

\subsection{Big blue steps}

We next use our bound on the number of big blue steps, Lemma~\ref{lem:few:big:blue}, to bound the decrease in the size of $X$ during big blue steps.

\begin{lemma}\label{lem:bigblue:total:decrease}
$$\prod_{i \in \cB} \frac{|X_i|}{|X_{i-1}|} \ge 2^{o(k)} \mu^{\ell-s}.$$
\end{lemma}

\begin{proof}
For each $i \in \cB$, let $b_i$ be the number of vertices added to $B$ in the corresponding big blue step. Recall that $|B| < \ell$, since $\chi$ contains no blue $K_\ell$, and that $s$ vertices are added to $B$ during density-boost steps. It follows that 
$$\sum_{i \in \cB} b_i \le \ell - s \qquad \text{and} \qquad \prod_{i \in \cB} \frac{|X_i|}{|X_{i-1}|} \ge \prod_{i \in \cB} \frac{\mu^{b_i}}{2} \ge 2^{-|\cB|} \mu^{\ell-s}.$$
Since $|\cB| \le \ell^{3/4} = o(k)$, by Lemma~\ref{lem:few:big:blue}, the claimed bound follows.
\end{proof}

\subsection{Density-boost steps}

In order to bound the decrease in the size of $X$ due to density-boost steps we will have to work a little harder. 

\begin{lemma}\label{lem:densityboost:total:decrease}
$$\prod_{i \in \cS} \frac{|X_i|}{|X_{i-1}|} \ge 2^{o(k)} \beta^s.$$
\end{lemma}

The lemma follows from a simple application of the AM-GM inequality, together with the weak lower bound on $\beta_i$ given by Lemma~\ref{lem:density:boost:weak}, and the following bound on the number of immoderate density-boost steps.  

\begin{lemma}\label{lem:few-big-density-jumps} 
\begin{equation}\label{eq:few-big-density-jumps}
|\cS \setminus \cS^*| \le 3 \eps^{1/4} k.
\end{equation}
\end{lemma}

\begin{proof}
Observe first that
\begin{equation}\label{eq:sum:of:h:upper}
\sum_{i = 1}^m \big( h(p_i) - h(p_{i-1}) \big) = h(p_m) - h(p_0) \le \frac{2\log k}{\eps} = o(k),
\end{equation}
since $1 \le h(p) \le 2\eps^{-1} \log k$ for all $p \in (0,1)$, and $\eps = k^{-1/4}$. Now, note that 
\begin{equation}\label{eq:sum:of:h:notD}
\sum_{i = 1}^m \big( h(p_i) - h(p_{i-1}) \big) \ge \sum_{i \in [m] \setminus \cD}  \big( h(p_{i}) - h(p_{i-2}) \big),
\end{equation}
since the steps of the algorithm alternate between the sets $\cD$ and $\cR \cup \cB \cup \cS$, and $p_i \ge p_{i-1}$ for all $i \in \cD$. Moreover, by Lemma~\ref{lem:jumping:down} and the definition~\eqref{def:Sstar} of $\cS^*$, we have
\begin{equation}\label{eq:sum:of:h:RS}
\sum_{i \in \cR \cup \cS} \big( h(p_i) - h(p_{i-2}) \big) \ge \eps^{-1/4} \cdot |\cS \setminus \cS^*| - 2t,
\end{equation}
and by Lemmas~\ref{lem:few:big:blue} and~\ref{lem:jumping:down}, we have
\begin{equation}\label{eq:sum:of:h:B}
\sum_{i \in \cB} \big( h(p_i) - h(p_{i-2}) \big) \ge - 2 \eps^{-1/2} \cdot |\cB| \ge - 2k^{7/8} = o(k),
\end{equation}
since $\eps = k^{-1/4}$ and $|\cB| \le \ell^{3/4} \le k^{3/4}$. Combining the inequalities above, it follows that 
$$\sum_{i = 1}^m \big( h(p_i) - h(p_{i-1}) \big) \ge \eps^{-1/4} \cdot |\cS \setminus \cS^*| - 2t - o(k),$$ 
and together with~\eqref{eq:sum:of:h:upper} and the bound $t \le k$, this implies~\eqref{eq:few-big-density-jumps}. 
\end{proof}

We can now bound the decrease in the size of $X$ during density-boost steps.  

\begin{proof}[Proof of Lemma~\ref{lem:densityboost:total:decrease}]
Recall from Step~\ref{Alg:Step5} of the algorithm and~\eqref{def:beta:is} that if $i \in \cS$, then 
$$X_i = N_B(x_i) \cap X_{i-1} \qquad \text{and} \qquad |N_B(x_i) \cap X_{i-1}| = \beta_i |X_{i-1}|,$$
and therefore
$$\prod_{i \in \cS} \frac{|X_i|}{|X_{i-1}|} = \prod_{i \in \cS} \beta_i.$$
To deal with the immoderate steps, we use Lemmas~\ref{lem:density:boost:weak} and~\ref{lem:few-big-density-jumps} to deduce that
$$\prod_{i \in \cS \setminus \cS^*} \frac{1}{\beta_i} \le k^{2 |\cS \setminus \cS^*|} \le \exp\Big( 6 \eps^{1/4} k \log k \Big) = 2^{o(k)},$$
where $\cS^*$ was defined in~\eqref{def:Sstar}, and in the last step we used the fact that $\eps = k^{-1/4}$. 

For the moderate steps, recall from~\eqref{def:beta} the definition of $\beta$, and observe that
$$\prod_{i \in \cS^*} \frac{1}{\beta_i} \, \le \bigg( \frac{1}{|\cS^*|} \sum_{i \in \cS^*} \frac{1}{\beta_i} \bigg)^{|\cS^*|} = \, \beta^{-|\cS^*|},$$
by the AM-GM inequality. Combining these bounds, and recalling that $|\cS^*| \le s$, we obtain
$$\prod_{i \in \cS} \beta_i \ge 2^{o(k)} \beta^{|\cS^*|} \ge 2^{o(k)} \beta^s,$$
as required.
\end{proof}

\subsection{Degree regularisation steps}\label{regularisation:step}

The final step in the proof of Lemma~\ref{lem:Xbound} is to show that $|X|$ does not decrease by a significant amount during degree regularisation steps. 

\begin{lemma}\label{lem:degree:total:decrease}
$$\prod_{i \in \cD} \frac{|X_i|}{|X_{i-1}|} = 2^{o(k)}.$$
\end{lemma}

In order to prove Lemma~\ref{lem:degree:total:decrease}, we will consider separately those degree regularisation steps that occur before big blue steps, and those that occur before red or density-boost steps. It will be straightforward to bound the total effect of the former, since by Lemma~\ref{lem:few:big:blue} there are few big blue steps. Bounding the effect of the remaining steps is also not difficult, but will require a little more work. We begin with the following simple observation, which bounds the change in $p$ in a degree regularisation step in terms of the change in the size of $X$.

\begin{obs}\label{obs:degree:boost} 
If\/ $i \in \cD$, then 
$$p_i - p_{i-1} \ge \frac{|X_{i-1} \setminus X_i|}{|X_i|} \cdot \eps^{-1/2} \alpha_{h(p_{i-1})}.$$
\end{obs}

\begin{proof}
Set $p = p_{i-1}$ and $\xi = \eps^{-1/2} \alpha_{h(p)}$, and recall from Step~\ref{Alg:Step1} of the algorithm that 
$$X_i = \big\{ x \in X_{i-1} : |N_R(x) \cap Y| \ge (p - \xi) |Y| \big\},$$
where $Y = Y_{i-1} = Y_i$. Observe that 
$$e_R(X_{i-1} \setminus X_i,Y) \le |X_{i-1} \setminus X_i| \cdot (p - \xi) |Y| ,$$
and therefore, since $e_R(X_{i-1},Y) = p |X_{i-1}| |Y|$, that 
$$e_R(X_i,Y) \ge p |X_i||Y| + |X_{i-1} \setminus X_i| \cdot \xi \hspace{0.04cm} |Y|.$$
It follows that the density of red edges between $X_i$ and $Y$ is at least
$$p + \frac{|X_{i-1} \setminus X_i|}{|X_i|} \cdot \xi,$$ 
as claimed. 
\end{proof}

We can now bound the maximum decrease in $X$ during a degree regularisation step. 

\begin{lemma}\label{lem:degree:onestep:decrease}
If\/ $i \in \cD$, then 
$$|X_i| \ge \frac{|X_{i-1}|}{k^2}.$$
\end{lemma}

\begin{proof}
Note that $\eps^{-1/2} \cdot \alpha_{h(p_{i-1})} \ge 2/k^2$, by Observation~\ref{obs:alpha:bounds}. Since $p_{i-1} \ge 0$ and $p_i \le 1$, it follows from Observation~\ref{obs:degree:boost} that 
$$|X_i| \ge \frac{2}{k^2} \cdot |X_{i-1} \setminus X_i|.$$ 
This implies that $|X_i| \ge |X_{i-1}| / k^2$, as claimed.
\end{proof}

Using Lemmas~\ref{lem:few:big:blue} and~\ref{lem:degree:onestep:decrease}, it will be straightforward to bound the decrease in $|X|$ due to degree regularisation steps that are followed by big blue steps. For those that are followed by red or density-boost steps, we will use the following lemma.

\begin{lemma}\label{lem:moderate:degree:reg}
Let $i \in \cD$, and suppose that $p_{i-1} \ge p_0$ and $h(p_i) \le h(p_{i-1}) + \eps^{-1/4}$. Then
$$|X_i| \ge \big( 1 - 2 \eps^{1/4} \big) |X_{i-1}|.$$
\end{lemma}

\begin{proof}
We claim that
\begin{equation}\label{eq:degree:decrease:step}
\frac{|X_{i-1} \setminus X_i|}{|X_i|} \cdot \eps^{-1/2} \alpha_{h(p_{i-1})} \, \le \, p_i - p_{i-1} \, \le \, 2\eps^{-1/4} \alpha_{h(p_{i-1})}.
\end{equation}
Indeed, the lower bound follows from Observation~\ref{obs:degree:boost}, while the upper bound holds since $p_i \le q_{h(p_i)}$ and $p_{i-1} \ge q_{h(p_{i-1})-1}$, which by~\eqref{def:height} and~\eqref{def:alpha} implies that 
$$p_i - p_{i-1} \le q_{h(p_i)} - q_{h(p_{i-1})-1} = \frac{\alpha_{h(p_{i-1})}}{\eps} \Big( (1+\eps)^{h(p_i)- h(p_{i-1})+1} - 1 \Big) \le 2\eps^{-1/4} \alpha_{h(p_{i-1})},$$
as claimed, since $h(p_i) - h(p_{i-1}) \le \eps^{-1/4}$ and $(1+x)^{x^{-1/4} + 1} \le 1 + 2x^{3/4}$ for all sufficiently small $x \ge 0$. It follows from~\eqref{eq:degree:decrease:step} that 
$$|X_{i-1} \setminus X_i| \le 2 \eps^{1/4} \cdot |X_i| \le 2 \eps^{1/4} \cdot |X_{i-1}|,$$
as required.
\end{proof}

We next prove two variants of Lemma~\ref{lem:few-big-density-jumps}, which we will use to bound the number of steps with $i+1 \in \cR \cup \cS$ such that the conditions of Lemma~\ref{lem:moderate:degree:reg} do not hold. 

\begin{lemma}\label{lem:few-big-regularisation-jumps} 
$$\big|\big\{ i \in \cR \cup \cS : h(p_{i-1}) \ge h(p_{i-2}) + \eps^{-1/4} \big\} \big| \le 3 \eps^{1/4} k.$$
\end{lemma}

\begin{proof}
Recall from~\eqref{eq:sum:of:h:upper},~\eqref{eq:sum:of:h:notD} and~\eqref{eq:sum:of:h:B}, that
$$\sum_{i \in \cR \cup \cS} \big( h(p_i) - h(p_{i-2}) \big) \le o(k).$$
Moreover, by Lemma~\ref{lem:jumping:down}, we have (cf.~\eqref{eq:sum:of:h:RS}) 
$$\sum_{i \in \cR \cup \cS} \big( h(p_i) - h(p_{i-2}) \big) \ge \eps^{-1/4} \cdot \big|\big\{ i \in \cR \cup \cS : h(p_{i-1}) \ge h(p_{i-2}) + \eps^{-1/4} \big\} \big| - 2t.$$ 
Combining these inequalities, and recalling that $t \le k$, the claimed bound follows.  
\end{proof}

We will use the following variant in order to bound increases that begin below $p_0$.

\begin{lemma}\label{lem:few-big-jumps:below:p0} 
$$\big|\big\{ i \in \cR \cup \cS \,:\, p_{i-1} \ge p_{i-2} + \eps^{-1/4} \alpha_1 \,\textup{ and }\, p_{i-2} \le p_0 \big\} \big| \le 4 \eps^{1/4} k.$$
\end{lemma}

\begin{proof}
We will bound the total decrease in $p$ that occurs below $q^* := p_0 + \eps^{-1/4} \alpha_1$ during the algorithm. To do so, set $p_i^* := \min\{ p_i,q^*\}$ for each $i \in [m]$, and observe that $h(q^*) \le \eps^{-1/4}$. By Lemmas~\ref{lem:jumping:p} and~\ref{lem:jumping:down}, the definition~\eqref{def:alpha} of $\alpha_h$, and the bound $t \le k$, it follows that 
$$\sum_{i \in \cR} \big( p_i^* - p_{i-1}^* \big) \ge - \alpha_{h(q^*) + 2} \cdot t \ge - 2\eps,$$
since if $h(p_{i-1}) > h(q^*) + 2$, then $p_i^* = p_{i-1}^* = q^*$. Similarly,~by Lemmas~\ref{lem:jumping:p} and~\ref{lem:jumping:down}, the definition of $\alpha_h$, and Lemma~\ref{lem:few:big:blue}, we have
$$\sum_{i \in \cB} \big( p_i^* - p_{i-2}^* \big) \ge - \eps^{-1/2} \alpha_{h(q^*) + 2\eps^{-1/2}} \cdot |\cB| \ge - \frac{2\eps^{1/2}}{k^{1/4}} \ge -\eps,$$
since if $h(p_{i-2}) > h(q^*) + 2\eps^{-1/2}$, then $p_i^* = p_{i-2}^* = q^*$. 

Now, since $p_i^* \ge p_{i-1}^*$ for every $i \in \cS \cup \cD$, by Lemma~\ref{lem:jumping:p}, it follows that
$$q^* - p_0 \,\ge\, p_m^* - p_0^* \,\ge\, \sum_{i \in \cR \cup \cB \cup \cS} \big( p_{i}^* - p_{i-2}^* \big) \ge \sum_{i \in \cR \cup \cS} \big( p_{i-1}^* - p_{i-2}^* \big) - 3\eps.$$ 
Moreover, since $p_i^* \ge p_{i-1}^*$ for every $i \in \cD$, 
we have
$$\sum_{i \in \cR \cup \cS} \big( p_{i-1}^* - p_{i-2}^* \big) \ge \eps^{-1/4} \alpha_1 \cdot \big|\big\{ i \in \cR \cup \cS \,:\, p_{i-1} \ge p_{i-2} + \eps^{-1/4} \alpha_1 \,\textup{ and }\, p_{i-2} \le p_0 \big\} \big|,$$
by our choice of $q^*$, and since if $i \in \cR \cup \cS$, then $i - 1 \in \cD$. Since $\alpha_1 = \eps/k$, it follows that
$$\big|\big\{ i \in \cR \cup \cS \,:\, p_{i-1} \ge p_{i-2} + \eps^{-1/4} \alpha_1 \,\textup{ and }\, p_{i-2} \le p_0 \big\} \big| \le \frac{q^* - p_0 + 3\eps}{\eps^{-1/4} \alpha_1} \le 4\eps^{1/4} k,$$
as claimed.
\end{proof}

Combining Lemmas~\ref{lem:few-big-regularisation-jumps} and~\ref{lem:few-big-jumps:below:p0}, we obtain the following bound on the number of `immoderate' degree regularisation steps that are followed by red or density-boost steps. 

\begin{lemma}\label{lem:few:immoderate:degree:reg:steps}
$$\big| \big\{ i \in \cR \cup \cS \,:\, |X_{i-1}| < \big( 1 - 2 \eps^{1/4} \big) |X_{i-2}| \big\} \big| \le 7\eps^{1/4} k.$$
\end{lemma}

\begin{proof}
Let $i \in \cR \cup \cS$, and observe that if $|X_{i-1}| < \big( 1 - 2 \eps^{1/4} \big) |X_{i-2}|$, then
$$p_{i-1} - p_{i-2} \ge \frac{|X_{i-2} \setminus X_{i-1}|}{|X_{i-1}|} \cdot \eps^{-1/2} \alpha_1 \ge \eps^{-1/4} \alpha_1,$$
by Observation~\ref{obs:degree:boost} and~\eqref{def:alpha}. It follows, by Lemma~\ref{lem:few-big-jumps:below:p0}, that there are at most $4 \eps^{1/4} k$ such steps with $p_{i-2} \le p_0$. On the other hand, by Lemmas~\ref{lem:moderate:degree:reg} and~\ref{lem:few-big-regularisation-jumps} there are at most $3 \eps^{1/4} k$ steps $i \in \cR \cup \cS$ with $p_{i-2} \ge p_0$ and $|X_{i-1}| < \big( 1 - 2 \eps^{1/4} \big) |X_{i-2}|$, so the lemma follows.
\end{proof}

\pagebreak

We can now bound the decrease in the size of $X$ caused by degree regularisation steps. 

\begin{proof}[Proof of Lemma~\ref{lem:degree:total:decrease}]
Recall that each degree regularisation step (except possibly the last) is followed by either a big blue step, a red step, or a density-boost step. It follows, by Lemmas~\ref{lem:few:big:blue} and~\ref{lem:few:immoderate:degree:reg:steps}, that there are at most 
$$7 \eps^{1/4} k + k^{3/4} + 1$$ 
steps $i \in \cD$ such that $|X_i| < \big( 1 - 2 \eps^{1/4} \big) |X_{i-1}|$. Since $|\cD| \le k+\ell+1$, by Observation~\ref{obs:RBSD:sizes}, and $|X_i| \ge |X_{i-1}| / k^2$ for every $i \in \cD$, by Lemma~\ref{lem:degree:onestep:decrease}, it follows that
$$\prod_{i \in \cD} \frac{|X_i|}{|X_{i-1}|} \ge \bigg( \frac{1}{k^2} \bigg)^{7 \eps^{1/4} k + k^{3/4} + 1} \big( 1 - 2 \eps^{1/4} \big)^{k+\ell+1} = 2^{o(k)},$$
as required. 
\end{proof}

\subsection{Bounding the size of $X$}\label{Xbound:proof}

Combining Lemmas~\ref{lem:red:total:decrease},~\ref{lem:bigblue:total:decrease},~\ref{lem:densityboost:total:decrease} and~\ref{lem:degree:total:decrease}, we obtain the claimed lower bound on the size of $X$. 

\begin{proof}[Proof of Lemma~\ref{lem:Xbound}]
By Lemmas~\ref{lem:red:total:decrease} and~\ref{lem:bigblue:total:decrease}, we have 
$$\prod_{i \in \cR} \frac{|X_i|}{|X_{i-1}|} \ge 2^{o(k)} (1 - \mu)^t \qquad \text{and} \qquad \prod_{i \in \cB} \frac{|X_i|}{|X_{i-1}|} \ge 2^{o(k)} \mu^{\ell - s},$$
and by Lemmas~\ref{lem:densityboost:total:decrease} and~\ref{lem:degree:total:decrease}, we have 
$$\prod_{i \in \cS} \frac{|X_i|}{|X_{i-1}|} \ge 2^{o(k)} \beta^s \qquad \text{and} \qquad \prod_{i \in \cD} \frac{|X_i|}{|X_{i-1}|} = 2^{o(k)}.$$
Multiplying these four inequalities, and recalling that 
$\cR \cup \cB \cup \cS \cup \cD = [m]$, it follows that
$$\frac{|X_m|}{|X_0|} \,=\, \prod_{i = 1}^m \frac{|X_i|}{|X_{i-1}|} \, \ge \, 2^{o(k)} \mu^{\ell-s} (1 - \mu)^t \beta^s \, = \, 2^{o(k)} \mu^\ell (1 - \mu)^t \bigg( \frac{\beta}{\mu} \bigg)^s,$$
as required. 
\end{proof}

\section{The Zigzag Lemma}\label{zigzag:sec}

The aim of this section is to bound the number of density-boost steps, and also their total contribution to the decrease in the size of $X$ during the algorithm. Recall from~\eqref{def:Sstar} that 
$$\cS^* = \big\{ i \in \cS : h(p_i) - h(p_{i-1}) \le \eps^{-1/4} \big\},$$ 
and that if $i \in \cS$, then $x_i$ is the central vertex of the corresponding density-boost step, and
$$|N_B(x_i) \cap X_{i-1}| = \beta_i |X_{i-1}|.$$

The following `zigzag' lemma is one of the key steps in the proof of Theorem~\ref{thm:diagonal}.

\pagebreak

\begin{lemma}[The Zigzag Lemma]\label{lem:zigzag}
\begin{equation}\label{eq:zigzag:stronger}
\sum_{i \in \cS^*} \frac{1 - \beta_i}{\beta_i} \le t + k^{1-c}
\end{equation}
for some constant $c > 0$, and all sufficiently large $k \in \N$.
\end{lemma}

The proof of this lemma is a little technical, but the underlying idea is based on a very simple analysis of how the density ``zigzags'' up and down, throughout the course of the process. Say that at some stage $i$ we have density $p_i \in [q_h,q_{h+1}]$ and then we have a (moderate) density boost of $\alpha_h \cdot (1-\beta_i)/\beta_i$. We now ask: how many red steps would it take for the density to return to the interval $[q_h,q_{h+1}]$? (Of course, there are also big blue steps that can bring the density down, but there are few enough of these to absorb into the error term.) The answer is approximately $(1-\beta_i)/\beta_i$, since $\alpha_h$ remains almost unchanged for these moderate steps. Thus seeing a density boost of $\alpha_h \cdot (1-\beta_i)/\beta_i$ ensures there are $(1-\beta_i)/\beta_i$ corresponding red steps, in this case. 

Of course, it is possible that the density \emph{never} returns to the interval $[q_h,q_{h+1}]$, in which case the above reasoning fails. However, in this case, we can assume that $p_j > q_{h+1}$ for the entire remainder of the process. Since there are only $k^{1-c}$ such intervals, this ``loss'' in the analysis in negligible and absorbed into the error.

\begin{proof}[Proof of Lemma~\ref{lem:zigzag}]
For each $i \in [m]$, set $p_i(1) = \min\{ p_i, q_1 \}$, and define  
\begin{equation}\label{def:pih}
    p_i(h) \,=\, \left\{
    \begin{array} {c@{\quad \textup{if} \quad}l}
      q_{h-1} & p_i \le q_{h-1}, \\[+1ex]
      p_i & p_i \in (q_{h-1},q_h), \\[+1ex]
      q_{h} & p_i \ge q_h,
    \end{array}\right.
\end{equation}
for each integer $h \ge 2$, and
\begin{equation}\label{def:Delta}
\Delta_i = p_i - p_{i-1} \qquad \text{and} \qquad \Delta_i(h) = p_i(h) - p_{i-1}(h)
\end{equation} 
for each 
$h \ge 1$. Note that, for each $i \in [m]$, we have
\begin{equation}\label{eq:Delta:is:sum:of:Deltas} 
\Delta_i = \sum_{h \ge 1} \Delta_i(h);
\end{equation} 
and $\Delta_i \ge 0$ if and only if $\Delta_i(h) \ge 0$ for every $h \ge 1$. We claim first that
\begin{equation} \label{eq:zigzag:total-sum}  
\sum_{h \ge 1} \sum_{i = 1}^m \frac{\Delta_i(h)}{\alpha_h} \le \frac{2\log k}{\eps}.
\end{equation}
To see this, observe that by~\eqref{def:pih} and~\eqref{def:Delta}, we have
$$\sum_{i = 1}^m \Delta_i(h) \le q_h - q_{h-1} = \alpha_h,$$
for all $h \ge 1$. Moreover, if $h > h(p_i)$ for all $i \in [m]$, then $\sum_{i = 1}^m \Delta_i(h) = 0$. Recalling that $h(p) \le 2\eps^{-1} \log k$ for all $p \le 1$, we obtain~\eqref{eq:zigzag:total-sum}. 

We now consider the contributions to the sum on the left-hand side of~\eqref{eq:zigzag:total-sum} of the various different types of step. We begin with the density-boost steps.

\begin{claim}\label{claim:S:sum}
$$\sum_{h \ge 1} \sum_{i \in \cS} \frac{\Delta_i(h)}{\alpha_h} \ge \big( 1 - \eps^{1/2} \big) \sum_{i \in \cS^*} \frac{1-\beta_i}{\beta_i}.$$
\end{claim}

\begin{clmproof}{claim:S:sum}
Observe that $\Delta_i(h) \ge 0$ for every $i \in \cS$ and every $h \ge 1$, by Lemma~\ref{lem:density:boost:weak}. It will therefore suffice to show that 
\begin{equation}\label{eq:blue:single} 
\sum_{h \ge 1} \frac{\Delta_i(h)}{\alpha_h} \ge \big( 1 - \eps^{1/2} \big) \bigg( \frac{1-\beta_i}{\beta_i} \bigg)
\end{equation} 
for every $i \in \cS^*$. To prove~\eqref{eq:blue:single}, recall that if $i \in \cS^*$, then $h(p_i) - h(p_{i-1}) \le \eps^{-1/4}$, that
\begin{equation}\label{eq:applying:density:boost:steps} 
\frac{\Delta_i}{\alpha_{h(p_{i-1})}} \ge (1 - \eps) \bigg( \frac{1-\beta_i}{\beta_i} \bigg), 
\end{equation} 
by Lemma~\ref{lem:density:boost:steps:boost:the:density:notation}, and that $\Delta_i(h) = 0$ for every $h > h(p_i) \ge h(p_{i-1})$. By~\eqref{eq:Delta:is:sum:of:Deltas}, it follows that
$$\sum_{h \ge 1} \frac{\Delta_i(h)}{\alpha_h} \ge \frac{\Delta_i}{\alpha_{h(p_i)}} = (1+\eps)^{h(p_{i-1}) - h(p_i)} \cdot \frac{\Delta_i}{\alpha_{h(p_{i-1})}} \ge \big( 1 - \eps^{1/2} \big) \bigg( \frac{1-\beta_i}{\beta_i} \bigg),$$
as claimed, where in the first two steps we used the fact that $\alpha_h = \eps (1+\eps)^{h-1} / k$ for every $h \ge 1$, and in the last step we used~\eqref{eq:applying:density:boost:steps} and the bound $h(p_{i-1}) - h(p_i) \ge - \eps^{-1/4}$.   
\end{clmproof}

We will next deal with the red steps, where the argument is even simpler. 

\begin{claim}\label{claim:red:sum}
\begin{equation}\label{eq:red:sum} 
\sum_{h \ge 1} \sum_{i \in \cR} \frac{\Delta_i(h)}{\alpha_h} \ge - (1 +\eps)^2 t.
\end{equation}
\end{claim}

\begin{clmproof}{claim:red:sum}
Let $i \in \cR$, and suppose that $\Delta_i < 0$, so $\Delta_i(h) \le 0$ for every $h \ge 1$. Observe that $h(p_i) \ge h(p_{i-1}) - 2$, by Lemma~\ref{lem:jumping:down}, and therefore $\Delta_i(h) = 0$ for every $h < h(p_{i-1}) - 2$. By~\eqref{def:alpha} and~\eqref{eq:Delta:is:sum:of:Deltas}, it follows that
$$\sum_{h \ge 1} \frac{\Delta_i(h)}{\alpha_h} \ge (1+\eps)^2 \cdot \frac{\Delta_i}{\alpha_{h(p_{i-1})}} \ge -(1+\eps)^2,$$
since $\Delta_i \ge - \alpha_{h(p_{i-1})}$, by Lemma~\ref{lem:jumping:p}. 
Summing over $i \in \cR$, and noting that the same bound holds when $\Delta_i \ge 0$, we obtain~\eqref{eq:red:sum}. 
\end{clmproof}

It remains to deal with the big blue and degree regularisation steps. We deal with these together, since Lemma~\ref{lem:jumping:down} only allows us to bound the change in height of a big blue step together with the preceding degree regularisation step (and, in fact, if $\Delta_i$ is large for some $i \in \cD$, then $\Delta_{i+1}$ might be very large and negative if $i+1 \in \cB$). 

\pagebreak

\begin{claim}\label{claim:bigblue:sum}
\begin{equation}\label{eq:zigzag:bigblue:steps} 
\sum_{h \ge 1} \sum_{i \in \cB \cup \cD} \frac{\Delta_i(h)}{\alpha_h} \ge - 2 k^{7/8}.
\end{equation} 
\end{claim}

\begin{clmproof}{claim:bigblue:sum}
Note first that $\Delta_i \ge 0$ for every $i \in \cD$, so we can ignore every degree regularisation step that is followed by a red or density-boost step. Recalling that $|\cB| \le \ell^{3/4}$, by Lemma~\ref{lem:few:big:blue}, and that $\eps = k^{-1/4}$, it will therefore suffice to show that 
\begin{equation}\label{eq:zigzag:one:bigblue:step} 
\sum_{h \ge 1} \frac{\Delta_{i-1}(h) + \Delta_i(h)}{\alpha_h} \ge - 2\eps^{-1/2}
\end{equation} 
for every $i \in \cB$. Note that if $\Delta_{i-1} + \Delta_i \ge 0$, then every term of the sum is non-negative, and~\eqref{eq:zigzag:one:bigblue:step} holds trivially. We may therefore assume that $\Delta_{i-1}(h) + \Delta_i(h) \le 0$ for every $h \ge 1$. 

Now, recall that if $i \in \cB$, then
\begin{equation}\label{eq:twosteps:DthenB}
\Delta_{i-1} + \Delta_i \ge - \eps^{-1/2} \cdot \alpha_{h(p_{i-2})} \qquad \text{and} \qquad h(p_i) \ge h(p_{i-2}) - 2\eps^{-1/2},
\end{equation}
by 
Lemmas~\ref{lem:jumping:p} and~\ref{lem:jumping:down}, and therefore $\Delta_{i-1}(h) + \Delta_i(h) = 0$ for every $h < h(p_{i-2}) - 2\eps^{-1/2}$. By~\eqref{eq:Delta:is:sum:of:Deltas}, it follows that 
$$\sum_{h \ge 1} \frac{\Delta_{i-1}(h) + \Delta_i(h)}{\alpha_h} \ge (1 + \eps)^{2\eps^{-1/2}} \cdot \frac{\Delta_{i-1} + \Delta_i}{\alpha_{h(p_{i-2})}} \ge -2\eps^{-1/2},$$
as claimed, where in the first step we used the fact that $\alpha_h = \eps (1+\eps)^{h-1} / k$ for every $h \ge 1$, and in the second we used~\eqref{eq:twosteps:DthenB}. This proves~\eqref{eq:zigzag:one:bigblue:step}, and hence also the claim.
\end{clmproof}

Finally, combining~\eqref{eq:zigzag:total-sum} with Claims~\ref{claim:S:sum},~\ref{claim:red:sum} and~\ref{claim:bigblue:sum}, 
we obtain  
\begin{equation}\label{eq:zigzag:finalstep}
\frac{2\log k}{\eps} \ge \big( 1 - \eps^{1/2} \big) \sum_{i \in \cS^*} \frac{1-\beta_i}{\beta_i} - (1+\eps)^2 t - 2k^{7/8}.
\end{equation}
Since $\eps = k^{-1/4}$ and $t \le k$, the lemma follows. 
\end{proof}


We can now bound the number of density-boost steps in terms of $t$ and $\beta$. Recall that $s = |\cS|$ and that $\beta$ was defined in~\eqref{def:beta}. 

\begin{lemma}\label{lem:s:bound}
\begin{equation}\label{eq:s:bound}
s \le \bigg( \frac{\beta}{1 - \beta} \bigg) t + O\big( k^{1-c} \big)
\end{equation}
for some constant $c > 0$. 
\end{lemma}

\begin{proof}
Observe first that $|\cS \setminus \cS^*| \le 3\eps^{1/4} k \le k^{1-c}$, by Lemma~\ref{lem:few-big-density-jumps}, and since $\eps = k^{-1/4}$. By~\eqref{eq:zigzag:stronger}, it follows that either $\cS^* = \emptyset$ (in which case we are done), or
$$\frac{s - k^{1-c}}{\beta} \le \frac{|\cS^*|}{\beta} = \sum_{i \in \cS^*} \frac{1}{\beta_i} \le |\cS^*| + t + k^{1-c} \le s + t + k^{1-c}.$$

\pagebreak

\noindent Rearranging, it follows that
$$s \, \le \, \frac{\beta t + (1 + \beta)k^{1-c}}{1 - \beta} \, = \, \bigg( \frac{\beta}{1 - \beta} \bigg) t + 3k^{1-c},$$
as required, where in the last step we used the fact that $\beta \le \mu \le 1/2$.  
\end{proof}

We also note the following useful bound on $\beta$, which follows easily from Lemma~\ref{lem:s:bound}.

\begin{lemma}\label{lem:beta:bound}
There exists a constant $c > 0$ such that if $s \ge k^{1-c}$, then
$$\beta \ge \big(1 + o(1) \big) \bigg( \frac{s}{s + t} \bigg).$$ 
\end{lemma}

\begin{proof}
Let $c'$ be the constant in Lemma~\ref{lem:s:bound}, and set $c = c'/2$. It follows from~\eqref{eq:s:bound} that if $s \ge k^{1-c}$, then
$$\beta \ge \frac{s - O(k^{1-2c})}{s+t} \ge \big(1 - O(k^{-c}) \big) \bigg( \frac{s}{s + t} \bigg)$$
as required.
\end{proof}

\section{An exponential improvement far from the diagonal}\label{simple:sec}

In this section we prove our first exponential improvement for the Ramsey numbers $R(k,\ell)$. We will only obtain a bound when $\ell \le k/9$, which is not sufficient to deduce Theorem~\ref{thm:diagonal}, but the proof is simpler in this case, and so we hope that it will act as a useful warm-up for the reader. Moreover, the results of this section will be used again in Section~\ref{just:enough:sec}, below. The lower bounds that we require on $k$ and $\ell$ in this section depend only on $\gamma_0$. 

\begin{theorem}\label{thm:off:diagonal:weak}
Fix\/ $\gamma_0 > 0$, and let\/ $k,\ell \in \N$ be sufficiently large integers with $\gamma_0 \le \gamma \le 1/10$, where $\gamma = \frac{\ell}{k+\ell}$. Then 
$$R(k,\ell) \le e^{-\delta k} {k + \ell \choose \ell},$$
where $\delta = \gamma/20$. 
\end{theorem}



The idea is to apply the Book Algorithm with $\mu = \gamma$, and show that 
$$n \ge e^{-\delta k} {k + \ell \choose \ell} \qquad \Rightarrow \qquad |Y| \ge {k-t+\ell \choose \ell} \ge R(k-t,\ell)$$
at the end of the algorithm, where $t$ is the number of red steps. The main complication is that the initial density of blue edges may be significantly larger that $\gamma$, in which case $|Y|$ will shrink too fast if we apply the algorithm directly to this colouring. Instead, we prepare the colouring by first taking a sequence of blue steps (in the sense of Erd\H{o}s and Szekeres), gaining a small constant factor in each (and only moving further away from the diagonal), until we reach a set with suitable red density. 
We will then apply the following lemma to the colouring restricted to this set.


\pagebreak

\begin{lemma}\label{lem:off:diagonal:weak}
Fix\/ $\gamma_0 > 0$, and let\/ $k,\ell \in \N$ be sufficiently large integers with $\gamma_0 \le \gamma \le 1/10$, where $\gamma = \frac{\ell}{k+\ell}$.  
Let~$\delta = \gamma/20$ and $\eta \le \gamma/15$, and suppose that
\begin{equation}\label{eq:off:n:bound}
n \ge e^{-\delta k} {k + \ell \choose \ell}.
\end{equation}
Then every red-blue colouring of\/ $E(K_n)$ in which the density of red edges is\/ $1 - \gamma - \eta$ contains~either a red $K_k$ or a blue $K_\ell$. 
\end{lemma}



To prove Lemma~\ref{lem:off:diagonal:weak}, let $n \in \N$ satisfy~\eqref{eq:off:n:bound}, and let $\chi$ be a red-blue colouring of $E(K_n)$ that contains no red $K_k$ or blue $K_\ell$, and in which the density of red edges is $1 - \gamma - \eta$. Choose disjoint sets of vertices $X$ and $Y$ with $|X|,|Y| \ge \lfloor n/2 \rfloor$ and 
\begin{equation}\label{off:diag:weak:start:density}
\frac{e_R(X,Y)}{|X||Y|} \ge 1 - \gamma - \eta,
\end{equation}
and apply the Book Algorithm to the pair $(X,Y)$ with $\mu = \gamma$. Recall that we write $t$ and $s$  for the number of red and density-boost steps, respectively. We begin by bounding $t$ from below; the following rough bound will suffice for our purposes.\footnote{We will use Lemma~\ref{lem:t:big} again in Section~\ref{just:enough:sec}, and our choices of $\gamma$ and $\delta$ are partly motivated by this later application. We remark that any constant lower bound on $p_0$ would suffice.} 

\begin{lemma}\label{lem:t:big}
Fix\/ $\gamma_0 > 0$, and let\/ $k,\ell \in \N$ be sufficiently large integers with $\gamma_0 \le \gamma \le 1/5$, where $\gamma = \frac{\ell}{k+\ell}$. If\/ $p_0 \ge 1/2$, then 
$$n \ge e^{- \delta k} {k + \ell \choose \ell} \qquad \Rightarrow \qquad t \ge \frac{2k}{3},$$
where $\delta = \min\big\{ 1/200, \gamma/20 \big\}$.
\end{lemma}


In the proof of the lemma 
we will use the following fact, which is an easy consequence of Stirling's formula. 


\begin{fact}\label{fact:binomal:gammas}
Fix\/ $0 < \gamma \le 1/2$. Then
$${k + \ell \choose \ell} = 2^{o(k)} \gamma^{-\ell} (1 - \gamma)^{-k}$$
for every $k,\ell \in \N$ with $\gamma = \frac{\ell}{k+\ell}$.
\end{fact}

\begin{proof}[Proof of Lemma~\ref{lem:t:big}]
Observe that the algorithm continues until $|X| \le R(k,\ell^{3/4}) = 2^{o(k)}$, since we have $p \ge 1/4$ throughout the algorithm, by Lemma~\ref{lem:bounding:p}. In order to prove the lemma, it will therefore suffice to bound the size of the set $X$ from below in terms of $t$. By Lemma~\ref{lem:Xbound}, applied with $\mu = \gamma$, we have 
$$|X| \ge 2^{o(k)} \gamma^\ell (1 - \gamma)^t \bigg( \frac{\beta}{\gamma} \bigg)^s n.$$
By Fact~\ref{fact:binomal:gammas} and our choice of $n$, it follows that
$$e^{-\delta k} (1 - \gamma)^{-k+t} \bigg( \frac{\beta}{\gamma} \bigg)^s \le 2^{o(k)}.$$
Recall that $\beta \le \gamma$, and observe that therefore $s \le \beta t / (1 - \gamma) + O(k^{1-c})$, by Lemma~\ref{lem:s:bound}. Using the bound $(1 - \gamma)^{-1} \ge e^{\gamma}$, recalling that $t \le k$, and taking logs, it follows that
$$\gamma(k - t) \le \bigg( \frac{\beta t}{1 - \gamma} + O\big( k^{1-c} \big) \bigg) \log \frac{\gamma}{\beta} + \delta k + o(k),$$

\noindent and therefore, since $\beta \ge 1/k^2$, by Lemma~\ref{lem:density:boost:weak}, we obtain
\begin{equation}\label{eq:t:bound:with:beta}
\gamma(k - t) \le \bigg( \frac{\beta t}{1 - \gamma} \bigg) \log \frac{\gamma}{\beta} + \delta k + o(k). 
\end{equation}
Now, note that the right-hand side of~\eqref{eq:t:bound:with:beta} is maximised with $\beta = \gamma / e$. Plugging this value of $\beta$ into~\eqref{eq:t:bound:with:beta}, and rearranging, gives 
$$t \ge \bigg( 1 - \frac{\delta}{\gamma} \bigg) \bigg( 1 + \frac{1}{e(1 - \gamma)} \bigg)^{-1} k - o(k).$$
Since $\delta = \min\big\{ 1/200, \gamma/20 \big\}$, to deduce that $t \ge 2k/3$ it now suffices to check that\footnote{The inequality in~\eqref{eq:t:twothirds1} corresponds to a quadratic function being negative, so it suffices to check that the inequality holds for $\gamma = 1/10$ and $\gamma = 1/5$. For $\gamma = 1/10$ the left-hand side is at least $0.674$, by~\eqref{eq:t:twothirds2}, while for $\gamma = 1/5$ it is equal to $\big(1 - \frac{1}{40} \big)\big(1 + \frac{5}{4e} \big)^{-1} > 0.667$.}
\begin{equation}\label{eq:t:twothirds1}
\bigg( 1 - \frac{1}{200\gamma} \bigg)\bigg( 1 + \frac{1}{e(1-\gamma)} \bigg)^{-1} \ge \bigg(1 - \frac{1}{40} \bigg)\bigg(1 + \frac{5}{4e} \bigg)^{-1} > 0.667 > \frac{2}{3}
\end{equation}
for all $1/10 \le \gamma \le 1/5$, and that 
\begin{equation}\label{eq:t:twothirds2}
\bigg( 1 - \frac{1}{20} \bigg)\bigg( 1 + \frac{1}{e(1-\gamma)} \bigg)^{-1} \ge \bigg(1 - \frac{1}{20} \bigg)\bigg(1 + \frac{10}{9e} \bigg)^{-1} > 0.674 > \frac{2}{3}
\end{equation}
for all $\gamma \le 1/10$. 
\end{proof}

We will next use Lemma~\ref{lem:Ybound} to bound the size of $Y$ at the end of the algorithm. 

\begin{lemma}\label{lem:size:of:Y}
Fix $\delta,\gamma,\eta \ge 0$. If\/ $n \ge e^{-\delta k} {k + \ell \choose \ell}$ and\/ $p_0 \ge 1 - \gamma - \eta \ge 1/2$, then 
$$|Y| \ge e^{-\delta k + o(k)} \big( 1 - \gamma - \eta \big)^{\gamma t / (1 - \gamma)} \bigg( \frac{1 - \gamma - \eta}{1 - \gamma} \bigg)^t \exp\bigg( \frac{\gamma t^2}{2k} \bigg) {k-t+\ell \choose \ell}.$$
\end{lemma}

We'll use the following fact in the proof of Lemma~\ref{lem:size:of:Y}; a proof is given in~\cite[Appendix~D]{CGMS}.

\begin{fact}\label{binomial:fact4}
If $k,\ell,t \in \N$, with $\ell \le k$ and $t \le k$, then
$${k + \ell \choose \ell} \ge 2^{o(k)} (1-\gamma)^{-t} \exp\bigg( \frac{\gamma t^2}{2k} \bigg) {k - t + \ell \choose \ell},$$
where $\gamma = \frac{\ell}{k+\ell}$.  
\end{fact}

\begin{proof}[Proof of Lemma~\ref{lem:size:of:Y}]
Recall that $\beta \le \gamma$, and therefore $s \le \gamma t / (1 - \gamma) + o(k)$, by Lemma~\ref{lem:s:bound}. By Lemma~\ref{lem:Ybound}, and since $p_0 \ge 1 - \gamma - \eta \ge 1/2$, it follows that 
$$|Y| \, \ge 2^{o(k)} \big( 1 - \gamma - \eta \big)^{t + \gamma t / (1 - \gamma)} n.$$

\noindent By Fact~\ref{binomial:fact4} and our bound on $n$, it follows that
$$|Y| \, \ge e^{-\delta k + o(k)} \big( 1 - \gamma - \eta \big)^{\gamma t / (1 - \gamma)} \bigg( \frac{1 - \gamma - \eta}{1 - \gamma} \bigg)^t \exp\bigg( \frac{\gamma t^2}{2k} \bigg) {k-t+\ell \choose \ell},$$
as required.
\end{proof}

We can now prove Lemma~\ref{lem:off:diagonal:weak}, which provides the bound in Theorem~\ref{thm:off:diagonal:weak} when the red density is not too much smaller than $1 - \gamma$. 

\begin{proof}[Proof of Lemma~\ref{lem:off:diagonal:weak}]
It will suffice to show that
$$|Y| \ge {k-t+\ell \choose \ell} \ge R(k-t,\ell)$$
at the end of the algorithm. Observe first that $t \ge 2k / 3$, by Lemma~\ref{lem:t:big}, and that 
$$|Y| \ge e^{-\delta k + o(k)} \big( 1 - \gamma - \eta \big)^{\gamma t / (1 - \gamma)} \bigg( \frac{1 - \gamma - \eta}{1 - \gamma} \bigg)^t \exp\bigg( \frac{\gamma t^2}{2k} \bigg) {k-t+\ell \choose \ell},$$
by Lemma~\ref{lem:size:of:Y} and~\eqref{off:diag:weak:start:density}. To bound the right-hand side, we will use the fact that\footnote{This follows by monotonicity, and because $\log( (134/150)^{10/9}) > -0.126 > -1/3 + 1/5$.}
$$\big( 1 - \gamma - \eta \big)^{1 / (1 - \gamma)} \ge e^{-1/3 + 1/5}$$
for all $\gamma \le 1/10$ and $\eta \le \gamma/15$. Using the bound $t \ge 2k / 3$, it follows that
$$\big( 1 - \gamma - \eta \big)^{\gamma t / (1 - \gamma)} \exp\bigg( \frac{\gamma t^2}{2k} \bigg) \ge \exp\bigg( \frac{3\gamma t^2}{10k} \bigg).$$
Now, note that
$$\delta k = \frac{\gamma k}{20} \le \frac{9\gamma t^2}{80k} < \frac{3\gamma t^2}{20k},$$
and that 
$$\frac{1 - \gamma - \eta}{1 - \gamma} = 1 - \frac{\eta}{1-\gamma} \ge e^{-3\eta/2} \ge \exp\bigg( - \frac{3\gamma t}{20k} \bigg),$$
since $\gamma \le 1/10$, $\eta \le \gamma/15$ and $t \ge 2k / 3$. It follows that, if $k$ is sufficiently large, then
$$|Y| \ge {k-t+\ell \choose \ell} \ge R(k-t,\ell),$$
and hence $A \cup Y$ contains either a red $K_k$ or a blue $K_\ell$, as required.
\end{proof}

We can now deduce Theorem~\ref{thm:off:diagonal:weak} from Lemma~\ref{lem:off:diagonal:weak} by taking $m$ blue `Erd\H{o}s--Szekeres' steps, stopping when either the blue density in the set $U$ of remaining vertices is at most $16\gamma'/15$, where $\gamma' = \frac{\ell-m}{k+\ell-m}$, or $|U|$ is larger than the Erd\H{o}s--Szekeres bound on $R(k,\ell-m)$.

\begin{proof}[Proof of Theorem~\ref{thm:off:diagonal:weak}]
Let $n = R(k,\ell) - 1$, let $\chi$ be a red-blue colouring of $E(K_n)$ with no red $K_k$ or blue $K_\ell$, and set $\xi = 1/16$. Let $x_1,\ldots,x_m \in V(K_n)$ be a maximal sequence of distinct vertices such that 
$$x_i \in N_B(x_1) \cap \cdots \cap N_B(x_{i-1})$$
for each $i \in [m]$, and 
\begin{equation}\label{eq:after:ESz:steps}
|N_B(x_1) \cap \cdots \cap N_B(x_m)| \ge n \cdot (1 + \xi)^m \prod_{i = 0}^{m-1} \frac{\ell - i}{k + \ell - i},
\end{equation}
and observe that the vertices $x_1,\ldots,x_m$ induce a blue clique. Set 
$$\gamma' = \frac{\ell - m}{k + \ell - m},$$
and suppose first that $\gamma' \le \gamma^2$. In this case we simply apply the Erd\H{o}s--Szekeres bound~\eqref{eq:ESz:bound} inside $U = N_B(x_1) \cap \cdots \cap N_B(x_m)$, and deduce that if  
$$|U| \ge n \cdot (1 + \xi)^{m} \prod_{i = 0}^{m-1} \frac{\ell - i}{k + \ell - i} \ge {k + \ell - m \choose \ell - m},$$
then $U$ must contain either a red $K_k$ or a blue $K_{\ell-m}$, and therefore $\chi$ would contain either a red $K_k$ or a blue $K_\ell$, contradicting our assumption. Since $n = R(k,\ell) - 1$, it follows that if $\gamma' \le \gamma^2$, then
$$R(k,\ell) \le (1 + \xi)^{-m} {k + \ell \choose \ell} \le e^{-(1 - \gamma)\ell/20} {k + \ell \choose \ell} = e^{-\gamma k/20} {k + \ell \choose \ell},$$
as required, where in the second step we used the inequalities $1 + \xi > e^{1/20}$ and $\gamma' \le \gamma^2$, which implies that $\ell - m \le \gamma \ell$, and in the third we used the fact that $(1 - \gamma) \ell  = \gamma k$. 

If $\gamma^2 \le \gamma' \le \gamma$, on the other hand, then we apply Lemma~\ref{lem:off:diagonal:weak} to the colouring restricted to $U$. To do so, observe first that if $R(k,\ell) > e^{-\delta k} {k + \ell \choose \ell}$, then
$$|U| \ge e^{-\delta k} (1 + \xi)^m \prod_{i = m}^{\ell} \frac{k + \ell - i}{\ell - i} \gg 1,$$
by~\eqref{eq:after:ESz:steps} and since $\delta k = \gamma k/20 < \ell/20$, $1 + \xi > e^{1/20}$ and $\frac{k + \ell - i}{\ell - i} \ge \gamma^{-1} \ge 10$. Next, observe that, by the maximality of the sequence $x_1,\ldots,x_m$, we have 
$$|N_B(y) \cap U| \le (1+\xi) \bigg( \frac{\ell - m}{k + \ell - m} \bigg) |U| = (1+ \xi)\gamma' |U|$$
for each $y \in U$, and hence the colouring $\chi$ restricted to $U$ has red density at least 
$$1 - \frac{(1+\xi)\gamma'|U|}{|U| - 1} \ge 1 - \gamma' - \eta,$$
where $\eta = \gamma'/15$, since $\xi = 1/16$ and $|U|$ is sufficiently large. Since the colouring restricted to $U$ contains no red $K_k$ or blue $K_{\ell-m}$, it follows by Lemma~\ref{lem:off:diagonal:weak} that 
$$|U| \le e^{-\delta' k} {k + \ell - m \choose \ell - m},$$
where $\delta' = \gamma'/20$. Using~\eqref{eq:after:ESz:steps} once again, we deduce that
$$R(k,\ell) \le e^{-\delta' k} (1 + \xi)^{-m} {k + \ell \choose \ell} \le e^{-\delta k} {k + \ell \choose \ell},$$
as required, where the final step holds because
$$(1+\xi)^{m} \ge e^{m/20} \ge e^{\gamma k/20 - \gamma' k/20} = e^{\delta k} \cdot e^{- \delta' k},$$
since $1 + \xi > e^{1/20}$ and $\gamma - \gamma' = \frac{\ell}{k+\ell} - \frac{\ell - m}{k+\ell-m} \le m/k$ for all $0 \le m \le \ell$. 
\end{proof}

\section{An exponential improvement a little closer to the diagonal}\label{just:enough:sec}

We are now ready to prove~\eqref{eq:off:diag:just:enough}, the bound on the off-diagonal Ramsey numbers $R(k,\ell)$ that we will use to deduce our main result, Theorem~\ref{thm:diagonal}. We will in fact find it convenient to prove (and also to apply) the following slightly stronger bound.\footnote{Note that the bound $\gamma \le 1/5$ corresponds to $\ell \le k/4$, and that $\delta k = (1-\gamma)\ell/40 \ge \ell / 50$, so~\eqref{eq:off:diag:just:enough}~follows from Theorem~\ref{thm:off:diagonal:gamma}.}

\begin{theorem}\label{thm:off:diagonal:gamma}
Fix\/ $\gamma_0 > 0$, and let\/ $k,\ell \in \N$ be sufficiently large integers with $\gamma_0 \le \gamma \le 1/5$, where $\gamma = \frac{\ell}{k+\ell}$. Then 
$$R(k,\ell) \le e^{-\delta k + 1} {k + \ell \choose \ell},$$
where $\delta = \gamma/40$. 
\end{theorem}

As in the previous section, the lower bounds that we require on $k$ and $\ell$ depend only on $\gamma_0$. Once again, most of the complications are caused by the initial density of red edges being too low. Similarly to the previous section, we will deal with this issue by taking blue Erd\H{o}s--Szekeres steps until we can either apply the following lemma, or Theorem~\ref{thm:off:diagonal:weak}. 

\begin{lemma}\label{lem:off:diagonal:good:density}
Let $k,\ell \in \N$ be sufficiently large integers with $1/10 \le \gamma \le 1/5$, where $\gamma = \frac{\ell}{k+\ell}$. If 
\begin{equation}\label{eq:off:diagonal:good:density:n:bound}
n \ge e^{-k/200} {k + \ell \choose \ell},
\end{equation}
then every red-blue colouring of\/ $E(K_n)$ in which the density of red edges is at least $1 - \gamma$ contains either a red $K_k$ or a blue $K_\ell$. 
\end{lemma}

To prove Lemma~\ref{lem:off:diagonal:good:density}, we again apply the Book Algorithm with $\mu = \gamma$. To be precise, let $n \in \N$ satisfy~\eqref{eq:off:diagonal:good:density:n:bound}, and let $\chi$ be a red-blue colouring of $E(K_n)$ that contains no red $K_k$ or blue $K_\ell$, and in which the density of red edges is at least $1 - \gamma$. Choose disjoint sets of vertices $X$ and $Y$ with $|X|,|Y| \ge \lfloor n/2 \rfloor$ and 
\begin{equation}\label{off:diag:good:start:density}
\frac{e_R(X,Y)}{|X||Y|} \ge 1 - \gamma - \eta,
\end{equation}
and apply the Book Algorithm to the pair $(X,Y)$. As usual, we write $t = |\cR|$ and $s = |\cS|$, where $\cR$ and $\cS$ are the sets of indices of the red and density-boost steps. 

\begin{proof}[Proof of Lemma~\ref{lem:off:diagonal:good:density}]
It will suffice to show that
$$|Y| \ge {k-t+\ell \choose \ell} \ge R(k-t,\ell)$$
at the end of the algorithm. Since $p_0 \ge 1 - \gamma \ge 1/2$, by Lemma~\ref{lem:t:big} we have $t \ge 2k/3$. Moreover, by Lemma~\ref{lem:size:of:Y} (applied with $\eta = 0$) and our bounds~\eqref{eq:off:diagonal:good:density:n:bound} and~\eqref{off:diag:good:start:density} on $n$ and $p_0$, we have 
\begin{equation}\label{eq:Ybound:good:density}
|Y| \ge e^{-\delta k + o(k)} \big( 1 - \gamma \big)^{\gamma t / (1 - \gamma)} \exp\bigg( \frac{\gamma t^2}{2k} \bigg) {k-t+\ell \choose \ell},
\end{equation}
where $\delta = 1/200$. To bound the right-hand side of~\eqref{eq:Ybound:good:density}, define $\xi = \xi(\gamma)$ by 
$$\big( 1 - \gamma \big)^{1 / (1 - \gamma)} = e^{-1/3 + \xi},$$
and observe that 
$$\big( 1 - \gamma \big)^{\gamma t / (1 - \gamma)} \exp\bigg( \frac{\gamma t^2}{2k} \bigg) \ge e^{\xi \gamma t} \ge e^{2\xi \gamma k / 3},$$
since $t \ge 2k / 3$. It will therefore suffice to show that 
$$\delta = \frac{1}{200} < \frac{2\xi \gamma}{3},$$
that is, $\xi > 3/400\gamma$, which can be easily checked to hold for all $1/10 \le \gamma \le 1/5$. It therefore follows from~\eqref{eq:Ybound:good:density} and the Erd\H{o}s--Szekeres bound~\eqref{eq:ESz:bound} that
$$|Y| \ge {k-t+\ell \choose \ell} \ge R(k-t,\ell),$$
and thus $\chi$ contains either a red $K_k$ or a blue $K_\ell$, as required.
\end{proof}


We can now deduce Theorem~\ref{thm:off:diagonal:gamma} by taking blue Erd\H{o}s--Szekeres steps until we can apply either Theorem~\ref{thm:off:diagonal:weak} or Lemma~\ref{lem:off:diagonal:good:density}.

\begin{proof}[Proof of Theorem~\ref{thm:off:diagonal:gamma}]
If $\gamma \le 1/10$ then the claimed bound follows from Theorem~\ref{thm:off:diagonal:weak}, so we may assume that $1/10 \le \gamma \le 1/5$. Let $n = R(k,\ell) - 1$, and let $\chi$ be a red-blue colouring of $E(K_n)$ that contains no red $K_k$ or blue $K_\ell$. 
Let $x_1,\ldots,x_m \in V(K_n)$ be a sequence of distinct vertices such that
$$x_i \in N_B(x_1) \cap \cdots \cap N_B(x_{i-1})$$
for each $i \in [m]$, and 
\begin{equation}\label{eq:size:of:U:ESz:steps}
|N_B(x_1) \cap \cdots \cap N_B(x_m)| \ge n \cdot \prod_{i = 0}^{m-1} \frac{\ell - i}{k + \ell - i} - m
\end{equation}
and such that, setting 
$$\gamma' = \frac{\ell - m}{k + \ell - m} \qquad \text{and} \qquad U = N_B(x_1) \cap \cdots \cap N_B(x_m),$$
we have either 
\begin{itemize}
\item[$(a)$] $\gamma' \ge 1/10$ and the colouring $\chi$ restricted to the set $U$ has blue density at most $\gamma'$, or\smallskip
\item[$(b)$] $1/10 - 1/k \le \gamma' \le 1/10$.
\end{itemize}

We first claim that such a sequence exists. To see this observe first that if (for a given sequence $x_1,\ldots,x_m$) the colouring restricted to $U$ has blue density at least $\gamma'$, then there exists a vertex $y \in U$ such that 
$$|N_B(y) \cap U| \ge \gamma' |U| - 1 \ge n \cdot \prod_{i = 0}^{m} \frac{\ell - i}{k + \ell - i} - (m+1),$$
by~\eqref{eq:size:of:U:ESz:steps}, so we can set $x_{m+1} = y$ and continue the sequence. Since
$$0 \le \frac{\ell - m + 1}{k + \ell - m + 1} - \frac{\ell - m}{k + \ell - m} \le \frac{1}{k}$$
for every $1 \le m \le \ell$, it follows that if we never find a sequence such that $(a)$ holds, we must eventually reach a value of $m$ such that $(b)$ holds, as claimed.  

In case~$(a)$, we apply Lemma~\ref{lem:off:diagonal:good:density} to this restricted colouring, and deduce that if 
$$|U| \ge n \cdot \prod_{i = 0}^{m-1} \frac{\ell - i}{k + \ell - i} - m \ge e^{-k/200} {k + \ell - m \choose \ell - m},$$
then $\chi$ contains either a red $K_k$ or a blue $K_\ell$. Since $\gamma \le 1/5$, it follows in this case that
$$R(k,\ell) \le e^{-\gamma k/40 + 1} {k + \ell \choose \ell}$$
as required. On the other hand, if $1/10 - 1/k \le \gamma' \le 1/10$, then we apply Theorem~\ref{thm:off:diagonal:weak}, and deduce that if 
$$|U| \ge n \cdot \prod_{i = 0}^{m-1} \frac{\ell - i}{k + \ell - i} - m \ge e^{-\gamma' k/20} {k + \ell - m \choose \ell - m},$$
then $\chi$ contains either a red $K_k$ or a blue $K_\ell$. Since $\gamma \le 1/5 \le 2\gamma' + 2/k$, this gives
$$R(k,\ell) \le e^{-\gamma k/40 + 1} {k + \ell \choose \ell}$$
as claimed. 
\end{proof}


\section{From Off-diagonal to Diagonal}\label{moving:sec}

In this section we will use the Book Algorithm to bound diagonal Ramsey numbers in terms of near-diagonal Ramsey numbers; together with some straightforward analysis, this will allow us (in Section~\ref{diagonal:sec}) to deduce Theorem~\ref{thm:diagonal} from Theorem~\ref{thm:off:diagonal:gamma}. 
 
In order to state the main result of this section, we need to define two functions:
\begin{equation}\label{def:F}
F_k(x,y) = \frac{1}{k}\log_2 R\big( k, k - xk \big) + x + y
\end{equation}
and 
\begin{equation}\label{def:G}
G_\mu(x,y) = \log_2 \left( \frac{1}{\mu} \right) + x \cdot \log_2 \left(\frac{1}{1 - \mu} \right) + y \cdot \log_2 \left(\frac{\mu(x+y)}{y}\right).
\end{equation} 

\pagebreak

\noindent The reader should think of $x = t/k$ and $y = s/k$, where $t$ and $s$ are  the number of red and density-boost steps (respectively) that occur during our application of the Book Algorithm. We will apply the following theorem with $\mu = 2/5$ and $\eta > 0$ sufficiently small.  

\begin{theorem}[Diagonal vs off-diagonal]\label{thm:moving}
Fix\/ $0 < \mu \le 1/2$ and\/ $\eta > 0$. We have 
\begin{equation}\label{eq:diag:vs:offdiag}
\frac{\log_2 R(k)}{k} \le \max_{\substack{\\[-0.05ex] 0 \,\le\, x \,\le\, 1 \\[+0.3ex] 0 \,\le\, y \,\le\, \mu x/(1-\mu)+\eta}} \hspace{-0.6cm} \min \big\{ F_k(x,y),G_\mu(x,y) \big\} + \eta
\end{equation}
for all sufficiently large $k \in \N$. 
\end{theorem}

Throughout this section, we fix $\mu \in (0,1/2]$ and $\eta > 0$, and let $k \in \N$ be sufficiently large. Set $n = R(k)-1$, and let $\chi$ be a red-blue colouring of $E(K_n)$ with no monochromatic $K_k$. Assuming (by symmetry) that the initial density of red edges in $\chi$ is at least $1/2$, choose disjoint sets of vertices $X$ and $Y$ with $|X|,|Y| \ge \lfloor n/2 \rfloor$ and 
$$\frac{e_R(X,Y)}{|X||Y|} \ge \frac{1}{2}.$$
We now apply the Book Algorithm to the pair $(X,Y)$; as usual, we write $t$ for the number of red steps taken, and $s$ for the number of density-boost steps. The following lemma gives the second bound in~\eqref{eq:diag:vs:offdiag}; we will prove it by considering the size of the set $X$ at the end of the algorithm. 

\begin{lemma}\label{lem:diag:inequality1}
\begin{equation}\label{eq:first-s-t-inequal}
\log_2 R(k) \le k \cdot \log_2 \left( \frac{1}{\mu} \right) + t \cdot \log_2 \left(\frac{1}{1-\mu} \right) + s \cdot \log_2 \left(\frac{\mu(s+t)}{s}\right) + \eta k.
\end{equation}
\end{lemma}

\begin{proof}
Recall that at the start of the algorithm $|X| \ge \lfloor n/2 \rfloor$, and that therefore after $t$ red steps, $s$~density-boost steps, and an unknown number of big blue and degree-regularisation steps, we have 
\begin{equation}\label{eq:diag:Xbound}
|X| \ge 2^{o(k)} \mu^k (1 - \mu)^t \bigg( \frac{\beta}{\mu} \bigg)^s n
\end{equation}
by Lemma~\ref{lem:Xbound}, applied with $\ell = k$. Observe that $|X| = 2^{o(k)}$ at the end of the algorithm, since $p_0 \ge 1/2$, and hence $p \ge 1/4$ throughout the algorithm, by Lemma~\ref{lem:bounding:p}. It follows that
\begin{equation}\label{eq:diag:Gbound}
R(k) \le 2^{o(k)} \mu^{-k} (1 - \mu)^{-t} \bigg( \frac{\mu}{\beta} \bigg)^s.
\end{equation}
Now, by Lemma~\ref{lem:beta:bound}, there exists a constant $c > 0$ such that either $s \le k^{1-c}$, or
\begin{equation}\label{eq:diag:beta:bound}
\beta \ge \big(1 + o(1) \big) \bigg( \frac{s}{s + t} \bigg). 
\end{equation}
\pagebreak
\noindent If $s \le k^{1-c}$, then $( \mu / \beta )^s = 2^{o(k)}$, since $\mu$ is constant and $\beta \ge 1/k^2$, by Lemma~\ref{lem:density:boost:weak}. We therefore deduce from~\eqref{eq:diag:Gbound} that in this case
$$\log_2 R(k) \le k \cdot \log_2 \left( \frac{1}{\mu} \right) + t \cdot \log_2 \left(\frac{1}{1-\mu} \right) + o(k),$$ 
and this in turn implies~\eqref{eq:first-s-t-inequal}, since $s \cdot \log_2 \frac{\mu(s+t)}{s} = o(k)$. On the other hand, if~\eqref{eq:diag:beta:bound} holds, then
$$s \cdot \log_2 \bigg( \frac{\mu}{\beta} \bigg) \le s \cdot \log_2 \bigg( \frac{\mu(s+t)}{s} \bigg) + o(k),$$
and therefore~\eqref{eq:diag:Gbound} gives 
$$\log_2 R(k) \le k \cdot \log_2 \left( \frac{1}{\mu} \right) + t \cdot \log_2 \left(\frac{1}{1-\mu} \right) + s \cdot \log_2 \left(\frac{\mu(s+t)}{s}\right) + o(k),$$
as claimed. 
\end{proof}

To prove the other inequality in~\eqref{eq:diag:vs:offdiag}, we consider the size of $Y$ at the end of the algorithm.  

\begin{lemma}\label{lem:diag:inequality2} 
\begin{equation}\label{eq:second-s-t-inequal}
\log_2 R(k) \le \log_2 R(k,k-t) + s + t + \eta k.
\end{equation}
\end{lemma}

\begin{proof}
Noting that $p_0 \ge 1/2$ and $|Y| \ge \lfloor n/2 \rfloor$ at the start of the algorithm, it follows from Lemma~\ref{lem:Ybound} that after $t$ red steps and $s$~density-boost steps, we have
$$|Y|  \, \ge 2^{ - s - t + o(k)} n.$$
Now, if $|Y| \ge R(k,k-t)$ at the end of the algorithm, then it follows that $\chi$ must contain a monochromatic $K_k$, which contradicts our assumption on $\chi$. We therefore obtain 
$$R(k) \le 2^{ s + t + o(k)} R(k,k-t),$$
as claimed.
\end{proof}


Theorem~\ref{thm:moving} follows easily from Lemmas~\ref{lem:diag:inequality1} and~\ref{lem:diag:inequality2}, and Lemma~\ref{lem:s:bound}.

\begin{proof}[Proof of Theorem~\ref{thm:moving}]
By Lemmas~\ref{lem:diag:inequality1} and~\ref{lem:diag:inequality2}, we have 
$$\frac{\log_2 R(k)}{k} \le \min \big\{ F_k(x,y), G_\mu(x,y) \big\} + \eta,$$
where $x = t/k$ and $y = s/k$, and the pair $(s,t)$ is given to us by the algorithm, applied to the colouring $\chi$. To complete the proof of the theorem, it therefore suffices to observe that 
$$s \le \bigg( \frac{\beta}{1 - \beta} \bigg) t + O\big( k^{1-c} \big) \le \bigg( \frac{\mu}{1 - \mu} \bigg) t + O\big( k^{1-c} \big),$$ 
where the first inequality holds by Lemma~\ref{lem:s:bound}, the second holds since $\beta \le \mu$. 
Thus, the claimed inequality holds for all sufficiently large $k \in \N$, as required.
\end{proof}

\section{The proof of Theorem~\ref{thm:diagonal}}\label{diagonal:sec}

In this section we will show how to deduce Theorem~\ref{thm:diagonal} from Theorems~\ref{thm:off:diagonal:gamma} and~\ref{thm:moving}, the latter applied with $\mu = 2/5$. Recall from~\eqref{def:F} and~\eqref{def:G} the definitions of the functions $F_k(x,y)$ and $G_\mu(x,y)$, set $g(x,y) = G_{2/5}(x,y)$, and note that $$g(x,y) = \log_2 \left( \frac{5}{2} \right) + x \cdot \log_2 \left(\frac{5}{3} \right) + y \cdot \log_2 \left(\frac{2(x+y)}{5y}\right).$$ 
We will use the following standard fact to bound $F_k(x,y)$. 

\begin{fact}\label{fact:entropy}
Let $a,b \in \N$ with $a \ge b$. Then
$$\log_2 {a \choose b} \le a \cdot h\bigg( \frac{b}{a} \bigg),$$
where $h(p) = - p \log_2 p - (1 - p) \log_2(1 - p)$ is the binary entropy function.
\end{fact}

Using Fact~\ref{fact:entropy}, it follows from the usual Erd\H{o}s--Szekeres bound~\eqref{eq:ESz:bound} on $R(k,\ell)$ that
\begin{equation}\label{eq:F:ESz:bound}
F_k(x,y) \le f_1(x,y) := x + y + (2 - x) \cdot h\bigg( \frac{1}{2-x} \bigg)
\end{equation}
for all $0 \le x,y \le 1$ and all $k \in \N$. Moreover, let us fix a sufficiently small constant $\eta > 0$ and define
\begin{equation}\label{eq:F:our:bound}
f_2(x,y) := x + y + (2 - x) \cdot h\left(\frac{1}{2-x}\right) - \frac{\log_2 e}{40} \bigg( \frac{1-x}{2-x} \bigg) + \eta
\end{equation} 
for all $0 \le x,y \le 1$. Applying Theorem~\ref{thm:off:diagonal:gamma} for a sufficiently small constant $\gamma_0 > 0$, we obtain the following bound on $F_k(x,y)$. 



\begin{corollary}\label{cor:bounding:F}
If $k \in \N$ is sufficiently large (depending on $\eta$) and $0 \le x,y \le 1$, then
$$
F_k(x,y) \le f(x,y) :=\, \left\{
\begin{array} {c@{\quad \textup{if} \quad}l}
   f_1(x,y) & x < 3/4, \\[+1ex]
   f_2(x,y) & x \ge 3/4.  
\end{array}\right.
$$
\end{corollary}

\begin{proof}
By~\eqref{eq:F:ESz:bound}, we have $F_k(x,y) \le f_1(x,y)$ for all $0 \le x,y \le 1$ and all $k \in \N$, and if $x \ge 1 - \eta$ then $f_1(x,y) \le f_2(x,y)$, so it will suffice to prove that if $k$ is sufficiently large, then $F_k(x,y) \le f_2(x,y)$ for all $3/4 \le x \le 1 - \eta$ and $0 \le y \le 1$. This follows from Theorem~\ref{thm:off:diagonal:gamma}, applied with $\gamma_0 = \eta / (1+\eta)$, since if $\ell = (1-x)k$, then $\gamma = \frac{\ell}{k+\ell} = \frac{1-x}{2-x}$, 
and therefore the range $3/4 \le x \le 1 - \eta$ corresponds to the range $\gamma_0 \le \gamma \le 1/5$. 
\end{proof}

\begin{figure}[t]
  \centering
  \begin{subfigure}[b]{0.48\textwidth}
      \includegraphics[width=\textwidth]{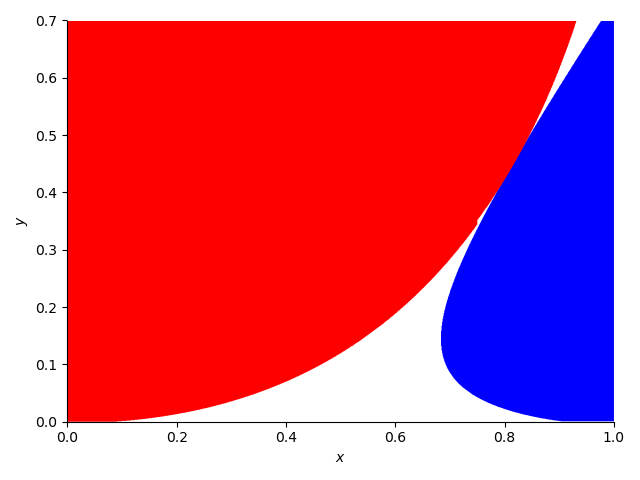}
   \end{subfigure}
  \hspace{0.35cm}
  \begin{subfigure}[b]{0.48\textwidth}
       \includegraphics[width=\textwidth]{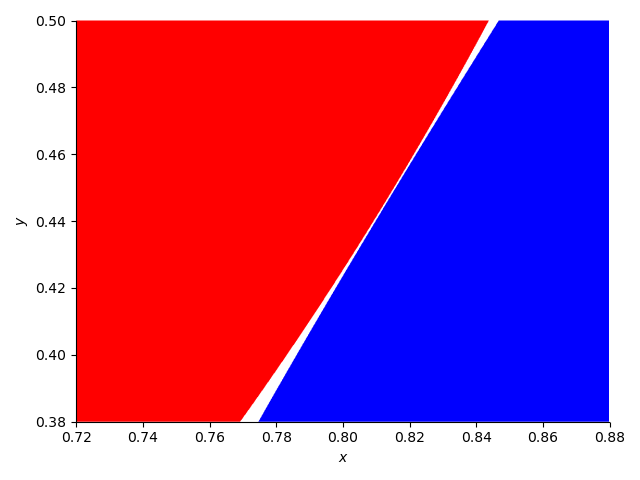}
  \end{subfigure}
  \vskip-0.15cm
  \caption{The sets $f(x,y) \ge 2$ (in red) and $g(x,y) \ge 2$ (in blue).}
  \label{fig:diagonal}
\end{figure}

We will use the following straightforward numerical fact.

\smallskip
\pagebreak
 
\begin{lemma}\label{lem:final:calc}
If\/ $\delta \le 2^{-11}$, then  
$$\max_{\substack{\\[-0.05ex] 0 \,\le\, x \,\le\, 1 \\[+0.3ex] 0 \,\le\, y \,\le\, 3/4}} \hspace{-0.2cm} \min \big\{ f(x,y),g(x,y) \big\} < 2 - \delta.$$
\end{lemma}

The proof of Lemma~\ref{lem:final:calc} is a simple calculation, and the reader may well be satisfied by Figure~\ref{fig:diagonal}, which illustrates the regions in which the functions $f(x,y)$ and $g(x,y)$ are at least~$2$. The statement of the lemma is (essentially) that these two regions do not intersect. We provide a detailed proof of Lemma~\ref{lem:final:calc} in Appendix~\ref{app:final:calc}. 

We can now easily deduce Theorem~\ref{thm:diagonal} from Theorem~\ref{thm:moving} and Lemma~\ref{lem:final:calc}.
 
\begin{proof}[Proof of Theorem~\ref{thm:diagonal}]
By Theorem~\ref{thm:moving}, applied with $\mu = 2/5$, we have 
$$\frac{\log_2 R(k)}{k} \le \max_{\substack{\\[-0.05ex] 0 \,\le\, x \,\le\, 1 \\[+0.3ex] 0 \,\le\, y \,\le\, 3/4 + \eta}} \hspace{-0.3cm} \min \big\{ F_k(x,y),g(x,y) \big\} + \eta.$$
By Corollary~\ref{cor:bounding:F} and Lemma~\ref{lem:final:calc}, and since $\eta$ is sufficiently small, it follows that
$$\frac{\log_2 R(k)}{k} \le 2 - \delta$$
for some constant $\delta > 0$ and all sufficiently large $k \in \N$, as required.
\end{proof} 
 
\medskip
 
\section{Ramsey numbers near the diagonal}\label{neardiagonal:sec}

In this section we will prepare the ground for the proof of Theorem~\ref{thm:off:diagonal}, and also for our second proof of Theorem~\ref{thm:diagonal}, by proving the following quantitative bound for Ramsey numbers that are fairly close to the diagonal. To obtain an exponential improvement in the remaining range, we will first `walk' towards the diagonal (by taking red Erd\H{o}s--Szekeres steps, while the blue density is too small), and then `jump' away from it (by finding a big blue book, as soon as the density is large enough) into the range covered by Theorem~\ref{thm:off:diagonal:near} (see Section~\ref{finalproof:sec} for the details). Since this jump will be expensive, it will be important that the bound we prove here is strong enough, and that it holds sufficiently close to the diagonal. 

\begin{theorem}\label{thm:off:diagonal:near}
$$R(k,\ell) \le e^{-\ell/80 + o(k)} {k + \ell \choose \ell}$$
for every sufficiently large $k,\ell \in \N$ with $\ell \le 9k/10$.
\end{theorem}

Note that when $\ell \le k/4$, the conclusion of Theorem~\ref{thm:off:diagonal:near} follows from Theorem~\ref{thm:off:diagonal:gamma}. 
We will prove Theorem~\ref{thm:off:diagonal:near} in two steps: first we obtain a stronger bound when $\ell \le 2k/3$; then we deduce the claimed bound for the full range. (We remark that this is similar to the strategy we used in Sections~\ref{simple:sec} and~\ref{just:enough:sec} to prove Theorem~\ref{thm:off:diagonal:gamma}.)

The first step is to prove a generalisation of Theorem~\ref{thm:moving} that is suitable for the off-diagonal setting, where an extra complication arises due to the fact we cannot use symmetry to guarantee that the density of red edges is sufficiently large. In order to state this generalisation, we need to define the following variants of the functions $F_k$ and $G_\mu$ from Section~\ref{moving:sec}. Given $k,\ell \in \N$, $0 < \mu \le 1/2 \le \nu \le 1$ and $0 < \gamma \le 1/2$, let 
\begin{equation}\label{def:Fstar}
F^*_{\nu,k,\ell}(x,y) = \frac{1}{k} \cdot \log R\big( (1-x)k,\ell \big)+ \big( x + y \big) \log \left( \frac{1}{\nu} \right)
\end{equation}
and 
\begin{equation}\label{def:Gstar}
G^*_{\mu,\gamma}(x,y) = \frac{\gamma}{1-\gamma} \cdot \log \left( \frac{1}{\mu} \right) + x \cdot \log \left(\frac{1}{1 - \mu} \right) + y \cdot \log \left(\frac{\mu(x+y)}{y}\right).
\end{equation}
Note that here our logs are base $e$, instead of base $2$, since this will be slightly more convenient in the calculations. The reader should think of the constant $\nu$ as being the density of red edges in the colouring~$\chi$. We will prove the following variant of Theorem~\ref{thm:moving}.

\begin{theorem}\label{thm:generalbound}
Fix\/ $\mu_0 > 0$ and\/ $\eta > 0$, let\/ $k,\ell\in \N$ be sufficiently large integers with\/ $\ell \le k$, and let $\mu_0 \le \mu \le 1/2 \le \nu \le 1$. If 
$$\frac{\log n}{k} \, \ge \max_{\substack{\\[-0.05ex] 0 \,\le\, x \,\le\, 1 \\[+0.3ex] 0 \,\le\, y \,\le\, \mu x/(1-\mu)+\eta}} \hspace{-0.6cm} \min \big\{ F^*_{\nu,k,\ell} (x,y),G^*_{\mu,\gamma}(x,y) \big\} + \eta,$$ 
where $\gamma = \frac{\ell}{k+\ell}$, then every red-blue colouring of $E(K_n)$ in which the density of red edges is at least\/ $\nu$ contains~either a red\/ $K_k$ or a blue\/ $K_\ell$. 
\end{theorem}

We remark that the lower bounds on $k$ and $\ell$ in Theorem~\ref{thm:generalbound} depend only on $\mu_0$ and $\eta$, and in particular can be chosen uniformly in $\nu$. 
This will be important in our applications, since we will not be able to control the exact value of $\nu$, but we will be able to guarantee that $\nu \ge 1/2$, see the proofs of Theorems~\ref{thm:off:diagonal:near} and~\ref{thm:off:diagonal:nearer}. 
 
\pagebreak
 
The proof of Theorem~\ref{thm:generalbound} is almost identical to that of Theorem~\ref{thm:moving}, so we will be slightly brief with the details. Fix $\gamma_0 > 0$ and $\eta > 0$, let $\ell \le k$ be sufficiently large integers, and let $\gamma_0 \le \mu \le 1/2 \le \nu \le 1$. 
Let $\chi$ be a red-blue colouring of $E(K_n)$ that does not contain either a red $K_k$ or a blue $K_\ell$, and in which the density of red edges is at least $\nu$. Let $X$ and $Y$ be disjoint sets of vertices with $|X|,|Y| \ge \lfloor n/2 \rfloor$ and 
$$\frac{e_R(X,Y)}{|X||Y|} \ge \nu,$$
and apply the Book Algorithm to the pair $(X,Y)$. As usual, we write $t$ for the number of red steps taken, and $s$ for the number of density-boost steps. 
  
The following lemma is a straightforward generalisation of Lemma~\ref{lem:diag:inequality1}.
  
\begin{lemma}\label{lem:offdiag:inequality1}
$$\log n \le \ell \cdot \log \left( \frac{1}{\mu} \right) + t \cdot \log \left(\frac{1}{1-\mu} \right) + s \cdot \log \left(\frac{\mu(s+t)}{s}\right) + \eta k$$
\end{lemma}

\begin{proof}
We repeat the proof of Lemma~\ref{lem:diag:inequality1}, except replacing the bound~\eqref{eq:diag:Xbound} with 
$$|X| \ge 2^{o(k)} \mu^\ell (1 - \mu)^t \bigg( \frac{\beta}{\mu} \bigg)^s n,$$
which also follows from Lemma~\ref{lem:Xbound}. Since $p_0 \ge \nu \ge 1/2$, it follows that $|X| = 2^{o(k)}$ at the end of the algorithm, and therefore 
$$n \le 2^{o(k)} \mu^{-\ell} (1 - \mu)^{-t} \bigg( \frac{\mu}{\beta} \bigg)^s.$$
The rest of the argument is now identical to that in the proof of Lemma~\ref{lem:diag:inequality1}.
\end{proof}

We will also need the following straightforward generalisation of Lemma~\ref{lem:diag:inequality2}.

\begin{lemma}\label{lem:offdiag:inequality2} 
$$n \le 2^{\eta k} \bigg( \frac{1}{\nu} \bigg)^{s+t} R(k-t,\ell).$$
\end{lemma}

\begin{proof}
Since $p_0 \ge \nu \ge 1/2$ and $|Y| \ge \lfloor n/2 \rfloor$ at the start of the algorithm, it follows from Lemma~\ref{lem:Ybound} that after $t$ red steps and $s$~density-boost steps, we have
$$|Y|  \, \ge 2^{o(k)} \nu^{s + t} n.$$
Now, if $|Y| \ge R(k-t,\ell)$ at the end of the algorithm, then it follows that $\chi$ must contain a red $K_k$ or a blue $K_\ell$, which contradicts our assumption on $\chi$. We therefore obtain 
$$n \le 2^{o(k)} \bigg( \frac{1}{\nu} \bigg)^{s+t} R(k-t,\ell),$$
as required.
\end{proof}

Theorem~\ref{thm:generalbound} follows easily from Lemmas~\ref{lem:offdiag:inequality1} and~\ref{lem:offdiag:inequality2}, and Lemma~\ref{lem:s:bound}.

\begin{proof}[Proof of Theorem~\ref{thm:generalbound}]
By Lemmas~\ref{lem:offdiag:inequality1} and~\ref{lem:offdiag:inequality2}, and since $\gamma k = (1 - \gamma)\ell$, we have 
$$\frac{\log n}{k} \le \min \big\{ F^*_{\nu,k,\ell}(x,y), G^*_{\mu,\gamma}(x,y) \big\} + \eta,$$
where $x = t/k$ and $y = s/k$, and the pair $(s,t)$ is given to us by the algorithm, applied to the colouring $\chi$. To complete the proof of the theorem, it therefore suffices to observe that 
$$s \le \bigg( \frac{\beta}{1 - \beta} \bigg) t + O\big( k^{1-c} \big) \le \bigg( \frac{\mu}{1 - \mu} \bigg) t + O\big( k^{1-c} \big),$$ 
by Lemma~\ref{lem:s:bound} and since $\beta \le \mu$. 
It follows that the claimed inequality holds for all sufficiently large $k,\ell \in \N$, as required.
\end{proof}


In order to apply Theorem~\ref{thm:generalbound}, we need a bound on $R(k-t,\ell)$. It turns out that the simple Erd\H{o}s--Szekeres bound~\eqref{eq:ESz:bound} will suffice to prove Theorem~\ref{thm:off:diagonal:near}, and we will not need the stronger bound given by Theorem~\ref{thm:off:diagonal:gamma}. To be precise, we will use the following bound. 


\begin{obs}\label{obs:generalF:ESzbound}
Let $k,\ell \in \N$ and $1/2 \le \nu \le 1$, and set $\theta = \ell / k$. 
Then
\begin{equation}\label{def:fnuthetastar}
F^*_{\nu,k,\ell}(x,y) \le f_{\nu,\theta}^*(x,y) := (x + y)\log \left(\frac{1}{\nu}\right) + \big( 1 + \theta - x \big) \cdot h^*\bigg( \frac{\theta}{1 + \theta - x} \bigg),
\end{equation}
where $h^*(p) = - p \log p - (1 - p) \log(1 - p)$ is the entropy function.
\end{obs}

\begin{proof}
By~\eqref{def:Fstar} and the Erd\H{o}s--Szekeres bound~\eqref{eq:ESz:bound}, it will suffice to show that 
$$\log {(1-x)k + \ell \choose \ell} \le \big( 1 + \theta - x \big) \cdot h^*\bigg( \frac{\theta}{1 + \theta - x} \bigg) \cdot k.$$
Since $\ell = \theta k$, this follows immediately from Fact~\ref{fact:entropy}. 
\end{proof}

Let us note here the following variant of Fact~\ref{fact:entropy}, which we will use in the proofs below. 

\begin{fact}\label{fact:entropy2}
Let $k,\ell \in \N$ with $\ell \le k$. Then
$$\log {k + \ell \choose \ell} = (k + \ell) \cdot h^*\bigg( \frac{\ell}{k + \ell} \bigg) + o(k),$$
\end{fact}

Using Observation~\ref{obs:generalF:ESzbound}, we can prove the following two lemmas. We emphasize that these are purely numerical claims and make no use of the structure of the colourings. The first of the two lemmas deals with the range $k/4 \le \ell \le 2k/3$. 

\begin{lemma}\label{lem:nearish:calc}
Let $k,\ell \in \N$, with $k/4 \le \ell \le 2k/3$, and set $\gamma = \ell/(k+\ell)$, $\theta = \ell/k$, and $\eta = \gamma/40$. If\/ $\mu = \gamma$ and\/ $\nu = 1 - \gamma - \eta$, then
$$\max_{\substack{\\[-0.05ex] 0 \,\le\, x \,\le\, 1 \\[+0.3ex] 0 \,\le\, y \,\le\, 3/4}} \hspace{-0.2cm} \min \big\{ f_{\nu,\theta}^* (x,y),G^*_{\mu,\gamma}(x,y) \big\} < \frac{h^*(\gamma)}{1-\gamma} - \frac{1}{50}.$$
\end{lemma}

The second lemma allows $\ell$ to be almost as large as $k$, but gives a weaker bound, and requires the initial density of red edges to be at least $1 - \gamma$.

\begin{lemma}\label{lem:even:nearer:calc}
Let $k,\ell \in \N$, with $2k/3 \le \ell \le 9k/10$, and set $\gamma = \ell/(k+\ell)$ and $\theta = \ell/k$. If\/ $\mu = 2/5$ and $\nu = 1 - \gamma$, then
$$\max_{\substack{\\[-0.05ex] 0 \,\le\, x \,\le\, 1 \\[+0.3ex] 0 \,\le\, y \,\le\, 5/7}} \hspace{-0.2cm} \min \big\{ f_{\nu,\theta}^* (x,y),G^*_{\mu,\gamma}(x,y) \big\} < \frac{h^*(\gamma)}{1-\gamma} - \frac{1}{80}.$$
\end{lemma}

The (easy) proofs of Lemmas~\ref{lem:nearish:calc} and~\ref{lem:even:nearer:calc} are given in Appendices~\ref{app:off:calc:gamma} and~\ref{app:off:calc:twofifths}, respectively. The statements of the two lemmas are illustrated in Figure~\ref{fig:offdiagonal}.

\begin{figure}[t]
  \centering
  \begin{subfigure}[b]{0.48\textwidth}
       \includegraphics[width=\textwidth]{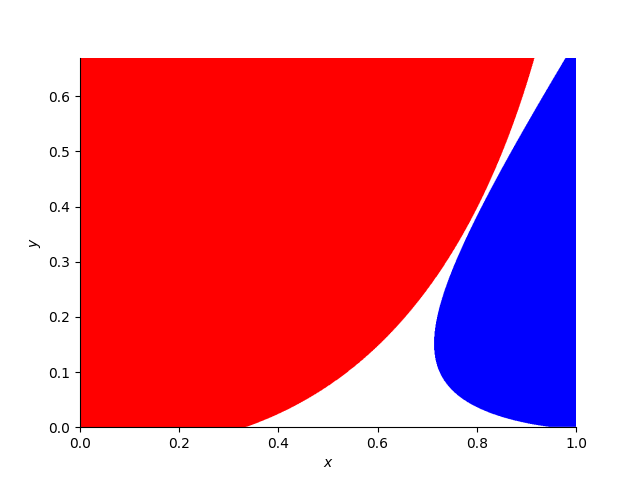}
    \caption{Lemma~\ref{lem:nearish:calc} when $\ell = 2k/3$.}
  \end{subfigure}
  \hspace{0.3cm}
  \begin{subfigure}[b]{0.48\textwidth}
   \includegraphics[width=\textwidth]{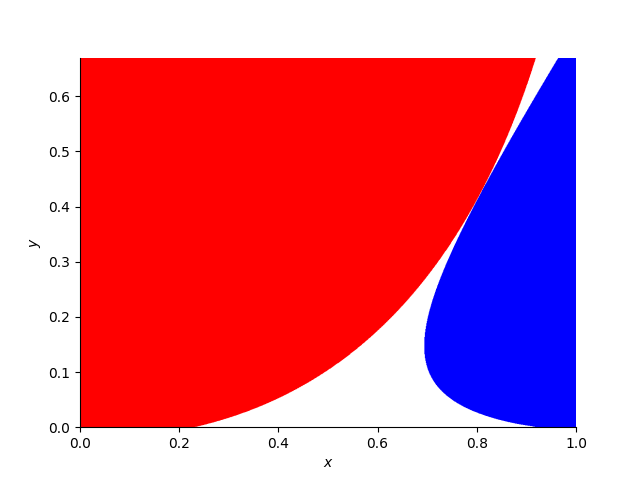}
    \caption{Lemma~\ref{lem:even:nearer:calc} when $\ell = 9k/10$.}
  \end{subfigure}
  \caption{The extreme cases of Lemmas~\ref{lem:nearish:calc} and~\ref{lem:even:nearer:calc}: roughly speaking, the statement of the lemmas is that the red and blue sets do not intersect.}
  \label{fig:offdiagonal}
\end{figure}


We can now deduce the following theorem from Theorem~\ref{thm:generalbound} and Lemma~\ref{lem:nearish:calc}, simply by taking blue steps until the red density is large enough. 
 
\begin{theorem}\label{thm:off:diagonal:nearer}
We have
$$R(k,\ell) \le e^{-\ell/50 + o(k)} {k + \ell \choose \ell}$$
for every $k,\ell \in \N$, with $\ell \le 2k/3$.
\end{theorem} 
 
\begin{proof}
Set $n = R(k,\ell) - 1$, let $\chi$ be a red-blue colouring of $E(K_n)$ containing no red $K_k$ or blue $K_\ell$, and set $\xi = 1/40$. If $\ell \le k/4$, then the claimed bound on $R(k,\ell)$ follows from Theorem~\ref{thm:off:diagonal:gamma}, so we may assume that $k/4 \le \ell \le 2k/3$. We may therefore choose a sequence of distinct vertices $x_1,\ldots,x_m \in V(K_n)$ such that
$$x_i \in N_B(x_1) \cap \cdots \cap N_B(x_{i-1})$$
for each $i \in [m]$, and 
\begin{equation}\label{eq:after:ESz:steps:again}
|N_B(x_1) \cap \cdots \cap N_B(x_m)| \, \ge \, n \cdot(1 + \xi)^m \prod_{i = 0}^{m-1} \frac{\ell - i}{k + \ell - i} - m
\end{equation}
and such that, writing 
$$\gamma' = \frac{\ell - m}{k + \ell - m},$$
one of the following holds:
\begin{itemize}
\item[$(a)$] $\ell - m > k/4$ and the colouring $\chi$ restricted to the set $U = N_B(x_1) \cap \cdots \cap N_B(x_m)$ has blue density at most $(1 + \xi)\gamma'$;\smallskip
\item[$(b)$] $\ell - m = \lfloor k/4 \rfloor$.\smallskip
\end{itemize}

If~$(b)$ holds, then we apply Theorem~\ref{thm:off:diagonal:gamma}, and obtain
$$n \le (1+\xi)^{-m} \bigg( \prod_{i = 0}^{m-1} \frac{k + \ell - i}{\ell - i} \bigg) \cdot e^{-(\ell-m)/50 + o(k)} {k + \ell - m \choose \ell - m},$$
since $\gamma' = 1/5 + o(1)$ and $\gamma' k = (1-\gamma')(\ell - m)$. Since $1 + \xi \ge e^{1/50}$, it follows that
$$R(k,\ell) \le e^{-\ell/50 + o(k)} {k + \ell \choose \ell},$$
as required.

If~$(a)$ holds, on the other hand, then we apply Theorem~\ref{thm:generalbound} to the colouring induced by the set $U$, with $\mu = \gamma'$ and $\nu = 1 - (1 + \xi)\gamma'$. It follows that 
$$\frac{\log |U|}{k} \le \max_{\substack{\\[-0.05ex] 0 \,\le\, x \,\le\, 1 \\[+0.3ex] 0 \,\le\, y \,\le\, 3/4}} \hspace{-0.3cm} \min \big\{  F^*_{\nu,k,\ell-m} (x,y),G^*_{\gamma',\gamma'}(x,y) \big\} + o(1),$$
since $\gamma' \le \gamma \le 2/5$, and therefore $\mu/(1-\mu) \le 2/3 < 3/4$. Hence, by Observation~\ref{obs:generalF:ESzbound} and Lemma~\ref{lem:nearish:calc}, and since $\ell - m > k/4$, we obtain
$$\frac{\log |U|}{k} \le \frac{h^*(\gamma')}{1-\gamma'} - \frac{1}{50} + o(1).$$
This implies that
$$n \le (1+\xi)^{-m} \bigg( \prod_{i = 0}^{m-1} \frac{k + \ell - i}{\ell - i} \bigg) \cdot \exp\bigg( \big( k + \ell - m \big) \cdot h^*\bigg( \frac{\ell - m}{k + \ell - m} \bigg) - \frac{k}{50}  + o(k) \bigg),$$
and hence, by Fact~\ref{fact:entropy2}, we obtain
$$R(k,\ell) \le e^{-k/50 + o(k)} {k + \ell - m \choose \ell - m} \prod_{i = 0}^{m-1} \frac{k + \ell - i}{\ell - i} = e^{-k/50 + o(k)} {k + \ell \choose \ell},$$
as required. 
\end{proof} 
 
We are finally ready to deduce Theorem~\ref{thm:off:diagonal:near}. To do so, we simply take blue steps until either $\gamma' = 2/5$, in which case we can apply Theorem~\ref{thm:off:diagonal:nearer}, or the red density is at least $1 - \gamma'$, in which case we can apply Theorem~\ref{thm:generalbound} and Lemma~\ref{lem:even:nearer:calc}.

\begin{proof}[Proof of Theorem~\ref{thm:off:diagonal:near}]
Set $n = R(k,\ell) - 1$, and suppose that $\chi$ is a red-blue colouring of $E(K_n)$ containing no red $K_k$ or blue $K_\ell$. If $\ell \le 2k/3$, then the claimed bound on $R(k,\ell)$ follows from Theorem~\ref{thm:off:diagonal:nearer}, so we may assume that $2k/3 \le \ell \le 9k/10$. We may therefore choose a sequence of distinct vertices $x_1,\ldots,x_m \in V(K_n)$ such that
$$x_i \in N_B(x_1) \cap \cdots \cap N_B(x_{i-1})$$
for each $i \in [m]$, and
\begin{equation}\label{eq:after:ESz:steps:the:third}
|N_B(x_1) \cap \cdots \cap N_B(x_m)| \, \ge \, n \cdot \prod_{i = 0}^{m-1} \frac{\ell - i}{k + \ell - i} - m
\end{equation}
and such that, setting 
$$\gamma' = \frac{\ell - m}{k + \ell - m} \qquad \text{and} \qquad U = N_B(x_1) \cap \cdots \cap N_B(x_m),$$
one of the following holds: \smallskip
\begin{itemize}
\item[$(a)$] $\ell - m > 2k/3$ and the colouring $\chi$ restricted to the set $U$ has blue density at most $\gamma'$;\smallskip
\item[$(b)$] $\ell - m = \lfloor 2k/3 \rfloor$.\smallskip
\end{itemize}

If~$(b)$ holds, then we apply Theorem~\ref{thm:off:diagonal:nearer}, and obtain
$$n \le e^{-(\ell-m)/50 + o(k)} {k + \ell - m \choose \ell - m} \prod_{i = 0}^{m-1} \frac{k + \ell - i}{\ell - i},$$
It follows that
$$R(k,\ell) \le e^{-k/75 + o(k)} {k + \ell \choose \ell},$$
as required.

If~$(a)$ holds, on the other hand, then we apply Theorem~\ref{thm:generalbound} to the colouring induced by the set $U$, with $\mu = 2/5$ and $\nu = 1 - \gamma'$. It follows that 
$$\frac{\log |U|}{k} \le \max_{\substack{\\[-0.05ex] 0 \,\le\, x \,\le\, 1 \\[+0.3ex] 0 \,\le\, y \,\le\, 5/7}} \hspace{-0.3cm} \min \big\{ F^*_{\nu,k,\ell-m} (x,y),G^*_{2/5,\gamma'}(x,y) \big\} + o(1),$$
since $\mu/(1-\mu) = 2/3 < 5/7$. Hence, by Observation~\ref{obs:generalF:ESzbound} and Lemma~\ref{lem:even:nearer:calc}, and recalling that $\ell - m > 2k/3$, we obtain
$$\frac{\log |U|}{k} \le \frac{h^*(\gamma')}{1-\gamma'} - \frac{1}{80} + o(1).$$
This implies that
$$n \le \bigg( \prod_{i = 0}^{m-1} \frac{k + \ell - i}{\ell - i} \bigg) \cdot \exp\bigg( \big( k + \ell - m \big) \cdot h^*\bigg( \frac{\ell - m}{k + \ell - m} \bigg) - \frac{k}{80}  + o(k) \bigg),$$
and thus, by Fact~\ref{fact:entropy2}, we obtain
$$R(k,\ell) \le e^{-k/80 + o(k)} {k + \ell - m \choose \ell - m} \prod_{i = 0}^{m-1} \frac{k + \ell - i}{\ell - i} = e^{-k/80 + o(k)} {k + \ell \choose \ell},$$
as required. 
\end{proof}

 \section{The proof of Theorem~\ref{thm:off:diagonal}}\label{finalproof:sec}

 In this section we will prove Theorem~\ref{thm:off:diagonal} in the following explicit form. Our proof of this bound also provides a second, somewhat different proof of Theorem~\ref{thm:diagonal}. 
 
\begin{theorem}\label{thm:off:diagonal:explicit}
$$R(k,\ell) \le e^{-\ell/400 + o(k)} {k + \ell \choose \ell}$$
for all $k,\ell \in \N$ with $\ell \le k$. 
\end{theorem}

We will deduce Theorem~\ref{thm:off:diagonal:explicit} from Theorem~\ref{thm:off:diagonal:near} as follows. First we take \emph{red} steps 
until either the density of blue edges is sufficiently high, 
or we hit the diagonal. We then use a variant of Lemma~\ref{lem:big:blue:step} to `jump' away from the diagonal (by taking roughly $k/10$ blue steps all at once) and into the range covered by Theorem~\ref{thm:off:diagonal:near}. The key step is the following lemma, which deals with colourings with suitable blue density. 

\begin{lemma}\label{lem:off:diagonal:full:range:good:density}
Let\/ $k,\ell \in \N$ be sufficiently large integers with $9k/10 \le \ell \le k$. If
$$n \ge e^{-k/400} {k + \ell \choose \ell},$$
then every red-blue colouring of\/ $E(K_n)$ in which the density of blue edges is at least\/ $\gamma$ contains~either a red $K_k$ or a blue $K_\ell$. 
\end{lemma}

We will `jump' away from the diagonal using the following variant of Lemma~\ref{lem:big:blue:step}. 

\begin{lemma}\label{lem:big:blue:step:repeat}
Let\/ $k,\ell \in \N$, with\/ $9k/10 \le \ell \le k$, and let\/ $b \le k/10$ and\/ $a \ge 9\ell/10$. Let $\chi$~be a red-blue colouring of\/ $E(K_n)$, and suppose that there are~at least\/ $R(k,a)$ vertices of $K_n$ with at least\/ $\big( \gamma - o(1) \big) n$ blue neighbours, where $\gamma = \frac{\ell}{k+\ell}$. 

Then\/ $\chi$ contains either a red $K_k$, or a blue book $(S,T)$ with $|S| = b$ and 
$$|T| \ge \exp\bigg( - \frac{b^2}{k} + o(k) \bigg) \cdot \gamma^b \cdot n.$$
\end{lemma}

In order to prove Lemma~\ref{lem:big:blue:step:repeat}, we'll need a slightly stronger version of Fact~\ref{binomial:fact1}.\footnote{Note that if $\ell \ge 9k/10$, then $\gamma \ge 9/19 > 7/15$, and if also $b \le k/10$ and $a \ge 9\ell/10$, then $b \le a/7$.}

\begin{fact}\label{binomial:fact1:stronger}
Let $a,b \in \N$ with $b \le a/7$, and let $\sigma \ge 7/15$. Then
$${\sigma a \choose b} \ge \exp\bigg( - \frac{3b^2}{4 a} \bigg) \sigma^b {a \choose b}.$$ 
\end{fact}

Fact~\ref{binomial:fact1:stronger} is proved in~\cite[Appendix~D]{CGMS}. 
To deduce Lemma~\ref{lem:big:blue:step:repeat}, we now simply repeat the proof of Lemma~\ref{lem:big:blue:step}, using Fact~\ref{binomial:fact1:stronger} in place of Fact~\ref{binomial:fact1}. 

\begin{proof}[Proof of Lemma~\ref{lem:big:blue:step:repeat}]
Let $W \subset V(K_n)$ be a set of $R(k,a)$ vertices with blue degree at least $\big( \gamma - o(1) \big)  n$, and note that $W$ contains either a red $K_k$ or a blue~$K_a$. In the former case we are done, so assume that $U \subset W$ is the vertex set of a blue $K_a$. Since each vertex of $U$ has at least $\big( \gamma - o(1) \big) n$ blue neighbours, and recalling that $n \ge R(k,a) \gg a = |U|$, it follows that the density of blue edges between $U$ and $U^c$ is at least $\gamma - o(1)$. 

Let $S \subset U$ be a uniformly-chosen random subset of size $b$, and let $Z = |N_B(S) \cap U^c|$ be the number of common blue neighbours of $S$ in $U^c$. Using convexity and Fact~\ref{binomial:fact1:stronger}, exactly as in the proof~\eqref{eq:bigblue:ExZ:convexity}, we have
$$\Ex[Z] \,\ge\,  2^{o(k)} {\gamma a \choose b} {a \choose b}^{-1} n \, \ge \, \exp\bigg( - \frac{b^2}{k} + o(k) \bigg) \cdot \gamma^b \cdot n,$$
where in the final inequality we used the bound $a \ge 3k/4$, which holds because $\ell \ge 9k/10$. Hence there exists a blue clique $S \subset U$ of size $b$ with at least this many common blue neighbours in $U^c$, as required. 
\end{proof}

We'll also use the following simple fact, which is proved in~\cite[Appendix~D]{CGMS}.

\begin{fact}\label{final:fact}
If $k,\ell,b \in \N$ and $b \le \ell \le k$, then 
$${k + \ell - b \choose \ell - b} \le \exp\bigg( - \frac{b^2}{4k} + o(k) \bigg) \bigg( \frac{\ell}{k + \ell} \bigg)^b {k + \ell \choose \ell}.$$ 
\end{fact}

We are now ready to deduce Lemma~\ref{lem:off:diagonal:full:range:good:density} from Theorem~\ref{thm:off:diagonal:near}. 

\begin{proof}[Proof of Lemma~\ref{lem:off:diagonal:full:range:good:density}]
Let $n \ge e^{-k/400} {k + \ell \choose \ell}$, let $\chi$ be a red-blue colouring of $E(K_n)$ in which the density of blue edges is at least $\gamma = \frac{\ell}{k+\ell}$, and set $a = 9\ell/10$. We claim that $\chi$ contains at least $R(k,a)$ vertices with at least $(\gamma - 2\eps) n$ blue neighbours (where, as usual, $\eps = k^{-1/4}$). To see this, note that 
$$R(k,a) \le {k + a \choose a} \le 2^{-(\ell - a)} {k + \ell \choose \ell} \le 2^{-\ell/20} {k + \ell \choose \ell} \le 2^{-k/40} \cdot n,$$
by the Erd\H{o}s--Szekeres bound~\eqref{eq:ESz:bound} and Fact~\ref{final:fact}. 
Therefore, if at most $R(k,a)$ vertices have at least $(\gamma - 2\eps) n$ blue neighbours, then there are at most
$$\frac{1}{2} \Big( R(k,a) \cdot n + \big( n - R(k,a) \big) \big( \gamma - 2\eps \big) n \Big) \le \big( \gamma - \eps \big) {n \choose 2}$$
blue edges in $\chi$, contradicting our assumption that the density of blue edges is at least $\gamma$. 

By Lemma~\ref{lem:big:blue:step:repeat}, it follows that there either exists a red copy of $K_k$ (in which case we are done), or a blue book $(S,T)$ with $|S| = b = \ell - 9k/10$ and 
$$|T| \ge \exp\bigg( - \frac{b^2}{k} + o(k) \bigg) \cdot \gamma^b \cdot e^{-k/400} {k + \ell \choose \ell}.$$  
By Fact~\ref{final:fact}, it follows that
$$|T| \ge \exp\bigg( - \frac{b^2}{k} - \frac{k}{400} + \frac{b^2}{4k} + o(k) \bigg) {k + \ell - b \choose \ell - b}.$$ 
Moreover, since $b \le k/10$, we have
$$\frac{b^2}{k} + \frac{k}{400} - \frac{b^2}{4k} \le \frac{k}{100},$$
and therefore, recalling that $\ell - b = 9k/10$,
$$|T| \ge \exp\bigg( - \frac{k}{100} + o(k) \bigg) {k + \ell - b \choose \ell - b}  = \exp\bigg( - \frac{\ell - b}{90} + o(k) \bigg) {k + \ell - b \choose \ell - b}.$$ 
By Theorem~\ref{thm:off:diagonal:near}, it follows that the colouring restricted to $T$ contains either a red $K_k$ or a blue $K_{\ell - b}$, and hence $\chi$ contains either a red $K_k$ or a blue $K_\ell$, as claimed.
 \end{proof}

Finally, let us deduce Theorem~\ref{thm:off:diagonal:explicit} from Lemma~\ref{lem:off:diagonal:full:range:good:density}.

\begin{proof}[Proof of Theorem~\ref{thm:off:diagonal:explicit}]
If $\ell \le 9k/10$ then the claimed bound follows from Theorem~\ref{thm:off:diagonal:near}, so we may assume that $9k/10 \le \ell \le k$. Set $\gamma = \ell/(k+\ell)$, let
$$n \ge e^{-\ell/400 + 1} {k + \ell \choose \ell},$$
and let $\chi$ be a red-blue colouring of $E(K_n)$. If the density of blue edges in $\chi$ is at most $\gamma$, then choose a sequence of distinct vertices $x_1,\ldots,x_m \in V(K_n)$ such that
$$x_i \in N_R(x_1) \cap \cdots \cap N_R(x_{i-1})$$
for each $i \in [m]$, and
\begin{equation}\label{eq:final:proof:steps}
|N_R(x_1) \cap \cdots \cap N_R(x_m)| \ge n \cdot \prod_{i = 0}^{m-1} \frac{k - i}{k - i + \ell} - m
\end{equation}
and such that either the colouring $\chi$ restricted to the set $U = N_R(x_1) \cap \cdots \cap N_R(x_m)$ has blue density at least 
$$\gamma' = \frac{\ell}{k - m + \ell},$$
or $m = k - \ell$. In either case, we apply Lemma~\ref{lem:off:diagonal:full:range:good:density} to the colouring restricted to $U$, noting that if $m = k - \ell$ then $\gamma' = 1/2$, and one of the two colours must have density at least $1/2$. Since
$$|U| \ge n \cdot \prod_{i = 0}^{m-1} \frac{k - i}{k + \ell - i} - m \ge e^{-\ell/400} {k - m + \ell \choose \ell},$$
it follows that $\chi$ contains either a red $K_k$ or a blue $K_\ell$, and hence
$$R(k,\ell) \le e^{-\ell/400 + 1} {k + \ell \choose \ell}$$
as claimed. 
\end{proof}
   
\section*{Acknowledgements}

The research described in this paper was carried out during several visits of JS to IMPA. The authors are grateful to IMPA for their support over many years, and for providing us with a wonderful working environment.
 
 \medskip

\appendix
  
\section{The proof of Lemma~\ref{lem:final:calc}}\label{app:final:calc}
 
Recall that
$$g(x,y) = \log_2 \left( \frac{5}{2} \right) + x \cdot \log_2 \left(\frac{5}{3} \right) + y \cdot \log_2 \left(\frac{2(x+y)}{5y}\right)$$
that  
$$f(x,y) = f_1(x,y) = x + y + (2 - x) \cdot h\left(\frac{1}{2-x}\right)$$ 
when $x < 3/4$, and that, for some sufficiently small constant $\eta > 0$, 
$$f(x,y) = f_2(x,y) = x + y + (2 - x) \cdot h\left(\frac{1}{2-x}\right) - \frac{\log_2 e}{40} \bigg( \frac{1-x}{2-x} \bigg) + \eta,$$
when $x \ge 3/4$. We begin by noting some useful monotonicity properties of $f$ and $g$.
 
\begin{obs}\label{obs:Gff:monotone}
For each fixed $y \in (0,1)$, the following hold:
\begin{itemize}
\item[$(a)$] $g$ is increasing in $x$ on $(0,1)$.\smallskip
\item[$(b)$] $f$ is decreasing in $x$ for $x \ge 1/2$. 
\end{itemize}
\end{obs}

\begin{proof}
Part $(a)$ is immediate. For part $(b)$, it can be checked that 
$$\frac{\partial}{\partial x} f_1(x,y) = \log_2 \bigg( \frac{2-2x}{2-x} \bigg) \qquad \text{and} \qquad \frac{\partial}{\partial x} f_2(x,y) = \frac{\partial}{\partial x} f_1(x,y) + \frac{\log_2(e)}{40(2-x)^2},$$
and therefore that $\frac{\partial}{\partial x} f_1(x,y) \le 0$ for all $x \in (0,1)$ and $\frac{\partial}{\partial x} f_2(x,y) \le 0$ if $x \ge 1/2$. It follows that $f$ is decreasing in $x$ for $x \ge 1/2$ since $f_2(3/4,y) \le f_1(3/4,y)$ for all $0 \le y \le 1$.
\end{proof}

To prove the lemma we will show that the desired inequality holds on the line 
\begin{equation}\label{def:x:of:y}
x(y) = \frac{3y}{5} + 0.5454,
\end{equation}
and then use Observation~\ref{obs:Gff:monotone} to deduce the same bound for the entire range. Note that $1/2 \le x(y) \le 1$ for every $0 \le y \le 3/4$. We proceed with the following three simple claims, which can be easily be checked either by computer or by hand.

\begin{figure}[h]
  \centering
    \begin{subfigure}[b]{0.47\textwidth}
    \includegraphics[width=\textwidth]{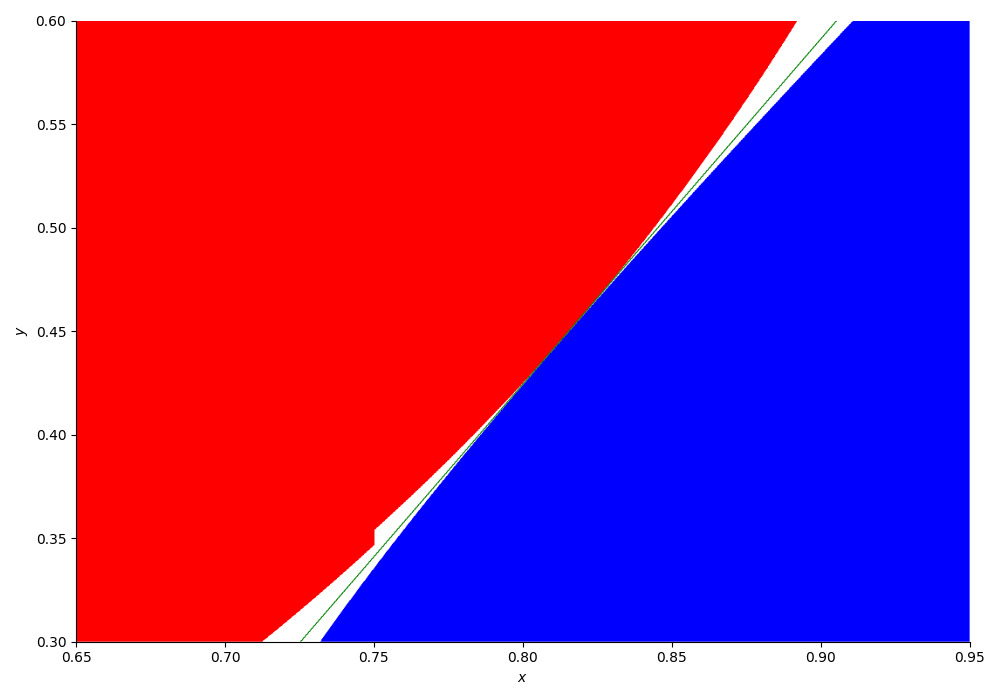}
  \end{subfigure}
  \hspace{0.5cm}
  \begin{subfigure}[b]{0.47\textwidth}
     \includegraphics[width=\textwidth]{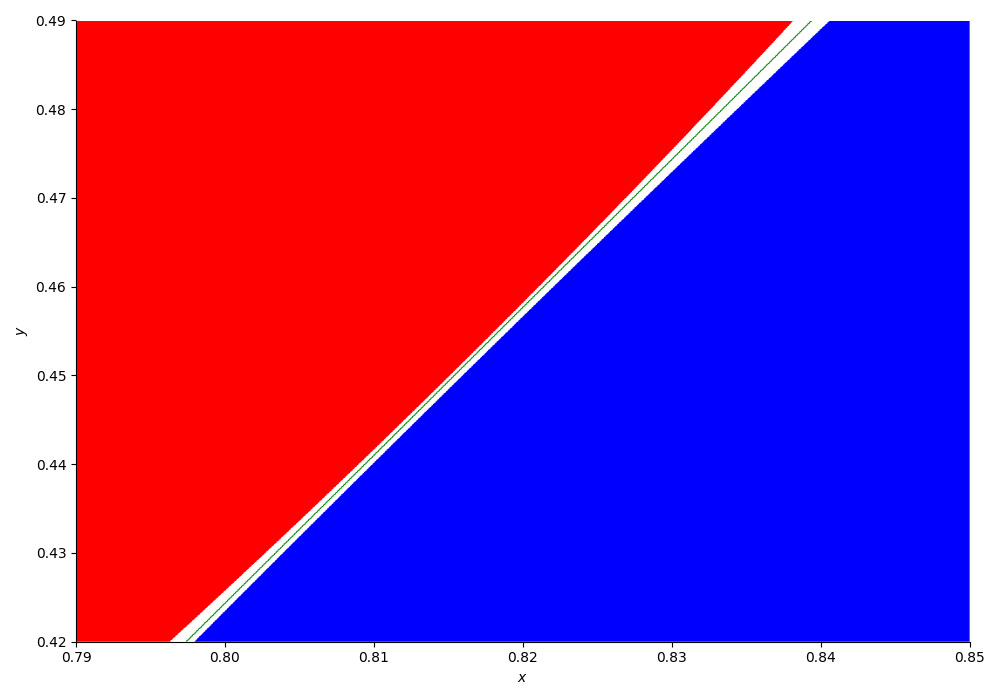}
  \end{subfigure}
  \caption{The sets $f(x,y) \ge 2$ and $g(x,y) \ge 2$, and the line $x = 3y/5 + 0.5454$.}   \label{fig:diagonal:app}
\end{figure}

\begin{claim}\label{claim:AppA:G}
$g(x(y),y) < 2 - \delta$ for every $0 \le y \le 3/4$. 
\end{claim}
 
\begin{clmproof}{claim:AppA:G}
One can check that
$$g(x(y),y) = \log_2 \left( \frac{5}{2} \right) + \bigg( \frac{3y}{5} + 0.5454 \bigg) \log_2 \left(\frac{5}{3} \right) + y \cdot \log_2 \left( \frac{16y + 5.454}{25y} \right).$$  
is maximised with $y \approx 0.434$, and that $g(x(y),y) < 1.9993$ for all $0 \le y \le 3/4$.  
 \end{clmproof}

To bound $f(x,y)$ on the line, we divide into two cases, depending on whether or not $x \ge 3/4$. Note that $x(0.341) = 3/4$.

\begin{claim}\label{claim:AppA:f1}
$f_1\big(x(y),y\big) < 2 - \delta$ for every $0 \le y \le 0.341$. 
\end{claim}
 
\begin{clmproof}{claim:AppA:f1}
To bound $f_1(x(y),y)$ it is more convenient to substitute for $y$, so note that $y = 5x(y)/3 - 0.909$, and therefore
\begin{equation}\label{eq:f1:line}
f_1\big(x(y),y\big) = \frac{8x}{3} - 0.909 + (2 - x) \cdot h\left(\frac{1}{2-x}\right).
\end{equation}
One can check\footnote{Indeed, we have $\frac{d}{dx} f_1(x,y(x)) = 8/3 + \log_2\big( \frac{1-x}{2-x} \big) > 0$ for all $0 \le x < (2^{8/3} - 2)/(2^{8/3} - 1) \approx 0.813$.} that this 
is increasing in $x$ on the interval $0 \le x \le 3/4$, and therefore 
$$f_1\big( x(y),x \big) \le f_1\big(3/4,0.341\big) < 1.994$$
for every $0 \le y \le 0.341$, as claimed. 
\end{clmproof}

\begin{claim}\label{claim:AppA:f2}
$f_2\big( x(y),y \big) < 2 - \delta$ for every $0.341 \le y \le 3/4$. 
\end{claim}
 
\begin{clmproof}{claim:AppA:f2}
Again substituting for $y$, we have 
$$f_2\big( x(y),y \big) = \frac{8x}{3} - 0.909 + (2-x) \cdot h\left(\frac{1}{2-x}\right) - \frac{\log_2 e}{40} \bigg( \frac{1-x}{2-x} \bigg) + \eta.$$
One can check\footnote{Indeed, note that $\frac{d}{dx} f_2(x,y(x)) = 8/3 + \log_2\big( \frac{1-x}{2-x} \big) + \frac{\log_2(e)}{40(2-x)^2}$.} that this function is maximised with $x \approx 0.817$, and that $f_2( x(y),y ) < 1.9993$ for every $3/4 \le x \le 1$, and hence for every $0.341 \le y \le 3/4$, as claimed.
\end{clmproof}

Combining the three claims with Observation~\ref{obs:Gff:monotone} gives Lemma~\ref{lem:final:calc}. Indeed, if we have 
$$g(x,y) \le g(x(y),y) < 2 - \delta$$ 
when $x \le x(y)$, by Observation~\ref{obs:Gff:monotone} and Claim~\ref{claim:AppA:G}, and 
$$f(x,y) \le f(x(y),y) < 2 - \delta$$ 
when $x \ge x(y)$, by Observation~\ref{obs:Gff:monotone} and Claims~\ref{claim:AppA:f1} and~\ref{claim:AppA:f2}, since $x(y) \ge 1/2$ for all $y \ge 0$.

\section{The proof of Lemma~\ref{lem:nearish:calc}}\label{app:off:calc:gamma}

We begin by noting some useful monotonicity properties of $G_{\mu,\gamma}^*(x,y)$ and $f_{\nu,\theta}^*(x,y)$.

\begin{obs}\label{obs:Gfstar:monotone}
Fix $0 < \mu \le 1/2 \le \nu \le 1$ and $0 < \gamma \le 1/2$ 
with\/ $\nu \ge (1-\gamma)/(1+\gamma)$, and set $\theta = \gamma / (1 - \gamma)$. Then for each fixed $y \in (0,1)$, the following hold:
\begin{itemize}
\item[$(a)$] $G^*_{\mu,\gamma}(x,y)$ is increasing in $x$ on $(0,1)$.\smallskip
\item[$(b)$] $f_{\nu,\theta}^*(x,y)$ is decreasing in $x$ for $x \ge 1/2$.
\end{itemize}
\end{obs}

\begin{proof}
Part $(a)$ is immediate from the definition~\eqref{def:Gstar}. For part $(b)$, it can be checked that 
$$\frac{\partial}{\partial x} f_{\nu,\theta}^*(x,y) = \log \left(\frac{1}{\nu}\right) - \log \bigg( \frac{1 + \theta - x}{1-x} \bigg).$$
Since $\nu \ge (1-\gamma)/(1+\gamma) = (1+2\theta)^{-1}$, we have $\nu(1 + \theta - x) \ge 1 - x$ for all $x \ge 1/2$, and therefore $\frac{\partial}{\partial x} f_{\nu,\theta}^*(x,y) \le 0$ for all $x \ge 1/2$, as claimed. 
\end{proof} 

Next, recall that in Lemma~\ref{lem:nearish:calc} we have $\mu = \gamma = \ell/(k+\ell)$ and $\theta = \ell/k$ for some $k,\ell \in \N$ with $k/4 \le \ell \le 2k/3$, and $\nu = 1 - \gamma - \eta$, where $\eta = \gamma/40$. In particular, we have 
$$\frac{1}{5} \le \gamma \le \frac{2}{5}, \qquad \nu \ge \frac{1 - \gamma}{1 + \gamma} \qquad \text{and} \qquad \theta = \frac{\gamma}{1 - \gamma}.$$ 
Note also that, setting $\mu = \gamma$ in~\eqref{def:Gstar}, we have
$$G^*_{\gamma,\gamma}(x,y) - \frac{h^*(\gamma)}{1-\gamma} = y \cdot \log \left( \frac{\gamma(x+y)}{y} \right) - (1 - x) \cdot \log \left(\frac{1}{1-\gamma} \right),$$
and recall from~\eqref{def:fnuthetastar} that 
$$f_{\nu,\theta}^*(x,y) = (x + y) \cdot \log \left(\frac{1}{\nu}\right) + \big( 1 + \theta - x \big) \cdot h^*\bigg( \frac{\theta}{1 + \theta - x} \bigg),$$
where we now work with the usual (non-binary) entropy function $h^*$. 

Again we check the desired inequality on a line and then use the monotonicity of the bounds to deduce the lemma. Here we use
\begin{equation}\label{def:xstar:of:y}
x(y) = \frac{4(y+1)}{7},
\end{equation}
see Figure~\ref{fig:twofifths:app}. Note that $4/7 \le x(y) \le 1$ whenever $0 \le y \le 3/4$. 

\begin{claim}\label{claim:AppB:G}
$G^*_{\gamma,\gamma}\big(x(y),y \big) < \displaystyle\frac{h^*(\gamma)}{1-\gamma} - \frac{1}{50}$ for every $0 \le y \le 3/4$. 
\end{claim}
 
\begin{clmproof}{claim:AppB:G}
Observe that
$$G^*_{\gamma,\gamma}\big( x(y),y \big) - \frac{h^*(\gamma)}{1-\gamma} = y \cdot \log \left( \frac{\gamma(11y+4)}{7y} \right) + \frac{4y-3}{7} \log \left(\frac{1}{1-\gamma} \right).$$
It can be checked that in the range $1/5 \le \gamma \le 2/5$ and $0 \le y \le 3/4$, this function is maximised with $\gamma = 2/5$ and $y \approx 0.397$, and 
$$G^*_{\gamma,\gamma}\big( x(y),y \big) \le \frac{h^*(\gamma)}{1-\gamma} - 0.029,$$
as claimed. 
\end{clmproof}

\begin{figure}[t]
\centering
\vskip-1cm
 \includegraphics[width=0.7\textwidth]{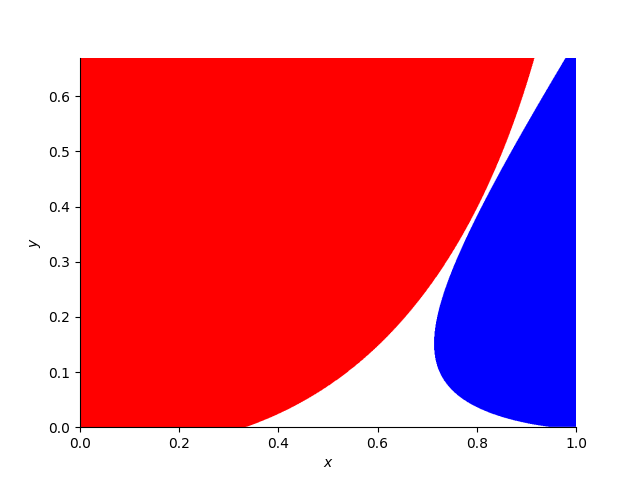}
    \caption{Lemma~\ref{lem:nearish:calc} when $\ell = 2k/3$.} 
\label{fig:twofifths:app}
\end{figure}

\begin{claim}\label{claim:AppB:fstar}
$f_{\nu,\theta}^*\big(x(y),y\big) < \displaystyle\frac{h^*(\gamma)}{1-\gamma} - \frac{1}{50}$ for every $0 \le y \le 3/4$.
\end{claim}
 
\begin{clmproof}{claim:AppB:fstar}
Recall that $\theta = \gamma/(1-\gamma)$ and that $\nu = 1 - 41\gamma/40$. It is again more convenient to substitute for $y$, so note that $y = 7x(y)/4 - 1$, and therefore 
$$f_{\nu,\theta}^*\big(x(y),y\big) = \bigg( \frac{11x}{4} - 1 \bigg) \log \left(\frac{1}{\nu}\right) + \big( 1 + \theta - x \big) \cdot h^*\bigg( \frac{\theta}{1 + \theta - x} \bigg).$$
It can be checked that, in the range $1/5 \le \gamma \le 2/5$ and $0 \le y \le 3/4$, the function $f_{\nu,\theta}^*\big(x(y),y\big) - (1-\gamma)^{-1} h^*(\gamma)$ is maximised with $\gamma = 2/5$ and $x \approx 0.796$, and that
$$f_{\nu,\theta}^*\big(x(y),y\big) \le \frac{h^*(\gamma)}{1-\gamma} - 0.02,$$
as claimed.
\end{clmproof}

Combining the two claims with Observation~\ref{obs:Gfstar:monotone} gives Lemma~\ref{lem:nearish:calc}. Indeed, we have 
$$G^*_{\gamma,\gamma}(x,y) \le G^*_{\gamma,\gamma}\big( x(y),y \big) < \frac{h^*(\gamma)}{1-\gamma} - \frac{1}{50}$$ 
whenever $x \le x(y)$, by Observation~\ref{obs:Gfstar:monotone} and Claim~\ref{claim:AppB:G}, and 
$$f_{\nu,\theta}^*(x,y) \le f_{\nu,\theta}^*\big( x(y),y \big) < \frac{h^*(\gamma)}{1-\gamma} - \frac{1}{50}$$ 
whenever $x \ge x(y)$, by Observation~\ref{obs:Gfstar:monotone} and Claim~\ref{claim:AppB:fstar}, and since $x(y) \ge 1/2$ for all $y \ge 0$.

\section{The proof of Lemma~\ref{lem:even:nearer:calc}}\label{app:off:calc:twofifths}

Recall that in Lemma~\ref{lem:even:nearer:calc} we have $\gamma = \ell/(k+\ell)$ and $\theta = \ell/k$ for some $k,\ell \in \N$ with $2k/3 \le \ell \le 9k/10$, and that $\mu = 2/5$ and $\nu = 1 - \gamma$. In particular, we have 
$$\frac{2}{5} \le \gamma \le \frac{9}{19}, \qquad \nu \ge \frac{1 - \gamma}{1 + \gamma} \qquad \text{and} \qquad \theta = \frac{\gamma}{1 - \gamma}.$$ 
Note also that, setting $\mu = 2/5$ in~\eqref{def:Gstar} gives
$$G_{2/5,\gamma}^*(x,y) = \frac{\gamma}{1-\gamma} \log \left( \frac{5}{2} \right) + x \cdot \log \left(\frac{5}{3} \right) + y \cdot \log \left(\frac{2(x+y)}{5y}\right),$$
and that, setting $\nu = 1 - \gamma$ in~\eqref{def:fnuthetastar} gives 
$$f_{\nu,\theta}^*(x,y) = (x + y) \log \left(\frac{1}{1 - \gamma}\right) + \big( 1 + \theta - x \big) \cdot h^*\bigg( \frac{\theta}{1 + \theta - x} \bigg).$$
We use the same strategy as in Appendices~\ref{app:final:calc} and~\ref{app:off:calc:gamma}, this time using the line 
\begin{equation}\label{def:xstar:of:y}
x(y) = \frac{3y}{5} + 0.552,
\end{equation}
see Figure~\ref{fig:ninetenths:app}. Note that $1/2 \le x(y) \le 1$ for every $0 \le y \le 5/7$. 



\begin{claim}\label{claim:AppC:G}
$G_{2/5,\gamma}^*\big(x(y),y \big) < \displaystyle\frac{h^*(\gamma)}{1-\gamma} - \frac{1}{80}$ for every $0 \le y \le 5/7$. 
\end{claim}
 
\begin{clmproof}{claim:AppC:G}
Observe that
\begin{align*}
& G_{2/5,\gamma}^*\big( x(y),y \big) - \frac{h^*(\gamma)}{1-\gamma} = \bigg( \frac{3y}{5} + 0.552 \bigg) \log \left(\frac{5}{3} \right) + y \log \left( \frac{16y + 5.52}{25y} \right)\\
& \hspace{8cm} + \frac{\gamma}{1-\gamma} \log \left( \frac{5\gamma}{2} \right) - \log \left(\frac{1}{1 - \gamma}\right).
\end{align*}
It can be checked that in the range $2/5 \le \gamma \le 9/19$ and $0 \le y \le 5/7$, this function is maximised with $\gamma = 9/19$ and $y \approx 0.438$, and 
$$G_{2/5,\gamma}^*\big( x(y),y \big) \le \frac{h^*(\gamma)}{1-\gamma} - 0.014,$$
as claimed. 
\end{clmproof}

\begin{figure}[t]
  \centering
    \begin{subfigure}[b]{0.47\textwidth}
    \includegraphics[width=\textwidth]{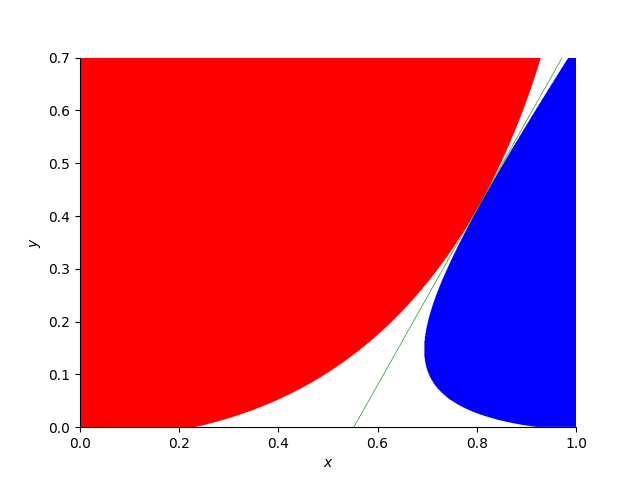}
  \end{subfigure}
  \hspace{0.5cm}
  \begin{subfigure}[b]{0.47\textwidth}
     \includegraphics[width=\textwidth]{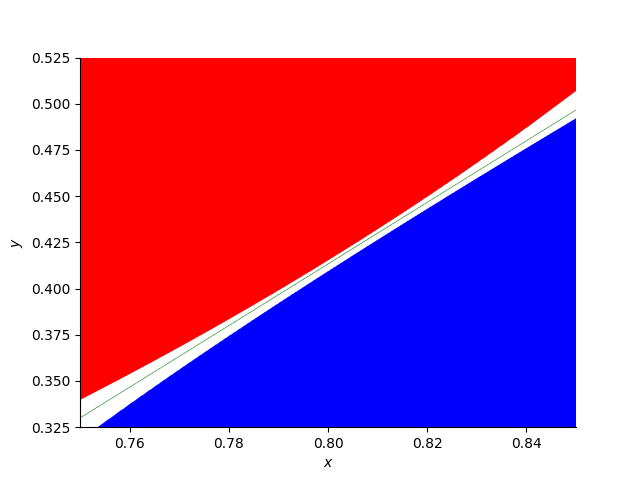}
  \end{subfigure}
   \caption{Lemma~\ref{lem:even:nearer:calc} when $\ell = 9k/10$, with the line $x(y) = 3y/5 + 0.552$.}
\label{fig:ninetenths:app}
\end{figure}

\begin{claim}\label{claim:AppC:fstar}
$f_{\nu,\theta}^*\big(x(y),y\big) < \displaystyle\frac{h^*(\gamma)}{1-\gamma} - \frac{1}{80}$ for every $0 \le y \le 5/7$.
\end{claim}
 
\begin{clmproof}{claim:AppC:fstar}
It is again more convenient to substitute for $y$, so note that $y = 5x(y)/3 - 0.92$, and therefore 
$$f_{\nu,\theta}^*\big(x(y),y\big) = \bigg( \frac{8x}{3} - 0.92 \bigg) \log \left(\frac{1}{1-\gamma}\right) + \big( 1 + \theta - x \big) \cdot h^*\bigg( \frac{\theta}{1 + \theta - x} \bigg).$$
It can be checked that in the range $2/5 \le \gamma \le 9/19$ and $0 \le y \le 5/7$, the function $f_{\nu,\theta}^*\big(x(y),y\big) - (1-\gamma)^{-1} h^*(\gamma)$ is maximised with $\gamma = 9/19$ and $x \approx 0.802$, and that
$$f_{\nu,\theta}^*\big(x(y),y\big) \le \frac{h^*(\gamma)}{1-\gamma} - 0.014,$$
as claimed.
\end{clmproof}

Combining the two claims with Observation~\ref{obs:Gfstar:monotone} gives Lemma~\ref{lem:even:nearer:calc}. Indeed, we have 
$$G_{2/5,\gamma}^*(x,y) \le G_{2/5,\gamma}^*\big( x(y),y \big) < \frac{h^*(\gamma)}{1-\gamma} - \frac{1}{80}$$ 
whenever $x \le x(y)$, by Observation~\ref{obs:Gfstar:monotone} and Claim~\ref{claim:AppC:G}, and 
$$f_{\nu,\theta}^*(x,y) \le f_{\nu,\theta}^*\big( x(y),y \big) < \frac{h^*(\gamma)}{1-\gamma} - \frac{1}{80}$$ 
whenever $x \ge x(y)$, by Observation~\ref{obs:Gfstar:monotone} and Claim~\ref{claim:AppC:fstar}, and since $x(y) \ge 1/2$ for all $y \ge 0$.

\section{Simple inequalities}\label{app:simple}

In this section we will prove various simple inequalities that we used during the proof. We begin with Facts~\ref{binomial:fact1} and~\ref{binomial:fact1:stronger}. 

\begin{fact}
Let $m,b \in \N$ and $\sigma \in (0,1)$, with $b \le \sigma m / 2$. Then
\begin{equation}\label{eq:binomial:fact1}
\sigma^b {m \choose b} \exp\bigg( - \frac{b^2}{\sigma m} \bigg) \le {\sigma m \choose b} \le \sigma^b {m \choose b}.
\end{equation}
Moreover, if $b \le m/7$ and $\sigma \ge 7/15$, then
\begin{equation}\label{eq:binomial:fact1:stronger}
{\sigma m \choose b} \ge \exp\bigg( - \frac{3b^2}{4m} \bigg) \sigma^b {m \choose b}.
\end{equation} 
\end{fact}

\begin{proof}
Observe first that
$${\sigma m \choose b} {m \choose b}^{-1} = \; \prod_{i = 0}^{b-1} \, \frac{\sigma m - i}{m - i} \, = \, \sigma^b \, \prod_{i = 0}^{b-1} \bigg( 1 - \frac{(1 - \sigma)i}{\sigma(m - i)} \bigg).$$
The upper bound in~\eqref{eq:binomial:fact1} follows immediately. To see the lower bound in~\eqref{eq:binomial:fact1}, observe that 
$$1 - \frac{(1 - \sigma)i}{\sigma(m - i)} \ge 1 - \frac{i}{\sigma m} \ge \exp\bigg( - \frac{2i}{\sigma m} \bigg),$$
for all $0 \le i \le b - 1$, since $b \le \sigma m / 2$, so $m - i \ge (1 - \sigma)m$, and $1 - x \ge e^{-2x}$ for all $0 \le x \le 1/2$. Since $\sum_{i = 0}^{b-1} i \le b^2/2$, the claimed bound follows. To prove~\eqref{eq:binomial:fact1:stronger}, observe that
$$1 - \frac{(1 - \sigma)i}{\sigma(m - i)} \ge 1 - \frac{8i}{7(m-b)}  \ge 1 - \frac{4i}{3m} \ge \exp\bigg( - \frac{3i}{2m} \bigg)$$
for all $0 \le i \le b - 1$, where in the first inequality we used $\sigma \ge 7/15$, in the second we used $b \le m/7$, and in the third the bound $1 - 4x/3 \ge e^{-3x/2}$, which holds for all $0 \le x \le 1/7$. Since $\sum_{i = 0}^{b-1} i \le b^2/2$, this implies~\eqref{eq:binomial:fact1:stronger}.
\end{proof}

We will next prove the following inequality, which implies Facts~\ref{binomial:fact4} and~\ref{final:fact}.
 
\begin{fact}\label{fact:app:D}
If $k,\ell,t \in \N$, with $t \le k$, then
$${k + \ell - t  \choose \ell} \le \exp\bigg( - \frac{\gamma (t-1)^2}{2k} \bigg) \bigg( \frac{k}{k + \ell} \bigg)^t {k + \ell \choose \ell}.$$
where $\gamma = \frac{\ell}{k+\ell}$.  
\end{fact}

\begin{proof}
Observe that 
$${k+\ell - t \choose \ell} {k+\ell \choose \ell}^{-1} = \, \prod_{i = 0}^{t-1} \, \frac{k - i}{k + \ell - i} \, = \bigg( \frac{k}{k + \ell} \bigg)^t \; \prod_{i = 0}^{t-1} \,\bigg( 1 - \frac{i\ell}{k(k+\ell - i)} \bigg).$$
Since $1 - x \le e^{- x}$ for all $x \ge 0$, it follows that
$${k+\ell - t \choose \ell} {k+\ell \choose \ell}^{-1} \le \bigg( \frac{k}{k + \ell} \bigg)^t \exp\bigg( - \sum_{i = 0}^{t-1} \frac{i\ell}{k(k+\ell)} \bigg).$$
The claimed bound follows since $\sum_{i = 0}^{t-1} i \ge (t-1)^2/2$ and $\gamma = \frac{\ell}{k+\ell}$.
\end{proof}

Fact~\ref{binomial:fact4} follows immediately from Fact~\ref{fact:app:D}, since $t \le k$. To deduce Fact~\ref{final:fact}, note that
$${k + \ell - b \choose \ell - b} = {k + \ell - b \choose k} \le \exp\bigg( - \frac{(1-\gamma) (b-1)^2}{2\ell} \bigg) \bigg( \frac{\ell}{k + \ell} \bigg)^b {k + \ell \choose \ell}$$
for every $b \le \ell$, and that $1 - \gamma = k/(k+\ell) \ge 1/2$ if $\ell \le k$.\smallskip

Finally, we note that Fact~\ref{fact:binomal:gammas} follows from Stirling's formula. 

\begin{fact}
If $\ell \le k$, then
$${k + \ell \choose \ell} = 2^{o(k)} \gamma^{-\ell} (1 - \gamma)^{-k},$$
where $\gamma = \frac{\ell}{k+\ell}$.
\end{fact}

\begin{proof}
By Stirling's formula, we have 
$${k + \ell \choose \ell} = \frac{(k+\ell)!}{k! \cdot \ell!} = 
2^{o(k)} \frac{(k+\ell)^{k+\ell}}{k^k \cdot \ell^\ell} = 2^{o(k)} \gamma^{-\ell} (1 - \gamma)^{-k},$$
as claimed, since $\gamma = \frac{\ell}{k+\ell}$ and $1 - \gamma = \frac{k}{k+\ell}$. 
\end{proof}

\end{document}